\documentclass[10pt]{article} 
\usepackage[numbers]{natbib}
\usepackage[utf8]{inputenc}
\usepackage[top=30mm, left=30mm, right=30mm, bottom=30mm]{geometry}
\usepackage{amsmath,amssymb,amsthm}
\usepackage{dsfont}
\usepackage{color}
\usepackage[dvipsnames]{xcolor}
\usepackage{lineno}
\usepackage{yfonts}
\usepackage{bigints}
\usepackage{appendix}
\usepackage{esint} 
\usepackage{babel} 
\usepackage{cancel}
\usepackage{enumerate}
\usepackage{tabu} 
\usepackage{bm} 
\usepackage{bbm}
\usepackage{tkz-euclide}
\usepackage[hidelinks]{hyperref} 
\hypersetup{
	colorlinks,
	linkcolor={red!50!black},
	citecolor={blue!50!black},
	urlcolor={blue!80!black}
}

\newtheorem{theorem}{Theorem}[section]
\newtheorem{lem}[theorem]{Lemma}

\newtheorem{kor}[theorem]{Corollary}
\newtheorem{prop}[theorem]{Proposition}
\newtheorem{rem}[theorem]{Remark}
\theoremstyle{definition}
\newtheorem{dfn}[theorem]{Definition}
\newtheorem{claim}{Claim}

\numberwithin{equation}{section}

\newcommand{\esssup}{\mathrm{ess~sup}}
\DeclareMathOperator*{\supp}{supp}

\DeclareMathOperator*{\loc}{loc}

\newcommand*{\divv}{\mathrm{div}}
\newcommand*{\Ascr}{\mathcal A}
\newcommand*{\Bscr}{\mathcal B}

\newcommand*{\Fscr}{\mathcal F}

\newcommand*{\Mscr}{\mathcal M}

\newcommand*{\Pscr}{\mathcal P}

\newcommand*{\N}{\mathbb{N}}
\newcommand*{\E}{\mathbb{E}}
\newcommand*{\R}{\mathbb{R}}

\newcommand*{\PP}{\mathbb{P}}

\newcommand*{\vrho}{\varrho}

\newcommand*{\x}{\times}

\newcommand{\law}[1]{{\mathcal L}_{#1}}
\newcommand{\scalarproduct}[3][]{\langle #2, #3\rangle_{#1}}
\DeclareMathOperator{\M}{\mathrm{M}}
\newcommand{\rd}{{\R^d}}
\newcommand{\dx}{\mathrm{d}x}

\definecolor{orange}{rgb}{1.0, 0.55, 0.0} 

\begin{document}
\title{The Leibenson process}
	\author{Viorel Barbu\footnote{Al.I. Cuza University and Octav  Mayer Institute of Mathematics of Romanian Academy, Ia\c si, Romania. E-mail: vb41@uaic.ro}
    \and Sebastian Grube\footnote{Faculty of Mathematics, Bielefeld University, 33615 Bielefeld, Germany. E-mail: sgrube@math.uni-bielefeld.de}
    \and
	Marco Rehmeier\footnote{Institute of Mathematics, TU Berlin, 10623 Berlin, Germany. E-mail: rehmeier@tu-berlin.de}
	\and Michael Röckner \footnote{Faculty of Mathematics, Bielefeld University, 33615 Bielefeld, Germany. E-mail: roeckner@math.uni-bielefeld.de} \footnote{Academy of Mathematics and System Sciences, CAS, Beijing} 
    }
	\date{\today}
	\maketitle
	\begin{abstract}
	Consider the Leibenson equation
	\begin{equation*}
	\partial_t u = \Delta_p u^q,
	\end{equation*}
	where $\Delta_p f = \divv(|\nabla f|^{p-2}\nabla f)$ for $p>1$ and $q>0$,
	which is a simultaneous generalization of the porous media and the $p$-Laplace equation. In this paper we identify the Leibenson equation as a nonlinear Fokker--Planck equation and prove that it has a nonlinear Markov process in the sense of McKean as its probabilistic counterpart. More precisely, we obtain a probabilistic representation of its Barenblatt solutions as the one-dimensional marginal density curve of the unique solutions to the associated McKean--Vlasov SDE. The latter is of novel type, since its coefficients depend pointwise both on its solution's time marginal densities and also on their first and second order derivatives. Moreover, we show that these solutions constitute the aforementioned nonlinear Markov process, which we call the \textit{Leibenson process}. A further main result of this work is to 
  prove that despite
  the strong degeneracy of the diffusion and the irregularity of the drift coefficient (which is merely of bounded variation) of the McKean--Vlasov SDE these solutions are probabilistically strong, i.e., measurable functionals of the driving Brownian motion and the initial condition.
	\end{abstract}

\section{Introduction}
This paper should be seen as a substantial step forward in a general
program which, motivated by McKean's vision laid out in \cite{McKean1-classical}, we have started (see \cite{R.BR24-pLapl,BR18,BR18_2,NLFPK-DDSDE5,BR-IndianaJ,BR22,BR24-frac,BR-LN24,BRZ23,R.Romito24}), namely to develop a new theory whose
aim is to construct a probabilistic counterpart for a large class of nonlinear parabolic
PDEs in the same way as has been done in the last 60 years for linear parabolic PDEs
(see \cite{Markov1,Markov2,Markov3,Markov4-volume,Markov5,Markov6,Markov7,Markov8,Markov9,Markov10,Nualart-book,Markov11,Markov12,Markov13,SV-book}
and the references therein). The purpose is to then transform problems about the PDE to
the probabilistic counterpart and vice versa. This probabilistic counterpart is as in the
linear case a Markov process, but a nonlinear one in the sense of McKean and in the (so far)
most interesting cases they are given by the path laws of solutions to very singular McKean--Vlasov
SDEs. The mentioned transfer between analysis and probability was extremely successful in the
linear case with a huge literature (see above) until today. We are convinced that there are
very good chances that this also will happen in the nonlinear case.

In this paper, we realize the above program in the case where the nonlinear PDE is
the Leibenson equation (\cite{L45,L45-2})
\begin{equation}\label{eq:Leibenson}
	\partial_t u(t,x) = \Delta_p u(t,x)^q,\quad (t,x) \in (0,\infty)\times \R^d,
	\end{equation}
which is well known as the prototype of a degenerate doubly nonlinear parabolic partial differential equation, with its Barenblatt fundamental solutions (see \cite{B52_Barenblatt-original,Barenblatt-book}).

Here $q>0$ and $\Delta_p f = \divv(|\nabla f|^{p-2}\nabla f)$ denotes the $p$-Laplace operator, $p>1$.
The aforementioned probabilistic counterpart is a nonlinear Markov process in the sense of McKean (see \cite{McKean1-classical,R./Rckner_NL-Markov22}) consisting of the unique solutions to the corresponding McKean--Vlasov stochastic differential equation (McKean--Vlasov SDE)
\begin{align}\label{eq:Leibenson-MVSDE}\begin{cases}
	\mathrm{d}X(t) \quad\ \ =& q^{p-1}\nabla \big(|\nabla u(t,X(t))|^{p-2}   u(t,X(t))^{(q-1)(p-1)}\big)\mathrm{d}t\\
    &+ \sqrt{2q^{p-1}}\big( |\nabla u(t,X(t))|^{\frac{p-2}{2}}u(t,X(t))^{\frac{(q-1)(p-1)}{2}}\big) \mathrm{d}W(t),
	\\ \mathcal{L}_{X(t)}(\mathrm{d}x) =& u(t,x)\mathrm{d}x,
    \end{cases}
\end{align}
with one-dimensional time marginal densities $\frac{\mathrm{d}\mathcal{L}_{X(t)}}{\mathrm{d}x}$, $t \geq 0$, equal to the Barenblatt solutions to \eqref{eq:Leibenson} (see \eqref{def:Barenblatt-sol} below).
Here and throughout, $\mathrm{d}x$ denotes Lebesgue measure, $W$ is a standard $d$-dimensional Brownian motion and $\mathcal{L}_{X}$ denotes the distribution of a random variable $X$. We call this nonlinear Markov process the \emph{Leibenson process}, see Definition \ref{def:Leibenson-process} and Remark \ref{r5.4}. We significantly extend the techniques and results from the previous paper \cite{R.BR24-pLapl} of the first, third and last named author in which the probabilistic counterpart of the slow diffusion $p$-Laplace special case ($q=1, p>2$) was constructed and was called \emph{$p$-Brownian motion}. 

Before we give details on \eqref{eq:Leibenson}, \eqref{eq:Leibenson-MVSDE} and our results in Part 1 and 2 below, we summarize the main achievements of the present paper. First, in Part 1, we identify the Leibenson equation \eqref{eq:Leibenson} as a nonlinear Fokker--Planck equation (see \eqref{eq:Leibenson-FPE}). This leads to the corresponding singular McKean--Vlasov SDE \eqref{eq:Leibenson-MVSDE} ($p-2$ or $(p-1)(q-1)$ may be strictly negative), for which we then establish weak well-posedness for prescribed one-dimensional solution time marginals. Then we show that the path measures of the solution processes constitute a nonlinear Markov process.
This substantially generalizes the approach from \cite{R.BR24-pLapl} to the doubly nonlinear Leibenson equation \eqref{eq:Leibenson} for a large range of pairs $(p,q) \in (1,\infty) \times (0,\infty)$, including (unlike \cite{R.BR24-pLapl}) the fast diffusion case $p< 2$ (with $q>1$). Furthermore, the corresponding McKean--Vlasov SDE \eqref{eq:Leibenson-MVSDE} is the first of its kind in the literature, since its coefficients depend both on the solutions' time marginal densities as well as on their first and second order derivatives. Second, in Part 2, we finally solve the problem left open in \cite{R.BR24-pLapl} about $p$-Brownian motion even for the more general and doubly nonlinear Leibenson equation by proving that the Leibenson process is \emph{probabilistically strong}, i.e. the solutions to \eqref{eq:Leibenson-MVSDE} with Barenblatt one-dimensional time marginals are adapted measurable functionals of the driving Brownian motion. This is a quite surprising result in view of the strong degeneracy of the diffusion coefficient in \eqref{eq:Leibenson-MVSDE} (which, as shown below, is compactly supported in space) and the irregular drift coefficient, which is merely BV (and not Sobolev regular). Furthermore, in addition to the analytic proof for the crucial restricted uniqueness result for the linearized Leibenson equation (see Theorem \ref{restr-lin-u.theorem} below), we provide a genuinely new probabilistic proof for this result (see Subsection \ref{subsect:restr-lin-u-probabilistic}).
\\

Let us now describe our goals and results in more detail.
Probabilistic representations of solutions to (nonlinear) partial differential equations (PDEs) as time marginal densities of stochastic processes, more precisely of canonically associated (nonlinear) Markov processes, allow to transfer analytic questions on well-posedness, asymptotic behavior and quantitative estimates for PDEs to probabilistic ones in stochastic analysis and vice versa (see \cite{BR-IndianaJ} for an example in the nonlinear case). In the linear case the simplest, yet very instructive example is the relation between Brownian motion and the heat equation
\begin{equation*}
    \partial_t u(t,x) = \Delta u(t,x),\quad (t,x) \in (0,\infty) \times \R^d.
\end{equation*}
Its fundamental solution, the classical heat kernel $p(t,x,y)$, consists of the one-dimensional time marginal law densities of Brownian motion started at $y \in \R^d$ in the sense that
$$
\mathcal{L}_{W^y(t)}(\mathrm{d}x) = p(t,x,y)\mathrm{d}x,\quad \forall t >0, y \in \R^d,
$$
where we set $W^y(t) := \sqrt{2}W(t) + y$ (the factor $\sqrt{2}$ may be exchanged for a factor $\frac 1 2$ on the RHS of the heat equation). Of course, $W^y$ is the unique probabilistically strong solution to the SDE
$$\mathrm{d}X(t) = \sqrt{2} \mathrm{d}W(t),\quad X(0) = y,$$

and the family $\{P_y\}_{y \in \R^d}$ of path laws of $\{W^y\}_{y \in \R^d}$ constitutes a uniquely determined (linear) Markov process with transition kernels $y \mapsto p(t,x,y)\mathrm{d}x$, $t>0$.
\\

\textbf{The Leibenson equation.} In this paper we develop comparable relations for the Leibenson equation \eqref{eq:Leibenson} and the McKean--Vlasov SDE \eqref{eq:Leibenson-MVSDE}. \eqref{eq:Leibenson} is a simultaneous generalization of the porous medium ($p=2$) and the $p$-Laplace equation ($q=1$). In hydrodynamics, \eqref{eq:Leibenson} models filtration of a turbulent compressible fluid in a porous medium.
For more physical background, we refer, for instance, to the introduction of \cite{GS24}. In the linear case $(p,q) = (2,1)$, \eqref{eq:Leibenson} is the heat equation. To our knowledge, the Leibenson equation was first introduced in \cite{L45,L45-2}. In the past decades the theory of existence, regularity, and qualitative and quantitative properties of its solutions was developed, e.g., in \cite{I96,I97,S17,GS24,GS25+,S24} and in further references mentioned therein. In \cite{GS24} it is stated that particularly interesting physical cases arise for $q \geq 1$ and $\frac 3 2 \leq p \leq 2$. We give more details on the admissible values for $(p,q)$ in our results below, but mention already here that we cover large parts of this "physical range". 
We focus on the explicit \emph{Barenblatt solutions} with point source initial datum in the case $q(p-1)>1$ (see Definition \ref{def:Barenblatt-sol}), which were first obtained in \cite{B52_Barenblatt-original} and can be considered as the analgoue of the heat kernel for the heat equation (even though the Barenblatt solutions are not fundamental solutions in the usual sense since, due to the nonlinearity of \eqref{eq:Leibenson}, one cannot build from them solutions for non-Dirac initial data by simply taking convex combination). The other cases $q(p-1) =1$ and $q(p-1)<1$ give rise to completely different regimes: For $q(p-1)>1$ the Barenblatt solutions have finite speed of propagation, whereas in the other two regimes they have infinite speed of propagation, see Section \ref{sect:barenblatt}. We postpone the study of these other regimes and of more general solutions to future work.
\\

\textbf{Goals and results.}
We split our program into two parts. The first part is the construction of a uniquely determined nonlinear Markov process as the probabilistic counterpart of the Barenblatt solutions to \eqref{eq:Leibenson}, consisting of solutions to \eqref{eq:Leibenson-MVSDE}, see Sections \ref{sect:FP-form}-\ref{sect:Leibenson-Markov} and \ref{sect:restr-lin-u}. The second part consists in proving that these solutions are probabilistically strong, see Section \ref{sect:Strongsolution}.
\\

\noindent \textbf{Part 1: Construction of the Leibenson process.}
The derivation of a McKean--Vlasov SDE associated with a nonlinear PDE follows by identifying the latter as a nonlinear Fokker--Planck equation (FPE). These are parabolic equations for measures of type (here and throughout we use Einstein summation convention)
\begin{equation}\label{e1.3}
\partial_t\mu_t=\partial_{ij} (a_{ij}(t,x,\mu_t)\mu_t)-\partial_i(b_i(t,x,\mu_t)\mu_t),\ (t,x)\in(s,\infty)\times\R^d,
\end{equation}
where $s\geq 0$ is the initial time, $a_{ij},b_i: (0,\infty)\times \R^d \times \Mscr \to \R$ are measurable coefficients, $\Mscr$ is some subset of the set of signed locally finite Borel measures on $\R^d$, and solutions are curves $[s,\infty) \ni t\mapsto \mu_t \in \Mscr$ solving \eqref{e1.3} in the (Schwartz) distributional sense (see Definition \ref{dD.2} below). $a = (a_{ij})_{i,j \leq d}$ is assumed to take values in the space of nonnegative definite $d\times d$-matrices, thus there exists a (not necessarily unique) measurable square root $\sigma = (\sigma_{ij})_{i,j \leq d}: (0,\infty)\times \R^d \times \Mscr \to \R^{d \times d}$ with $a = \sigma \sigma^T$. 
If one takes $\mathcal{M}$ to be a set $\mathcal{P}_0$ of probability measures on $\rd$, it is well-known that for any weakly continuous probability measure-valued solution $t\mapsto \mu_t \in \mathcal{P}_0$ to \eqref{e1.3} satisfying an additional rather mild integrability condition (see Definition \ref{dD.2}), there exists a probabilistically weak solution $X = (X(t))_{t\geq s}$ to the McKean--Vlasov SDE,
\begin{eqnarray}\label{e1.5}
    \mathrm{d}X(t)&=&b(t,X(t),\mathcal{L}_{X(t)})\mathrm{d}t+\sqrt{2}\,\sigma(t,X(t),\mathcal{L}_{X(t)})\mathrm{d}W(t),\quad t>s,
\end{eqnarray}
such that $\mathcal{L}_{X(t)} = \mu_t$ for all $t \geq s$, see \cite{Ambrosio2008,Figalli09,Trevisan16,BR18,BR18_2,BRS19-SPpr,SV-book}.
This result, called \emph{superposition principle}, provides a probabilistic counterpart to solutions of \eqref{e1.3} by identifying the latter as the one-dimensional time marginals of the paths of stochastic processes solving \eqref{e1.5}. For this, \emph{no} regularity assumption on the coefficients in $x \in \R^d$ or $\mu \in \mathcal{P}_0$ is needed.

A particularly interesting type of \eqref{e1.3} is given by the \emph{Nemytskii-case}, where $\mathcal{P}_0$ is a set of absolutely continuous measures (w.r.t. Lebesgue measure $\mathrm{d}x$) and the dependence of $a_{ij}$ and $b_i$ on $\mu \in \mathcal{P}_0$ is pointwise via its density, i.e.
$$a_{ij}(t,x,\mu) = \bar{a}_{ij}(t,x,u(x)), \quad b_i(t,x,\mu) = \bar{b}(t,x,u(x)),$$
for Borel coefficients $\bar{a}_{ij},\bar{b}_i : (0,\infty) \times \R^d \times \R \to \R$ and where $\mu(\mathrm{d}x) = u(x)\mathrm{d}x$ (see \cite{BR-LN24} for a large number of concrete examples). Rewriting \eqref{e1.3} as an equation for the density $u(t)$ of $\mu_t$, one arrives at a nonlinear PDE
\begin{equation}\label{eq:Nem-PDE}
    \partial_t u(t,x) = \partial_{ij}\big(\bar{a}_{ij}(t,x,u(t,x))u(t,x)\big) - \partial_i \big(\bar{b}_i(t,x,u(t,x))u(t,x)\big).
\end{equation}
In this case the corresponding McKean--Vlasov SDE is
\begin{align} \label{e1.6}
	\begin{cases}
    \mathrm{d}X(t)&=\bar{b}(t,X(t),u(t,X(t)))\mathrm{d}t+\sqrt{2}\,\bar{\sigma}(t,X(t),u(t,X(t)))\mathrm{d}W(t), \\
    \mathcal{L}_{X(t)}(\mathrm{d}x) &= u(t,x)\mathrm{d}x,
    \end{cases}
\end{align}
where  $\bar{\sigma} = (\bar{\sigma}_{ij})_{i,j \leq d}: (0,\infty)\times \R^d \times \mathcal{P}_0 \to \R^{d \times d}$ is a measurable function with $\bar{a} = \bar{\sigma} \bar{\sigma}^T$.
Thus, for every weakly continuous probability measure-valued solution $u$ to \eqref{eq:Nem-PDE} there is a stochastic process solving \eqref{e1.6}. For instance, solutions to the generalized porous medium equation (PME) perturbed by a nonlinear drift
$$\partial_t u(t,x) = \Delta \beta(u(t,x)) - \divv\big( D(x)b(u(t,x))\big),$$
where $D: \R^d \to \R^d$, $b: \R \to \R_+$ and $\beta: \R \to \R$,
are of this type with $\bar{a}_{ij}(t,x,u(x)) = \delta_{ij} \beta(u(x))$ and $\bar{b}(t,x,u(x)) = D(x) b(u(x))$, see \cite{NLFPK-DDSDE5,BR22,BR-LN24} and the references therein. The classical PME is retrieved with $D = 0$, $\beta(r) = r^m$, $m >0$ (cf. \cite{vazquez2007pme}).

Since the Leibenson equation at first sight is not of type \eqref{eq:Nem-PDE}, the identification of its McKean--Vlasov SDE is not straightforward. Our approach is as follows. First, we show that \eqref{eq:Leibenson} can be rewritten as the nonlinear FPE 
\begin{equation}\label{eq:Leibenson-FPE}
	\partial_t u(t,x) = q^{p-1}\Delta\big(|\nabla u(t,x)|^{p-2} u(t,x)^{(p-1)(q-1)}u(t,x)\big) -q^{p-1} \divv\big( \nabla \big(  |\nabla u(t,x)|^{p-2} u(t,x)^{(p-1)(q-1)}   \big)  u(t)    \big).
\end{equation}
This is indeed of type \eqref{e1.3} for coefficients
\begin{align}\label{eq:Leibenson-FPE:coefficients}
    a_{ij}(t,x,\mu_t) = q^{p-1}\delta_{ij} |\nabla u(t,x)|^{p-2} u(t,x)^{(p-1)(q-1)}, \quad b_i(t,x,\mu_t)= q^{p-1} \partial_i \big(|\nabla u(t,x)|^{p-2} |u(t,x)|^{(p-1)(q-1)}  \big), 
\end{align}
where $\mu_t(\mathrm{d}x) = u(t,x)\mathrm{d}x$. Details, in particular a proof of the equivalence of \eqref{eq:Leibenson} and this FPE, are given in Section \ref{sect:FP-form}. The coefficients $a$ and $b$ from \eqref{eq:Leibenson-FPE:coefficients} are not of Nemytskii-type, which would require $a_{ij}(t,x,\cdot)$ to depend on $\mu_t(\mathrm{d}x) = u(t,x)\mathrm{d}x$ only via $u(t,x)$ (and similarly for $b_i$). Instead, due to the dependence on the gradient of $u(t)$, $a_{ij}(t,x,\cdot)$ depends on $\mu_t$ via the values of $u(t)$ in an arbitrary small open neighborhood of $x$ (similarly for $b_i$). Then the associated McKean--Vlasov SDE is \eqref{eq:Leibenson-MVSDE}.

Our first result is the following construction of the probabilistic counterpart of solutions to the Leibenson equation.
\\

\noindent\textbf{Theorem 1 (see Thm. \ref{thm:SP-princ-general}, Cor. \ref{kor:SP-princ-Leibenson}, Thm. \ref{prop:SP-for-Barenblatt} for details).} \,
Let $d \geq 1$, $p>1$, $q>0$.
For every solution $u$ to \eqref{eq:Leibenson} consisting of probability densities with suitable regularity and integrability assumptions there is a probabilistically weak solution to \eqref{eq:Leibenson-MVSDE}.
If additionally $q>\frac 1 {p-1}$ and $p> \frac{1+d}{d}$, this result applies to the Barenblatt solutions $u = w^y$, $y \in \R^d$, (see Definition \ref{def:Barenblatt-sol}) to \eqref{eq:Leibenson}.
\\

In the linear case, i.e. when $a_{ij}$ and $b_i$ in \eqref{e1.3} do not depend on $\mu$, it is well-known that if for every initial datum $(s,\delta_y)$ the FPE \eqref{e1.3} has a unique weakly continuous probability solution $(\mu^{s,y}_t)_{t\geq s}$, then the corresponding SDE \eqref{e1.5} has a unique probabilistically weak solution from any initial datum $(s,\delta_y)$ as well, and the family of these solutions forms a Markov process (for the definition of a Markov process see Section \ref{subsect:NL-MP}) with transition kernels $(y,A) \mapsto \mu^{s,y}_t(A)$. This is \emph{not} true for nonlinear FPEs and their associated McKean--Vlasov SDEs, since the proof relies on the stability of the associated martingale problem under convex combinations, which is lost in the nonlinear case. To overcome this issue, inspired by McKean \cite{McKean1-classical}, in \cite{R./Rckner_NL-Markov22} the third and last-named author introduced \emph{nonlinear Markov processes}, and showed that under general assumptions a solution family to a nonlinear FPE gives rise to a uniquely determined nonlinear Markov process, consisting of solution path laws to the associated McKean--Vlasov SDE (see Section \ref{subsect:NL-MP} for details). We stress that this result applies to ill-posed situations, i.e. uniqueness of the nonlinear FPE-solutions is not required, but only a "restricted linearized uniqueness" result, see Theorem \ref{t52} and Lemma \ref{lem:equiv}. In previous work (cf. e.g. \cite{R.BR24-pLapl,BR-IndianaJ,BRZ23,R./Rckner_NL-Markov22,R.Romito24} and also \cite{BR-LN24} and the references therein) this was utilized to construct nonlinear Markov processes with one-dimensional time marginals given by solutions to nonlinear PDEs, as e.g. Burgers' and (generalized) PMEs, as well as $2D$ vorticity Euler and Navier--Stokes equations. In each of these cases, it was first shown that the PDE under consideration can be written as a Neymtskii-type FPE.

In the present paper we construct a nonlinear Markov process associated with the Barenblatt solutions to \eqref{eq:Leibenson} via \cite[Theorem 3.8]{R./Rckner_NL-Markov22}. In particular, we have to prove a delicate restricted linearized uniqueness result, which boils down to a highly nontrivial uniqueness result for the degenerate second-order linear PDE \eqref{eq:Leibenson-FPE-lin}, see Theorem \ref{restr-lin-u.theorem}.
The conditions on $p,q$ in Theorem 2 below are essentially determined by the conditions in Theorem \ref{restr-lin-u.theorem}, which split into two different regimes, for which we prove the uniqueness result separately. One part, proven via analytic methods, is a close adaptation of the proof for the corresponding result in the $p$-Laplace special case. The second part is based on genuinely new purely probabilistic methods. Our result is:
\\

\noindent
\textbf{Theorem 2 (see Thm. \ref{t5.2} and Def. \ref{def:Leibenson-process} for details).} \,
Let $d \geq 2$, $p>\frac{d}{d-1}$, $q>\frac 1 {p-1}$. If $p<2$ assume additionally $q>\frac{2-p+d}{d(p-1)} (>\frac 1 {p-1})$. The family of McKean--Vlasov SDE-solutions from Theorem 1 with one-dimensional time marginals equal to the Barenblatt solutions to \eqref{eq:Leibenson}  constitutes a uniquely determined nonlinear Markov process, which we call \emph{Leibenson process}.
\\
\\
The construction of the Leibenson process is analogous to the construction of Brownian motion as a Markov process from the classical heat kernel. The slow diffusion $p$-Laplace special case of Theorems 1 and 2, i.e. $p>2$ and $q=1$, was proven in \cite{R.BR24-pLapl}, where the corresponding nonlinear Markov process was called \emph{$p$-Brownian motion} (for $p=2$, one recovers Brownian motion).
\\

\noindent \textbf{Part 2: The Leibenson process is probabilistically strong.}
By construction, the Leibenson process from Theorem 2 consists of probabilistically weak solutions to \eqref{eq:Leibenson-MVSDE}. It has been an imminent and challenging open question whether these solutions are actually probabilistically strong, i.e. measurable adapted functionals of the driving Brownian motion and the initial condition. For details on the definition of probabilistically strong solutions, see Definition \ref{def:DDSDE_strong}. In this paper, we answer this question affirmatively via the following result. Note that this includes as a new result the $p$-Brownian motion case ($p>2, q=1$), for which the question of being a strong solution to its corresponding McKean--Vlasov SDE was left open in \cite{R.BR24-pLapl}.
\\

\noindent
\textbf{Theorem 3 (see Thm. \ref{StrongSolution.theorem:existenceStrongSolution} for details).} \,
Let $d \geq 2, p>\frac{d-1}{d}, q>\frac{|p-2|+d}{d(p-1)}(>\frac 1{p-1})$.
Then, the probabilistically weak solution $(X,W)$ to \eqref{eq:Leibenson-MVSDE} with $u=w^y$ from Theorem 1 are probabilistically strong from any strictly positive time $\delta>0$ on (i.e. they are measurable adapted functionals of the driving Brownian motion $W$ and $X(\delta)$). 
\\

\noindent The proof follows from a \textit{restricted Yamada--Watanabe theorem} \cite[Theorem 1.3.1]{Grube-thesis},\cite{grube2021strong}, which in turn relies on the existence of a probabilistically weak solution with one-dimensional time marginal law densities $w^y(t+\delta), t\geq 0$, (as provided by Theorem 1) and a restricted pathwise uniqueness result for \eqref{eq:Leibenson-MVSDE}, see Theorem \ref{StrongSolution.theorem:pathwiseU}. Here, \emph{restricted} means uniqueness in the class of solutions with one-dimensional time marginals $w^y(t+\delta,x)\mathrm{d}x$, $t \geq 0$. Thus these marginals may be fixed beforehand, which allows to simplify the McKean--Vlasov SDE \eqref{eq:Leibenson-MVSDE} to a non-distribution dependent SDE
\begin{align*}
    \mathrm{d}X(t)=\tilde{b}(t,X(t))\mathrm{d}t+\tilde{\sigma}(t,X(t))\mathrm{d}W(t),
\end{align*}
with coefficients (using the definition of $b_i,a_{ij}$ from \eqref{eq:Leibenson-FPE:coefficients})
\begin{align*}
    \tilde{b}_i(t,x) := b_i(t,x,w^y(t+\delta,x)\mathrm{d}x),\quad \tilde{\sigma}_{ij}(t,x):= \sqrt{2a_{ij}(t,x,w^y(t+\delta,x)\mathrm{d}x}), \quad (t,x)\in \R_+\times\rd.
\end{align*}

\noindent The proof of this pathwise uniqueness result poses one of the main technical challenges of the present paper due to the degeneracy and irregularity of the diffusion coefficient $\tilde{\sigma}$, as well as the low differentiability of the drift coefficient.
Indeed, note that $\tilde{\sigma}(t,\cdot)$ is compactly supported and not Lipschitz-continuous. For $p<2$, $\tilde{\sigma}(t,\cdot)$ is even unbounded and discontinuous. Moreover, $\tilde{b}$ is only a function of bounded variation in its spatial coordinate.

Although the literature on pathwise uniqueness results for SDEs, also in degenerate cases, is vast (see, e.g., \cite{roeckner2010weakuniqueness, luo2011hoelder, kumar2013degenerate, luo2014uniq, wang2020existence, lee2023pathwise}, and in particular \cite{champagnat2018} for a result for bounded Sobolev-coefficients with certain drifts of (spatial) bounded variation, see also \cite{grube2024strongSDE} for a recent generalization), we could not apply any of these results for our desired result.

Our proof heavily relies on a careful analysis of the degeneracy and discontinuity of $\tilde{\sigma}$ and (spatial) jump sets of $\tilde{b}(t,\cdot),t>0,$ which in turn, of course, crucially relies on the explicit form of the Barenblatt solutions $w^y$. This analysis is based on the generalized Lipschitz-type estimate, sometimes called \textit{Crippa--De Lellis estimate} (\cite{crippa2008estimates,champagnat2018})
\begin{align}\label{introduction.lipschitz-type-esimate}
	|f(x)-f(y)|\lesssim_d (\M|\nabla f|(x)+\M|\nabla f|(y))|x-y|\ \text{ for $\mathrm{d}x$-a.e. } x,y \in \rd,
\end{align}
(cf. Appendix \ref{App-D}),
where $f\in BV(\rd)$, and $(\nabla f)_i$ denotes the (signed) measure-valued representative of the Schwartz distributional derivative in direction $x_i$ of $f$. Here, $\M$ denotes the usual (Hardy--Littlewood) maximal operator.
The above estimate applied to $\tilde{b}(t,\cdot)$ is essential for our pathwise uniqueness result. Here, the key is to control the maximal function, which involves the singular part (with respect to Lebesgue measure) of the drift's gradient, which turns out to be comparable to the $(d-1)$-dimensional Hausdorff (or surface) measure restricted to the boundary of the spatial support of the Barenblatt solution considered as a Borel measure on $\rd$, see Lemma \ref{lemma.drift.regularity}. 
In Lemma \ref{lemma.PU.b.uniformBound:1}, we control the maximal operator applied to this measure by the estimate \eqref{lemma.PU.b.uniformBound:1}.
This estimate is optimal, see Remark \ref{lemma.PU.b.uniformBound:remark-optimality}, which can be seen from explicit calculation using elementary geometric arguments in dimension $d=3$, see Lemma \ref{App-D.lemma.maximalFunction.surfaceMeasure} (ii) in Appendix \ref{App-D}.
We also employ \eqref{introduction.lipschitz-type-esimate} for differences of certain powers of $\tilde{\sigma}$.
In contrast to existing results where the diffusion coefficient is in $H^1$ with respect to its spatial coordinate (cf. e.g. \cite{bossy2019moderated,grube2021strong,grube2022strongTime}), which yields sufficient control over the maximal function on the right hand-side of the inequality when applied to $f=\tilde{\sigma}(t,\cdot)$ (with power $1$), such a control fails in the present case. To overcome this issue, we employ an inequality based on the mean-value theorem inspired by \cite{gess2023inventiones}, see \eqref{StrongSolutions.theorem.PU:MVTineq}. A similar idea was also used in \cite{grube2024strongDegenerate}. Furthermore, we use and extend the method from the latter work in order to deal with the degeneracy and low regularity of $\tilde{\sigma}$.

\paragraph{Structure of the paper.}
In Section \ref{sect:FP-form} we introduce the Leibenson equation and its reformulation as a nonlinear FPE.
Section \ref{sect:Leibenson-MVSDE} contains the associated McKean--Vlasov SDE and the \emph{superposition principle}, via which we lift solutions to the Leibenson equation to solutions to this SDE. 
In Section \ref{sect:barenblatt}, we consider the Barenblatt solutions and apply to them this lift, while in Section \ref{sect:Leibenson-Markov} we construct the \emph{Leibenson process}, a nonlinear Markov process constructed via the superposition principle from the Barenblatt solutions. A key ingredient for this is a delicate uniqueness result for a corresponding linear PDE, which is stated and proven in Section \ref{sect:restr-lin-u}. The proof consists of an analytic and a probabilistic part. Before, we prove in Section \ref{sect:Strongsolution} that the probabilistically weak solutions of the Leibenson process are probabilistically strong. The appendix contains some material on Fokker--Planck equations, Hardy--Littlewood maximal function and details on a few aspects of the proof of Theorems \ref{StrongSolution.theorem:pathwiseU} and \ref{restr-lin-u.theorem}.

\paragraph{Notation.}

We write $\R_+ := [0,\infty)$, $x \cdot y$ or $\langle x,y\rangle_\rd$ for the usual inner product on $\R^d$, $f_+=\max(f,0)$ for a real-valued function $f$, and $\mathrm{d}x$ for Lebesgue measure on $\R^d$. 

\emph{Classical function spaces.} For $U \subseteq \R^l$, $l \in \N$, we denote by $C^n(U,\R^k)$ the space of $n$-times continuously differentiable maps $f : U \to \R^k$, and write $C(U,\R^k) = C^0(U,\R^k)$.
We write $C(U,\R^k)_0$ to denote the set of all $f\in C(U,\R^k)$ with $f(0)=0$.
More generally, for a topological space $X$, we write $C(U,X)$ for the set of continuous maps from $U$ to $X$. 
Let $t\in U$. By $\pi_t: C(U,X) \to X$, we denote the canonical evaluation map at $t$, i.e. $\pi_t(w):=w(t)$.
$C_c^n(U,\R^k)$ and $C_b^n(U,\R^k)$ are the subsets of $C^n(U,\R^k)$ of compactly supported functions (in $U$) and those functions which together with its partial derivatives up to order $n$ are bounded, respectively. We set $C^\infty(U,\R^k) = \bigcap_{n\in \N}C^n(U,\R^k)$, and similarly for $C^\infty_c$ and $C^\infty_b$. Partial time and spatial derivatives of a sufficiently smooth function are denoted $\partial_t f = \frac{\partial f}{\partial_t}$, $\partial_if = \frac{\partial f}{\partial x_i}$ and $\partial_{ij}f = \frac{\partial }{\partial x_i} \frac{\partial f}{\partial x_j}$.

\emph{$L^p$- and Sobolev spaces.} For $p \in [1,\infty]$ and $U \subseteq \R^l$, the usual spaces of (equivalence classes of) Borel measurable $p$-integrable (w.r.t. $\mathrm{d}x$) functions $f: U \to \R^k$ are denoted $L^p(U,\R^k)$ with norm $|\cdot|_{L^p(U,\R^k)}$, shortly $|\cdot|_{L^p(U)}$ or simply $|\cdot|_{L^p}$ when no confusion can occur. $L^p_{\textup{loc}}(U,\R^k)$ are the corresponding spaces of locally integrable functions. When integrability is considered with respect to a measure $\mu$ instead of $\mathrm{d}x$, we write $L^p(U,\R^k; \mu)$. $W_{(\textup{loc})}^{m,p}(U,\R^k)$, $m \in \N, p \in [1,\infty]$, are the usual Sobolev spaces of (locally) $p$-integrable $m$-times weakly differentiable functions $f: U \to \R^k$ with (locally) $p$-integrable weak derivatives up to order $m$; we write $H_{(\textup{loc})}^m(U,\R^k) = W_{(\textup{loc})}^{m,2}(U,\R^k)$, and $H^{-1}(U,\R^k)$ for the corresponding topological dual space. All these spaces are equipped with their usual norms. If $U \subseteq I \times \R^d$, $I \subseteq \R_+$, we write $W^{m,n}_p(U)$ for the space of $p$-integrable maps with $n$ $p$-integrable weak derivatives in $t \in I$ and $m$ $p$-integrable weak derivatives in $x \in \R^d$. For a normed space $(X,||\cdot||)$, by $L_{(\textup{loc})}^p(I;X)$ we denote the space of measurable functions $\varphi: I \,(\subseteq \R_+) \to X$ such that $[t\mapsto ||\varphi (t)||] \in L^p_{(\textup{loc})}(I,\R)$. For all function spaces introduced here, we omit the state space $\R^k$ from the notation when $k =1$. 

\emph{Measures and spaces of measures.} $\mathcal{M}^+_b$ is the space of nonnegative finite Borel measures on $\R^d$. The vague and weak topology on it are defined as usual, i.e. as the initial topology of the maps $\mu \mapsto \int_{\R^d} h \,\mathrm{d}\mu$ for all $h \in C_c(\R^d)$ and $h \in C_b(\R^d)$, respectively. For a topological space $X$, we write $\mathcal{P}(X)$ for the space of probability measures on the Borel $\sigma$-algebra $\Bscr(X)$, and simply $\mathcal{P} = \mathcal{P}(\R^d)$. $\delta_x$ is the Dirac measure in $x$. For a random variable $X$ between a measure space $(X,\mathcal{X},\mu)$ and some measurable space, we write $\mathcal{L}_X = \mu \circ X^{-1}$ for its distribution, i.e. the image (or pushforward) measure of $X$ with respect to $\mu$. This includes the case where $X = (X(t))_{t\geq 0}$ is a stochastic process on a probability space $(\Omega,\mathcal{F},\mathbb{P})$, considered as a path-valued map $\omega \mapsto X(\omega)$. We then say $\mathcal{L}_X$ is the (path) law of $X$.
Let $M\subset \R^n$ be a submanifold. By $S$, we denote the standard \textit{surface measure} on $(M,\mathcal{B}(M))$. Here, $M$ is considered with respect to the subspace topology of $\R^n$. We only consider the case where $M$ is a $(d-1)$-dimensional sphere.

\emph{Functions of bounded variation.} 
$f\in L^1(\R^l,\R^k)$ is said to be \textit{of bounded variation}, in symbol $f\in BV(\R^l;\R^k)$, if all its components' first-order Schwartz distributional derivatives $\partial_{i}f^j$ are given by a finite signed Borel measure on $\R^l$, for all $1\leq i\leq l,\ 1\leq j \leq k$.

\emph{(Hardy--Littlewood) Maximal function/operator and Muckenhoupt weights.} For a brief introduction to both maximal operator (in symbol '$\M$') and Muckenhoupt weights (in symbol '$A_p$'), we refer to Appendix \ref{App-D}.

\emph{Probability theory.} We say that $(\Omega,\Fscr,(\Fscr_t),\mathbbm{P})$ is a \textit{stochastic basis} if $(\Omega,\Fscr, \mathbbm{P})$ is a complete probability space and $(\Fscr_t)$ a right-continuous filtration on $\Omega$ augmented by the $\mathbbm{P}$-zero sets.

\emph{Miscellaneous.} Let $a,b \in \R$. 
We write $a \lesssim_{(t)} b$ and $a\lesssim b$ if $a\leq C_{t} b$ or $a \leq C b$, where $C_t, C>0$ are constants dependent and independent from a parameter $t$, respectively.
Likewise, we write $a \cong_t b$ and $a \cong b$, if $a = Cb$ for a constant $C$ dependent or independent from $t$, respectively.

\section{The Leibenson equation and its Fokker--Planck reformulation}\label{sect:FP-form}
Let $d \in \N$, $q >0, \,p >1$. 
\begin{dfn}\label{def:Leibenson-eq}
$u: (0,\infty)\times \R^d \to \R_+$ is a \emph{weak solution} to the Leibenson equation \eqref{eq:Leibenson}, if
\begin{equation}\label{prop-def1}
u^q \in L^1_{\textup{loc}}((0,\infty);W^{1,1}_{\textup{loc}}(\R^d)),\quad u, |\nabla u^q|^{p-1} \in L^1_{\textup{loc}}((0,\infty)\times \R^d),
\end{equation}
$t\mapsto u(t,x)\mathrm{d}x$ is vaguely continuous, and for all $\psi \in C^1_c((0,\infty)\times \R^d)$
\begin{equation}\label{eq:def-Leibenson-eq}
\int_{(0,\infty)\times \R^d} -u\partial_t \psi + |\nabla u^q|^{p-2}\nabla u^q\cdot \nabla \psi\, \mathrm{d}x\mathrm{d}t = 0.
\end{equation}
	$u$ \emph{has initial condition} $\nu \in \Mscr_b^+$, if $u(t,x)\mathrm{d}x \xrightarrow{t \to 0} \nu$ vaguely.
$u$ is called \emph{probability solution}, if $\mathrm{d}t$-a.e. $u(t,x)\mathrm{d}x$ and its initial datum are probability measures.
\end{dfn}
The only reason why we restrict attention to nonnegative solutions is that, ultimatively, we will only be interested in probability solutions.

The following equivalent formulation via time-independent test functions will be useful. The equivalence can be proven as in the proof of \cite[Prop.6.1.2]{FPKE-book15}.
\begin{lem}\label{lem:equiv-Leibenson}
If $u$ satisfies \eqref{prop-def1}, $t\mapsto u(t,x)xd$ is vaguely continuous and
\begin{equation}\label{int-in-0}
	|\nabla u^q|^{p-1} \in L^1_{\textup{loc}}([0,\infty)\times \R^d),
\end{equation}
then $u$ is a weak solution with initial condition $\nu$ if and only if
\begin{equation}\label{e2.3}
	\int_{\R^d}\varphi \,u(t)\,\mathrm{d}x = \int_{\R^d} \varphi \,\mathrm{d}\nu  - \int_0^t \int_{\R^d} |\nabla u^q|^{p-2}\nabla u^q \cdot \nabla \varphi \,\mathrm{d}x \mathrm{d}t,\quad\forall t \geq 0,
\end{equation}
for every $\varphi \in C^1_c(\R^d)$. 
\end{lem}
Even though the solutions we consider are vaguely continuous, the following remark may be useful for extensions of our results to (a priori) discontinuous solutions.
\begin{rem}\label{rem:Leibenson-eq}
 The vague continuity assumption is not restrictive: If $u$ satisfies \eqref{int-in-0} as well as all assumptions of Definition \ref{def:Leibenson-eq} except for the continuity assumption and, in addition, 
 $$\esssup_{t>0}|u(t)|_{L^1(\R^d)} <\infty,$$
 then $t \mapsto u(t,x)\mathrm{d}x$ has a unique vaguely continuous (measure-valued) $\mathrm{d}t$-version on $[0,\infty)$. This follows from \cite[Lemma 2.3]{Rehmeier-nonlinear-flow-JDE}.  
		Moreover, if $|\nabla u^q|^{p-1} \in L^{1}_{\textup{loc}}([0,\infty),L^{1}(\R^d))$, then $[0,\infty)\ni t \mapsto |u(t)|_{L^1(\R^d)}$ is constant. The latter follows by considering \eqref{e2.3} for an increasing sequence $(\varphi_n)_{n\in \N}\subseteq C^1_c(\R^d)$ such that $0\leq \varphi_n \leq 1$, $\varphi_n(x) = 1$ for $|x| < n$, $\sup_n|\nabla \varphi_n(x)|\leq C$ for all $x\in \R^d$ for some $C>0$ not depending on $n$ or $x$, and by letting $n \to \infty$.
\end{rem}

\paragraph{Associated nonlinear Fokker--Planck equation.}
Let $\Fscr = \Fscr(p,q,d)\subseteq W^{1,1}_{\textup{loc}}(\rd)$ be the set of those nonnegative functions $u\in W^{1,1}_{\textup{loc}}(\R^d)$ such that $|\nabla u|^{p-2} u^{(p-1)(q-1)} \in W^{1,1}_{\textup{loc}}(\R^d)$. Note that both $p-2$ and $(q-1)(p-1)$ could be negative and that hence some care is required in writing expressions as $|\nabla u|^{p-2} u^{(q-1)(p-1)}$, since $|\nabla u|$ and $u$ may vanish. For such terms to be well-defined in the sense of Sobolev functions, an implicit assumption is that either functions raised to a negative power are strictly positive $\mathrm{d}x$-a.s. or, for products, that zeros raised to a negative power are multiplied by zeros raised to a positive power. In the latter case we use the convention $0\times  \infty = 0$.

Then let $a_{ij},b_i: \Fscr \times \R^d \to \R$, $ 1 \leq i,j \leq d$, be defined as in \eqref{eq:Leibenson-FPE:coefficients}, i.e. 
\begin{equation*}
	a_{ij}(u,x) = q^{p-1}\delta_{ij} |\nabla u(x)|^{p-2} u(x)^{(p-1)(q-1)}, \quad b_i(u,x)= q^{p-1} \partial_i \big(|\nabla u(x)|^{p-2} |u(x)|^{(p-1)(q-1)}  \big)
\end{equation*}
and consider the nonlinear Fokker--Planck equation (here and throughout using Einstein summation convention)
\begin{equation*}
	\partial_t u(t,x) = \partial_{ij}\big(a_{ij}(u(t),x)u(t,x)\big) - \partial_i\big(b_i(u(t),x)u(t,x)\big),\quad (t,x)\in (0,\infty)\times \R^d,
\end{equation*}
i.e. \eqref{eq:Leibenson-FPE}.
We stress that these coefficients are not of the usual Nemytskii-type, since they depend pointwise on the gradient of the solution $u$, hence on the values of $u$ in a(n arbitrary small) neighborhood of $x$.

\begin{dfn}\label{def:Leibenson-FP}
	A nonnegative function $u \in L^1_{\textup{loc}}((0,\infty)\times \R^d)$ is a \emph{distributional solution} to \eqref{eq:Leibenson-FPE}, if $u(t) \in \Fscr$ for all $t >0$, $t \mapsto u(t,x)\mathrm{d}x$ is vaguely continuous,
\begin{equation}\label{int-assumpt}
	|\nabla u|^{p-2} u^{(q-1)(p-1)+1}   , \,\,\,  u\, \partial_i \big(  |\nabla u|^{p-2} u^{(q-1)(p-1)}   \big) \in L^1_{\textup{loc}}((0,\infty)\times \R^d), \,\, i \in \{1,\dots,d\},
\end{equation}
	and for all $\psi \in C^2_c((0,\infty)\times \R^d)$
\begin{equation}\label{eq:def-FP-eq}
	\int_{(0,\infty)\times \R^d} \bigg(\partial_t \varphi + q^{p-1}|\nabla u|^{p-2} u^{(p-1)(q-1)}\Delta \psi + q^{p-1} \nabla \big( |\nabla u|^{p-2} u^{(p-1)(q-1)} \big) \cdot \nabla \psi   \bigg)u\, \mathrm{d}x \mathrm{d}t = 0.
\end{equation}
$u$ \emph{has initial condition} $\nu \in \Mscr^+_b$, if $u(t,x)\mathrm{d}x \xrightarrow{t\to 0}\nu$ vaguely. $u$ is called \emph{probability solution}, if $\nu$ and $\mathrm{d}t$-a.e. $u(t,x)\mathrm{d}x$ are Borel probability measures.
\end{dfn}
The following lemma can be proven exactly as \cite[Prop. 6.1.2]{FPKE-book15}.
\begin{lem}\label{lem:FP-equiv}
	If $u$ is as in the previous definition and satisfies \eqref{int-assumpt} with $[0,\infty)$ replacing $(0,\infty)$, then $u$ is a distributional solution with initial condition $\nu$ if and only if for all $t \geq 0$ and $\varphi \in C^2_c(\R^d)$
	\begin{equation}\label{a1}
		\int_{\R^d} \varphi \, u(t)\, \mathrm{d}x = \int_{\R^d} \varphi \,\mathrm{d}\nu + \int_0^t \int_{\R^d}   q^{p-1}\bigg( |\nabla u|^{p-2} u^{(p-1)(q-1)}\Delta \varphi  +\nabla  \big( |\nabla u|^{p-2} u^{(p-1)(q-1)} \big)  \cdot \nabla \varphi \bigg)   \, u \, \mathrm{d}x \mathrm{d}t.
	\end{equation}
\end{lem}
An analogue of Remark \ref{rem:Leibenson-eq} holds in the Fokker--Planck case. We compare Definitions \ref{def:Leibenson-eq} and \ref{def:Leibenson-FP}:
\begin{lem}\label{lem:equiv-sol-notion}
Assume $u$ satisfies \eqref{prop-def1} with 
\begin{equation}\label{eq:chain-rule}
\nabla u^q = q u^{q-1} \nabla u,
\end{equation}
 $u(t) \in \Fscr$ for all $t >0$, \eqref{int-assumpt}, and $t\mapsto u(t,x)\mathrm{d}x$ is vaguely continuous.
Then $u$ is a weak solution to \eqref{eq:Leibenson} in the sense of Definition \ref{def:Leibenson-eq} if and only if it is a distributional solution to \eqref{eq:Leibenson-FPE} in the sense of Definition \ref{def:Leibenson-FP}.
 
If in addition $u$ satisfies all integrability assumptions in time on $[0,\infty)$ instead of $(0,\infty)$ and has initial condition $\nu$, then $u$ satisfies \eqref{e2.3} if and only if it satisfies \eqref{a1} (for all $\varphi \in C^1_c(\R^d)$ and $ C^2_c(\R^d)$, respectively).
\end{lem}
Note that \eqref{eq:chain-rule} does not follow from \eqref{prop-def1}, since the latter does not imply local integrability of $u^{q-1} \nabla u$ (in particular, note that $q-1$ need not be positive).
\begin{proof}
By assumption all regularity- and integrability properties of Definitions \ref{def:Leibenson-eq} and \ref{def:Leibenson-FP} are satisfied, so it remains to prove equivalence of \eqref{eq:def-Leibenson-eq} and \eqref{eq:def-FP-eq}. To this end, let $\psi \in C^2_c((0,\infty)\times \R^d)$. Then
\begin{align*}
	&\int_{(0,\infty)\times \R^d}|\nabla u^q|^{p-2}\nabla u^q\cdot \nabla \psi\, \mathrm{d}x\mathrm{d}t = q^{p-1} \int_{(0,\infty)\times \R^d} u^{(p-1)(q-1)}|\nabla u|^{p-2} \nabla u \cdot \nabla \psi \, \mathrm{d}x \mathrm{d}t
	\\& =q^{p-1}\int_{(0,\infty)\times \R^d} \nabla \bigg( u^{(p-1)(q-1)}|\nabla u|^{p-2} u\bigg)\cdot \nabla \psi - \nabla \bigg(u^{(p-1)(q-1)} |\nabla u|^{p-2}\bigg) u \cdot \nabla \psi  \, \mathrm{d}x \mathrm{d}t
	\\& = -q^{p-1} \int_{(0,\infty)\times \R^d}  \bigg(|\nabla u|^{p-2} u^{(p-1)(q-1)}\Delta \psi +  \nabla \big(u^{(p-1)(q-1)} |\nabla u|^{p-2}  \big)\cdot \nabla \psi  \bigg) \,u\, \mathrm{d}x \mathrm{d}t,
\end{align*}
where the first equality is due to \eqref{eq:chain-rule}, and the second follows from the product rule for Sobolev functions. To extend \eqref{eq:def-Leibenson-eq} to all $\psi \in C^1_c((0,\infty)\times \R^d)$, one uses a simple approximation argument. This concludes the first part of the proof. The second part follows immediately from the first part and Lemmas \ref{lem:equiv-Leibenson} and \ref{lem:FP-equiv}.
\end{proof}

\section{Associated McKean--Vlasov SDE and superposition principle}\label{sect:Leibenson-MVSDE}
For any nonlinear Fokker--Planck equation there is an associated McKean--Vlasov SDE with drift- and diffusion-coefficient given by the drift and the square root of twice the diffusion coefficient of the nonlinear Fokker--Planck equation. Thus, the McKean--Vlasov SDE associated with \eqref{eq:Leibenson-FPE} is
\begin{align}\label{eq:DDSDE}
\begin{cases}
	&\mathrm{d}X(t) = q^{p-1}\nabla \bigg(|\nabla u(t,X(t))|^{p-2}   u(t,X(t))^{(q-1)(p-1)}\bigg)\mathrm{d}t   + \sqrt{2q^{p-1}}\bigg( |\nabla u(t,X(t))|^{\frac{p-2}{2}}u(t,X(t))^{\frac{(q-1)(p-1)}{2}}\bigg) \mathrm{d}W(t)
	\\& \mathcal{L}_{X(t)}(\mathrm{d}x) = u(t,x)\mathrm{d}x,
\end{cases}
\end{align}
where $W = (W(t))_{t\geq 0}$ is a standard $d$-dimensional Brownian motion.
We stress that here $u$ is \emph{not} a priori given, but part of the solution.

\begin{dfn}\label{def:DDSDE}
A \emph{probabilistically weak solution} to \eqref{eq:DDSDE} is an adapted stochastic process $X = (X(t))_{t\geq 0}$ on a stochastic basis $(\Omega, \Fscr, (\Fscr_t)_{t\geq 0},\mathbb{P})$ with an $(\Fscr_t)$-standard Brownian motion $W$ such that $\mathcal{L}_{X(t)}(\mathrm{d}x) = u(t,x)\mathrm{d}x$ with $u(t,\cdot)\in \Fscr$ for all $t>0$,
\begin{equation}\label{def:DDSDE:integrabilityCondition}
	\mathbb{E}\bigg[\int_0^T \bigg( \big| \nabla \big(|\nabla u(t,X(t))|^{p-2}   u(t,X(t))^{(q-1)(p-1)}\big) \big|    +   |\nabla u(t,X(t))|^{p-2}   u(t,X(t))^{(q-1)(p-1)} \bigg) \mathrm{d}t    \bigg] \quad \forall T>0,
\end{equation}
and $\mathbb{P}$-a.s.
\begin{align*}
	X(t) = X(0) + q^{p-1} \int_0^t & \nabla \big(|\nabla u(t,X(t))|^{p-2}   u(t,X(t))^{(q-1)(p-1)}\big) \mathrm{d}t  
	\\&+ \sqrt{2q^{p-1}} \int_0^t  |\nabla u(t,X(t))|^{\frac{p-2}{2}}u(t,X(t))^{\frac{(q-1)(p-1)}{2}} d W(t), \quad \forall t \geq 0.
\end{align*}
The \emph{initial condition} of $X$ is the probability measure $\mathcal{L}_{X(0)}$. We call the probability measure $P$ on $\Bscr(C(\R_+,\R^d))$, $P := \mathbb{P}\circ X^{-1}$, where $X$ is considered as the path-valued map $X: \Omega \to C(\R_+,\R^d)$, $X(\omega) = [t\mapsto X(t)(\omega)]$, a \emph{solution (path) law} to \eqref{eq:DDSDE}.
\end{dfn}
Strictly speaking, a weak solution is a tuple $((\Omega, \Fscr, (\Fscr)_{t\geq 0}, \mathbb{P}), X, W)$. However, we usually shortly refer to $X$ as the weak solution or, when it is relevant in the context of strong solutions (see Definition \ref{def:DDSDE_strong}), to $(X,W)$.

\paragraph{Superposition principle.}
The crucial relation between a nonlinear Fokker--Planck equation and its associated Mckean--Vlasov SDE is the following: By Itô's formula, the curve of one-dimensional time marginals of any McKean--Vlasov probabilistically weak solution is a weakly continuous probability solution to the nonlinear Fokker--Planck equation (in the sense of Definition \ref{dD.2}). Conversely, by the \emph{superposition principle}, see \cite{Trevisan16,BR18,BR18_2}, for any such distributional solution with sufficient spatial integrability, there exists a probabilistically weak solution to the McKean--Vlasov SDE with one-dimensional time marginals equal to this distributional solution. Now we state these relations precisely for the nonlinear FPE \eqref{eq:Leibenson-FPE} and its associated McKean--Vlasov SDE \eqref{eq:DDSDE}.
\begin{prop}
Let $X = (X(t))_{t\geq 0}$ be a probabilistically weak solution to \eqref{eq:DDSDE}. Then $u$, given by $t \mapsto u(t,x)\mathrm{d}x := \mathcal{L}_{X(t)}(\mathrm{d}x)$, is a weakly continuous probability solution to \eqref{eq:Leibenson-FPE}. Thus, it is a weak solution to \eqref{eq:Leibenson}, provided $u$ satisfies \eqref{prop-def1} and \eqref{eq:chain-rule}.
\end{prop}
\begin{proof}
The first assertion follows from a straightforward application of Itô's formula, while the second part follows from Lemma \ref{lem:equiv-sol-notion}.
\end{proof}
The following result yields probabilistic representations of weak solutions to the Leibenson equation as one-dimensional marginal density curves of solutions to the associated McKean--Vlasov SDE.
\begin{theorem}[Superposition principle]\label{thm:SP-princ-general}
Let $u$ be a weakly continuous probability solution to \eqref{eq:Leibenson-FPE} with initial condition $\nu$ such that
\begin{equation}\label{eq:SP-int-1}
	\int_0^T \int_{\R^d}  |\nabla u|^{p-2} u^{(q-1)(p-1)+1} \mathrm{d}x \mathrm{d}t < \infty,\quad  \forall T>0
\end{equation}
and
\begin{equation}\label{eq:SP-int-2}
	\int_0^T \int_{\R^d} |\nabla \big(  |\nabla u|^{p-2} u^{(q-1)(p-1)}   \big)|u\, \mathrm{d}x\mathrm{d}t < \infty,\quad \forall T>0.
\end{equation}
Then there exists a probabilistically weak solution $(X(t))_{t\geq 0}$ to \eqref{eq:DDSDE} with $\mathcal{L}_{X(t)}(\mathrm{d}x) = u(t,x)\mathrm{d}x$ for all $t > 0$ and $\mathcal{L}_{X(0)}(\mathrm{d}x) = \nu(\mathrm{d}x)$.
\end{theorem}
The following corollary follows immediately from the previous theorem and Lemma \ref{lem:equiv-sol-notion}.
\begin{kor}\label{kor:SP-princ-Leibenson}
	Let $u$ be a weakly continuous weak probability solution to \eqref{eq:Leibenson} with initial condition $\nu$ such that $u(t)\in \Fscr$ for all $t >0$, and \eqref{eq:chain-rule}, \eqref{eq:SP-int-1}+\eqref{eq:SP-int-2} hold. Then there exists a probabilistically weak solution $(X(t))_{t\geq 0}$ to \eqref{eq:DDSDE} with $\mathcal{L}_{X(t)}(\mathrm{d}x) = u(t,x)\mathrm{d}x$ for all $t > 0$ and $\mathcal{L}_{X(0)}(\mathrm{d}x) = \nu(\mathrm{d}x)$.
\end{kor}
\begin{proof}[Proof of Theorem \ref{thm:SP-princ-general}]
Consider $u$ as a solution (in the sense of Definition \ref{dD.1}) to the \emph{linear} FPE with coefficients $\tilde{a}$ and $\tilde{b}$ obtained by a priori fixing $u$ in $a$ and $b$ from \eqref{eq:Leibenson-FPE:coefficients}, i.e.
	\begin{equation}
		\tilde{a}_{ij}(t,x) := q^{p-1}\delta_{ij} |\nabla u(x)|^{p-2} |u(x)|^{(p-1)(q-1)},\quad \tilde{b}_i(t,x):= q^{p-1} \partial_i \big(|\nabla u(x)|^{p-2} |u(x)|^{(p-1)(q-1)}\big).
	\end{equation}
By \cite[Thm.2.5]{Trevisan16} (which applies due to \eqref{eq:SP-int-1}+\eqref{eq:SP-int-2}) one obtains a probabilistically weak solution $X = (X(t))_{t \geq 0}$ to the corresponding (non-distribution dependent) SDE with coefficients $\tilde{b}_i$ as above and $\tilde{\sigma} = (\tilde{\sigma}_{ij})$ given by $\frac 1 2 \tilde{\sigma } \tilde{\sigma}^T = \tilde{a}(t,x)$ such that $\mathcal{L}_{X(t)}(\mathrm{d}x) = u(t,x)\mathrm{d}x$. Clearly, $X$ then solves \eqref{eq:DDSDE}.
\end{proof}

\section{Barenblatt solutions}\label{sect:barenblatt}

The Leibenson equation admits explicit solutions with point source initial condition, known as \emph{Barenblatt solutions}. Three very different regimes arise, subject to $q(p-1)$ being strictly greater, equal, or strictly less than $1$. In this paper, we focus on the first case. This includes the $p$-Laplace case ($q=1$, $p>2$).
\begin{dfn}[Barenblatt solutions]\label{def:Barenblatt-sol}
Let $d \in \N$, $p>1$, $q >0$ such that $q(p-1)>1$. For $y \in \R^d$, the function
	\begin{equation}\label{def-w}
		w^y(t,x) := t^{-\frac d \beta} \bigg[C- \kappa \big( t^{-\frac 1 \beta} |x-y|\big)^{\frac{p}{p-1}}\bigg]_+^\gamma,\quad (t,x)\in (0,\infty)\times \R^d,
	\end{equation}
where 
$$\beta = p + d(q(p-1)-1), \quad \gamma = \frac{p-1}{q(p-1)-1},\quad \kappa = \frac{q(p-1)-1}{pq}\beta^{-\frac{1}{p-1}},$$
and $C>0$ is any constant, is called \emph{Barenblatt solution} to \eqref{eq:Leibenson}. 
\end{dfn}
A straightforward calculation shows that there is a unique choice $C = C(d,p,q)$ independent of $t$ such that $\int_{\R^d}w^y(t,x) \mathrm{d}x = 1$ for all $t>0$. We fix this choice, so that $t \mapsto w^y(t,x)\mathrm{d}x$ is a curve of probability measures.
We often abbreviate
\begin{align*}
    f(t,x) &:=  C- \kappa \big( t^{-\frac 1 \beta} |x|\big)^{\frac{p}{p-1}}.
\end{align*}
Furthermore, we set
\begin{align}\label{eq:def-R}
    R(t) &:= \left(\frac{C}{\kappa}\right)^\frac{p-1}{p} t^\frac{1}{\beta}.
\end{align}
\begin{rem}
	Note that $\beta > p$ and $\gamma,\kappa >0$.
\end{rem}
For the definition of Barenblatt solutions in the cases $q(p-1) = 1$ and $0< q(p-1) < 1$, we refer to \cite{B52_Barenblatt-original,GS24}. In these cases the solutions behave fundamentally different than those in \eqref{def-w}. In particular, while $w^y$ above has \emph{finite} speed of propagation, the Barenblatt solutions in the case $q(p-1) = 1$ and $q(p-1)<1$ have \emph{infinite} speed of propagation. We postpone the study of these case to future work.

From now on, we consider without loss of generality $y=0$ and write $w= w^0$. All subsequent statements hold for any $y \in \R^d$, with obvious modifications where necessary.

\begin{lem}\label{lem:prop-Barenblatt}
$w$ has the following properties.
	\begin{enumerate}[(i)]
		\item 	$w$ is a weakly continuous weak probability solution to \eqref{eq:Leibenson} and attains its initial datum $\delta_0$ weakly in the sense of measures, i.e. $w(t,x)\mathrm{d}x \to \delta_0$ weakly as $t \to 0$.
		\item\label{lem:prop-Barenblatt:ii} For all $t >0$, $w(t) \in C_c(\R^d)$ and $\supp w(t)=\overline{B_{R(t)}(y)}$, with $R(t)$ as in \eqref{eq:def-R}.
		\item $w \in C((0,\infty)\times \R^d) \cap \bigcap_{0 <r < R} C_c([r,R]\times \R^d) \cap \bigcap_{r>0}L^\infty((r,\infty)\times \R^d)$.
	\end{enumerate}
\end{lem}
\begin{proof}
The nonnegativity as well as (ii) and (iii) follow immediately from the definition of $w$. The weak continuity of $(0,\infty)\ni t \mapsto w(t,x)\mathrm{d}x$ follows from the continuity of $w$ in $(t,x)\in (0,\infty)\times \R^d$, and the weak continuity in $t=0$ follows by a straightforward calculation via the transformation $\Phi(x) = t^{\frac 1 \beta}x$ (similarly to the proof that the classical heat kernel converges weakly to $\delta_0$ when $t \to 0$). That $w$ solves \eqref{eq:Leibenson} in the sense that both sides of the equation are well-defined and equal in the sense of Sobolev spaces is a straightforward calculation. In particular, $w$ is a weak solution to \eqref{eq:Leibenson} in the sense of Definition \ref{def:Leibenson-eq}.
\end{proof}

\subsection{Probabilistic representation for Barenblatt solutions}

We apply Corollary \ref{kor:SP-princ-Leibenson} to the Barenblatt solution $w$. More precisely, we have the following result.
\begin{theorem}\label{prop:SP-for-Barenblatt}
	Let $d\in \N$, $p>1$, $q>0$ such that $q(p-1)>1$ and $p > \frac{1+d}{d}$.
	There is a probabilistically weak solution $(X(t))_{t\geq 0}$ to \eqref{eq:DDSDE} such that $\mathcal{L}_{X(t)}(\mathrm{d}x)= w(t,x)\mathrm{d}x$ for all $t \geq 0$ (here and below with a slight abuse of notation we write $w(0,x)\mathrm{d}x = \delta_0(\mathrm{d}x)$).
\end{theorem}
Recall that we write "$\lesssim$" and "$\lesssim_t$" whenever we suppress absolute multiplicative constants, independent or dependent of $t$, respectively, in an inequality.
\begin{proof}
	We show that $w$ satisfies the assumptions of Corollary \ref{kor:SP-princ-Leibenson}.
	That $w$ is a weakly continuous probability solution to \eqref{eq:Leibenson} was already stated in Lemma \ref{lem:prop-Barenblatt}, so it remains to prove $u(t) \in \Fscr$ for all $t>0$, \eqref{eq:SP-int-1}-\eqref{eq:SP-int-2} and \eqref{eq:chain-rule}. To this end, we first prove $w(t)\in W^{1,1}(\R^d)$ for all $t>0$. Note that $w(t) \in L^1(\R^d)$ follows from Lemma \ref{lem:prop-Barenblatt}.  If $\gamma\geq1$, $w(t)\in W^{1,1}(\R^d)$ is obvious, and [recall that we set $f(t,x) = C- \kappa \big( t^{-\frac 1 \beta} |x|\big)^{\frac{p}{p-1}}$]
	\begin{equation}\label{eq:grad-w}
		\nabla w(t,x) = \gamma t^{-\frac d \beta} f(t,x)^{\gamma-1} \nabla f(t,x) \mathds{1}_{ f(t)\geq 0}(x).
	\end{equation}
Since $f$ is locally bounded, $\{f(t) \geq 0\}$ is compact and
\begin{equation}\label{eq:grad-w-2}
	\nabla f(t,x) = -\kappa \frac{p}{p-1}t^{-\frac{p}{\beta(p-1)}}|x|^{\frac{2-p}{p-1}}x,
\end{equation}
we have indeed $\nabla w(t) \in L^1(\R^d,\R^d)$. If $0< \gamma < 1$, the RHS of \eqref{eq:grad-w} is still the weak gradient of $w(t)$, provided we can show $f(t)^{\gamma-1} \nabla f(t) \mathds{1}_{f(t)\geq 0} \in L^1(\R^d,\R^d)$. Since $f(t,x)^{\gamma-1}|\nabla f(t,x)| \lesssim_t |x|^{\frac{1}{p-1}} f(t,x)^{\gamma-1}$, this translates into proving $f(t)^{\gamma-1} \mathds{1}_{ f(t)\geq 0} \in L^1(\R^d)$. Using polar coordinates and the transformation rule, we find
\begin{equation}\label{eq:help1}
	\int_{\R^d}  f(t,x)^{\gamma-1} \mathds{1}_{f(t)\geq 0}(x)\mathrm{d}x \lesssim_t \int_0^C r^{\gamma-1}\bigg(\frac{C-r}{\kappa}\bigg)^{\frac{(d-1)(p-1)-1}{p}} \mathrm{d}r.
\end{equation}
Splitting the latter in integrals over $(0,\frac C 2)$ and $(\frac C 2, C)$, we obtain the conditions
\begin{equation*}
	\gamma -1 >-1 \quad \quad \text{ and } \quad \quad \frac{(d-1)(p-1)-1}{p} > -1
\end{equation*}
to ensure finiteness of the RHS in \eqref{eq:help1}. Since $\gamma >0$ and the second inequality is equivalent to $dp-d>0$ and we assume $p>1$, we obtain $w(t)\in W^{1,1}(\R^d)$ for all $t>0$. 

Next, we show $|\nabla w(t)|^{p-2} w(t)^{(q-1)(p-1)} \in W^{1,1}(\R^d)$ for all $t >0$. From \eqref{eq:grad-w}-\eqref{eq:grad-w-2}, we see
\begin{align*}
	|\nabla w(t,x)|^{p-2} w^{(q-1)(p-1)} \cong t^{   -\frac d \beta (p-2)  -\frac{p(p-2)}{\beta(p-1) }  -   \frac d \beta (q-1)(p-1)  }       \mathds{1}_{f(t)\geq 0} (x)  f(t,x)^{    (\gamma-1)(p-2)  +\gamma(q-1)(p-1)   }   |x|^{\frac {p-2} {p-1}},
\end{align*}
where we recall that by $\cong$ we indicate that we suppress absolute multiplicative constants on the RHS, which are in particular independent of $t$ and $x$. Note that at least one of the exponents $p-2$ and $(q-1)(p-1)$ is strictly positive, thus we can use our convention $0\times \infty = 0$ to obtain the factor $\mathds{1}_{f(t)\geq 0}$ on the RHS above. Indeed, if $p<2$, then $q(p-1)>1$ implies $(q-1)(p-1)>0$. To continue, note that
\begin{equation}\label{eq:eq}
(\gamma-1)(p-2)  +\gamma(q-1)(p-1) = 1,
\end{equation}
since this can easily be seen to be equivalent to
$$(p-1)(\gamma-1 + \gamma(q-1))+1-\gamma = 1 \iff (p-1)(\gamma q- 1) = \gamma,$$
which is true by definition of $\gamma$. Thus
$$\mathds{1}_{f(t)\geq 0}  f(t)^{    (\gamma-1)(p-2)  +\gamma(q-1)(p-1)   } = f(t)_+ \in (W^{1,1}\cap W^{1,\infty})(\R^d).$$ 
Since $p > \frac{d+1}{d}> \frac{d+2}{d+1}$, $x \mapsto |x|^{\frac{p-2}{p-1}}$ is in $W^{1,1}_{\textup{loc}}(\R^d)$, hence we obtain that $|\nabla w(t)|^{p-2} w(t)^{(q-1)(p-1)}$ is weakly differentiable, and, by the product rule, for a.e. $x$
\begin{align}\label{1a}
	|\nabla \big(|\nabla &w(t,x)|^{p-2} w(t,x)^{(q-1)(p-1)} \big)|
	\lesssim  t^{   -\frac d \beta (p-2)  -\frac{p(p-2)}{\beta(p-1) }  -   \frac d \beta (q-1)(p-1)  }\mathds{1}_{f(t)\geq 0} \big[ t^{-\frac{p}{\beta(p-1)}}|x|^{\frac{1}{p-1}} + f(t,x)_+ |x|^{-\frac{1}{p-1}  }   \big].
\end{align}
To obtain that this term is in $L^1(\R^d,\R^d)$ for arbitrary fixed $t>0$, it suffices to note that $|x|^{-\frac{p}{p-1}}x \in L^1_{\textup{loc}}(\R^d,\R^d)$, which is true since $p > \frac{d+1}{d}$. Hence $u(t) \in \Fscr$ for all $t >0$.

Regarding \eqref{eq:chain-rule}, we have
$
	u^q(t,x) =  t^{-\frac {qd} \beta}f^{\gamma q}_+(t,x),
$
with $\gamma q >1$ by definition of $\gamma$, and hence one can simply calculate $\nabla u^q$ by the chain rule to deduce \eqref{eq:chain-rule}.

Regarding \eqref{eq:SP-int-1} we see that, by \eqref{eq:grad-w}-\eqref{eq:grad-w-2}
\begin{align}\label{aaa}
\notag \int_0^T \int_{\R^d} & |\nabla w(t,x)|^{p-2}w(t,x)^{(q-1)(p-1)+1} \, \mathrm{d}x \mathrm{d}t \\&
\notag
\lesssim \int_0^T \int_{B_{R(t)(0)}
}        t^{-(p-2)(\frac d \beta + \frac{p}{\beta(p-1)})     -    \frac d \beta [(q-1)(p-1)+1]}      f(t,x)^{1+\gamma} |x|^{\frac{p-2}{p-1}} \, \mathrm{d}x \mathrm{d}t \\&
 \lesssim \int_0^T t^{-(p-2)(\frac d \beta + \frac{p}{\beta(p-1)})     -    \frac d \beta [(q-1)(p-1)+1]   } \int_{B_{R(t)}(0)}
   |x|^{\frac{p-2}{p-1}}   \, \mathrm{d}x \mathrm{d}t,
\end{align}
where we used
\begin{equation*}\label{aa}
	(\gamma-1)(p-2)+\gamma(q-1)(p-1)+\gamma = 1+\gamma
\end{equation*}
(see \eqref{eq:eq}) and $1+\gamma >0$.
Using polar coordinates and $p > \frac{d+2}{d+1}$, we find
\begin{equation*}
	\int_{B_{R(t)}(0)}
      |x|^{\frac{p-2}{p-1}}   \, \mathrm{d}x \lesssim \int_0^{R(t)}
     r^{\frac{p-2}{p-1}+d-1} \mathrm{d}r \lesssim t^{\frac{p-2}{\beta(p-1)} + \frac d \beta }.
\end{equation*}
Thus the RHS of \eqref{aaa} is further estimated from above by
\begin{equation*}
	\int_0^T   t^{-(p-2)(\frac d \beta + \frac{p}{\beta(p-1)})     -    \frac d \beta [(q-1)(p-1)] +\frac{p-2}{\beta(p-1)}   } \mathrm{d}t
\end{equation*}
i.e. to conclude its finiteness we need
\begin{equation*}
	-(p-2)\left(\frac d \beta + \frac{p}{\beta(p-1)}\right)     -    \frac d \beta [(q-1)(p-1)] +\frac{p-2}{\beta(p-1)}   > - 1.
\end{equation*}
By a direct calculation, this inequality can be seen to be equivalent to $p-2 < p$ and is hence true. This concludes the proof of \eqref{eq:SP-int-1}.

Finally, regarding \eqref{eq:SP-int-2}, we consider \eqref{1a} to estimate
\begin{align*}
\int_0^T \int_{\R^d}& |\nabla \big(  |\nabla w|^{p-2} w^{(q-1)(p-1)}   \big) | w\, \mathrm{d}x\mathrm{d}t
\\&
 \lesssim \int_0^T \int_{B_{R(t)}(0)
 } t^{   -\frac d \beta (p-1)  -\frac{p(p-2)}{\beta(p-1) }  -   \frac d \beta (q-1)(p-1)  }   \bigg[ t^{-\frac{p}{\beta(p-1)}}|x|^{\frac{1}{p-1}} + |x|^{-\frac{1}{p-1}  }   \bigg]\, \mathrm{d}x\mathrm{d}t =: I + II.
\end{align*}
We further estimate, using again $p > \frac{d+1}{d}$ to have $|x|^{-\frac 1 {p-1}} \in L^1_{\textup{loc}}(\R^d)$,
\begin{align*}
	I \lesssim \int_0^T  t^{   -\frac d \beta (p-2)  -\frac{p(p-1) -1 }{\beta(p-1) }  -   \frac d \beta (q-1)(p-1)  } \mathrm{d}t
\end{align*}
and
$$ II \lesssim \int_0^T   t^{   -\frac d \beta (p-2)  -\frac{p(p-2)  +1  }{\beta(p-1) }  -   \frac d \beta (q-1)(p-1)  } \mathrm{d}t.$$
Elementary calculations show $$\min\bigg(   -\frac d \beta (p-2)  -\frac{p(p-1) -1 }{\beta(p-1) }  -   \frac d \beta (q-1)(p-1)   ,    -\frac d \beta (p-2)  -\frac{p(p-2)  +1  }{\beta(p-1) }  -   \frac d \beta (q-1)(p-1) \bigg ) > -1,$$
whereby \eqref{eq:SP-int-2} follows and the proof is complete.
\end{proof}

Since the computations from the previous proof obviously remain valid for $w^y$ instead of $w$, we obtain:
\begin{kor}\label{cor:SP-princ-Barenblatt-all-y}
	Let $d\in \N$, $p>1$, $q>0$ such that $q(p-1)>1$, $p > \frac{1+d}{d}$ and $y \in \R^d$. There is a probabilistically weak solution $X^y = (X^y(t))_{t\geq 0}$ to \eqref{eq:DDSDE} such that $\mathcal{L}_{X^y(t)}(\mathrm{d}x) = w^y(t,x)\mathrm{d}x$ for all $t \geq 0$ (where again with a slight abuse of notation we write $w^y(0,x)\mathrm{d}x = \delta_y(\mathrm{d}x)$).
\end{kor}
For later use, we state the following remark.
\begin{rem}
Definitions \ref{def:Leibenson-eq}, \ref{def:Leibenson-FP} and \ref{def:DDSDE} can be modified to initial times $s > 0$ instead of $s=0$ in the obvious way, and all previously stated results remain valid accordingly. In particular, Corollary \ref{cor:SP-princ-Barenblatt-all-y} generalizes in the sense that for all initial pairs $(s,y)$, there is a probabilistically weak solution $X^{s,y} = (X^{s,y}(t))_{t\geq s}$ to \eqref{eq:DDSDE} such that $\mathcal{L}_{X^{s,y}(t)} = w^y(t-s,x)\mathrm{d}x$ for all $t \geq s$. 
\end{rem}

\section{Associated nonlinear Markov process}\label{sect:Leibenson-Markov}
Our next goal is to construct a nonlinear Markov process consisting of the path laws of the solutions $X^{s,y}$ from the previous remark, and to show that this Markov process is uniquely determined by the Barenblatt solutions to \eqref{eq:Leibenson}. First, we briefly recall the linear case, i.e. $(p,q) = (2,1)$.
\subsection{The classical linear case}\label{sect:class-case}
In the linear case, i.e. $p=2$, $q=1$, when \eqref{eq:Leibenson} is the heat equation, its fundamental solutions given by the classical heat kernel $p(t,x,y)$, $t >0 , x,y \in \R^d$, uniquely determine a Markov process $(P_y)_{y \in \R^d} \subseteq \mathcal{P}(C([0,\infty),\R^d)$ with one-dimensional time marginals $P_y \circ \pi_t^{-1} = p(t,x,y)\mathrm{d}x$ (for the definition of a Markov process, see below). We recall that here and below $\pi_t: C(\R_+,\R^d) \to \R^d$, $\pi_t(w) = w(t)$, denotes the canonical projection at time $t \geq 0$. This Markov process is Brownian motion, $P_0$ is the \emph{standard Wiener measure}, and any stochastic process $(X^y(t))_{t\geq 0}$ with path law $P_y$ is a Brownian motion started in $y \in \R^d$. Of course, $P_y$ is the unique solution law of the associated SDE
$$\mathrm{d}X^y(t) = \mathrm{d}W(t), \quad t \geq 0, \quad X^y(0) = y.$$

\subsection{Nonlinear Markov processes}\label{subsect:NL-MP}
	The following definition of nonlinear Markov processes is inspired by McKean \cite{McKean1-classical} and was first studied in \cite{R./Rckner_NL-Markov22}. It extends the usual notion of Markov property in a suitable way in order to cover path laws of McKean--Vlasov solutions. We refer to \cite{R./Rckner_NL-Markov22} for more details.
	
	We use the following notation. For $0 \leq s \leq t$, we denote by $\pi^s_t$ the projection $\pi^s_t : C([s,\infty),\R^d) \to \R^d$, $\pi^s_t(w) := w(t)$, and we set $\Fscr_{s,t} := \sigma(\pi^s_u, s\leq u \leq t)$.
\begin{dfn}\label{d5.1}
	Let $\mathcal{P}_0 \subseteq \mathcal{P}$. A \textit{nonlinear Markov process} is a family $(P_{s,\zeta})_{(s,\zeta)\in[0,\infty)\times \mathcal{P}_0}$, $P_{s,\zeta} \in \mathcal{P}(C([s,\infty),\R^d))$, such that for all $0\leq s \leq r \leq t, \zeta \in \mathcal{P}_0$
	\begin{enumerate}
		\item[(i)] $\mu^{s,\zeta}_t := P_{s,\zeta}\circ (\pi^{s}_{t})^{-1} \in \mathcal{P}_0$,
		\item[(ii)] the \textit{nonlinear Markov property} holds, i.e.
		\begin{equation}\label{MP}
			P_{s,\zeta}(\pi^{s}_{t} \in A|\mathcal{F}_{s,r})(\cdot) = p_{(s,\zeta),(r,\pi^s_r(\cdot))}(\pi^r_t\in A) \quad P_{s,\zeta}\text{-a.s.} \text{ for all }A \in \mathcal{B}(\R^d),
		\end{equation}
		where $(p_{(s,\zeta),(r,z)})_{z\in \R^d}$ is a regular conditional probability kernel from $\R^d$ to $\mathcal{B}(C([r,\infty),\R^d))$ of ${P}_{r,\mu^{s,\zeta}_r}[\,\cdot\,| \pi^r_r{=}z],\, z \!\in\! \R^d$ (i.e. in particular $p_{(s,\zeta),(r,z)}\! \in\! \mathcal{P}(C([r,\infty),\R^d))$ and $p_{(s,\zeta),(r,z)}(\pi^r_r {=} z) = 1$).
	\end{enumerate}
\end{dfn}
The case of a classical (time-inhomogeneous) Markov process $(P_{s,y})_{y\in \R^d}$ is contained in the previous definition. In this case, $(P_{s,\zeta})_{s\geq 0, \zeta \in \mathcal{P}}$, where $P_{s,\zeta}:= \int_{\R^d} P_{s,y}d\zeta(y)$, satisfies the previous definition with 
$\mathcal{P}_0=\mathcal{P}$
and $p_{(s,y)(r,z)}=P_{r,z},$ for all $0\le s\le r,\ z\in\R^d,\ \zeta\in\mathcal{P}$. Then \eqref{MP} reduces to the usual time-inhomogeneous  Markov property, and the {\it Chapman--Kolmogorov equations} for the one-dimensional time marginals $\mu^{s,y}_t = P_{s,y}\circ( \pi^s_t)^{-1}$ hold, i.e.
\begin{equation}
	\label{e5.2'}
	\mu^{s,y}_t=\int_{\R^d} \mu^{r,z}_t \mathrm{d}\mu^{s,y}_r(z),\ \forall 0\le s\le r\le t,\ y\in\R^d.
\end{equation}
In the nonlinear case the map $\zeta\to P_{s,\zeta}$ is not linear on $\mathcal{P}_0$, even if $\mathcal{P}_0=\mathcal{P}$ (which we do not assume). Thus one loses the Chapman--Kolmogorov equations, but the one-dimensional time marginals of a nonlinear Markov process still satisfy the {\it flow property} 
\begin{equation}
	\label{e5.1'}
	\mu^{s,\zeta}_t \in \Pscr_0, \quad \mu^{s,\zeta}_t=\mu^{r,\mu^{s,\zeta}_r}_t,\  \forall 0\le s\le r\le t,\ \zeta\in\mathcal{P}_0,\end{equation} (in the linear case, this follows from \eqref{e5.2'}). A family $\{\mu^{s,\zeta}\}_{s\geq 0, \zeta \in \Pscr_0}$, $\mu^{s,\zeta} = (\mu^{s,\zeta}_t)_{t\geq s}$ of weakly continuous solutions to the nonlinear FPE \eqref{e1.3} in the sense of Definition \ref{dD.2} with initial condition $\mu^{s,\zeta}_s = \zeta$ and satisfying \eqref{e5.1'} is called a $\Pscr_0$-\emph{solution flow \eqref{e1.3}}.

The following result from \cite{R./Rckner_NL-Markov22}, which allows to construct nonlinear Markov processes consisting of solution laws to McKean--Vlasov SDEs with one-dimensional time marginals given by a prescribed family of solution curves to the associated nonlinear FPE, is the key for our purposes in this chapter.
\begin{theorem}{\rm\cite[Theorems 3.4+3.8]{R./Rckner_NL-Markov22}}\label{t52}
	Let $\Pscr_0 \subseteq \Pscr$ and $\{\mu^{s,\zeta}\}_{s\geq 0, \zeta \in \Pscr_0}$ be a $\Pscr_0$-solution flow to \eqref{e1.3} such that the following hypothesis holds.

    \textbf{($\Pscr_0$-$\text{lin}_{\text{ex}}$).} $\mu^{s,\zeta}$ is an extreme point of the set of all weakly continuous probability solutions to the linear FPE with coefficients $(t,x) \mapsto a_{ij}(t,x,\mu^{s,\zeta}_t)$ and $(t,x) \mapsto b_i(t,x,\mu^{s,\zeta}_t)$ with initial condition $(s,\zeta)$ in the sense of Definition \ref{dD.1} for each $(s,\zeta) \in \R_+ \times \Pscr_0$.
\\

	Then, for each $(s,\zeta) \in \R_+\times \Pscr_0$ there is a unique solution law $P_{s,\zeta}$ to the McKean--Vlasov SDE \eqref{e1.5} such that
	\begin{equation}\label{tg}
		P_{s,\zeta}\circ (\pi^s_t)^{-1} = \mu^{s,\zeta}_t, \quad \forall 0\leq s \leq t, \zeta \in \Pscr_0,
	\end{equation}
	and $(P_{s,\zeta})_{s\geq 0, \zeta \in \Pscr_0}$ is a nonlinear Markov process in the sense of Definition {\rm\ref{d5.1}}.
	In particular, this nonlinear Markov process is uniquely determined by its one-dimensional time marginals $(\mu^{s,\zeta}_t)_{0\leq s \leq t, \zeta \in \Pscr_0}$ and equation \eqref{e1.5}.
\end{theorem}

\begin{rem}
     A relation of extremality and Markov property was discovered in connection
 with constructing Markov selections (see \cite{K73} and also \cite{SV-book,FR08})
 in the situation where a (linear) martingale problem has more than one solution
 for Dirac measures as intial conditions (see \cite{Stroock1980}). But this is a result
 on path space and in the opposite direction, namely proving that the Markov
 selection consists of extremal measures in the convex set of all solutions
 to the (linear) martingale problem with initial condition a Dirac measure.
 So, our result above is a converse result, but for the time marginals, i.e.
 on state space, and which also holds in the nonlinear case.
\end{rem}

Regarding applications, the equivalence of the above extremality condition and the following restricted linearized uniqueness result has turned out very useful.
\begin{lem}[{\cite[Lemma 3.5]{R./Rckner_NL-Markov22}}]\label{lem:equiv}
    Let $(s, \zeta) \in \R_+ \times \Pscr_0$. In the situation of the previous theorem, $\mu^{s,\zeta}$ satisfies the claimed extremality condition if and only if
    $(\mu^{s,\zeta}_t)_{t\geq s}$ is the unique weakly continuous probability solution to the linear FPE with coefficients $(t,x)\mapsto a_{ij}(t,x,\mu^{s,\zeta}_t)$ and $(t,x)\mapsto b_i(t,x,\mu^{s,\zeta}_t)$ and initial condition $(s,\zeta)$ in the sense of Definition {\rm\ref{dD.1}} in the class
		\begin{equation*}\label{e5.3a}
			\big\{(\nu_t)_{t\geq s} \subseteq\mathcal{P} : \nu_t \leq C \mu^{s,\zeta}_t, \quad t\geq s,\text{ for some }C>0\big\}
		\end{equation*}
		("restricted linearized distributional uniqueness").
\end{lem}
In fact, we need the following corollary of the previous theorem. The intuition is that the assertion of the previous theorem remains true when condition ($\Pscr_0$-$\textup{lin}_{\textup{ex}}$) holds for less initial data, at the price of a uniqueness assertion for the corresponding solution path laws for less initial data. However, as we will see, in our situation, a posteriori we can restore the uniqueness assertion to all initial data.
\begin{kor} {\rm\cite[Corollary 3.10]{R./Rckner_NL-Markov22}} \label{c5.3}
	Let $\mathfrak{P}_0 \subseteq \Pscr_0 \subseteq \Pscr$ and suppose $\{\mu^{s,\zeta}\}_{s\geq 0, \zeta \in \Pscr_0}$ is a $\Pscr_0$-solution flow to \eqref{e1.3} such that the following two hypotheses hold.
    
     \textbf{($\mathfrak{P}_0$-$\text{lin}_{\text{ex}}$).} $\mu^{s,\zeta}$ is an extreme point in the set of all weakly continuous probability solutions to the linear FPE with coefficients $(t,x) \mapsto a_{ij}(t,x,\mu^{s,\zeta}_t)$ and $(t,x) \mapsto b_i(t,x,\mu^{s,\zeta}_t)$ with initial condition $(s,\zeta)$ in the sense of Definition \ref{dD.1} for each $(s,\zeta) \in \R_+ \times \mathfrak{P}_0$

     and

    \textbf{($\mathfrak{P}_0$-smoothing).} $\mu^{s,\zeta}_t \in \mathfrak{P}_0$ for all $0\leq s < t$, $\zeta \in \Pscr_0$.
    \\
    
    Then, there is a nonlinear Markov process $(P_{s,\zeta})_{s\geq 0, \zeta \in \Pscr_0} $ with \eqref{tg}, consisting of solution path laws to \eqref{e1.5}. In this case, the uniqueness-assertion for $P_{s,\zeta}$ of Theorem {\rm\ref{t52}} holds for all $(s,\zeta) \in \R_+ \times \mathfrak{P}_0$.
\end{kor}
In the next subsection we apply this corollary to the NLFPE \eqref{eq:Leibenson-FPE} and the associated McKean--Vlasov SDE \eqref{eq:DDSDE}.

\subsection{Application to Leibenson equation and its Barenblatt solutions}
Let
$$\mathcal{P}_0 := \{w^y(\delta,x)\mathrm{d}x : y \in \R^d, \delta \geq 0\},$$
where with some abuse of notation we write $w^y(0,x)\mathrm{d}x = \delta_y(\mathrm{d}x)$. For each $\zeta \in \mathcal{P}_0$, the pair  $(\delta,y)\in [0,\infty)\times \R^d$ such that  $\zeta = w^y(\delta,x)\mathrm{d}x$ is unique. 
	\begin{theorem}\label{t5.2}
    Let $d\geq 2, p>\frac{d}{d-1},q>0$ such that $q(p-1)>1$. If $p<2$, assume additionally that $q(p-1)>\frac{2-p+d}{d}(>1)$.
	\begin{enumerate}
		\item [\rm(i)] Let $(s,\zeta) \in [0,\infty)\times \mathcal{P}_0$, $\zeta = w^y(\delta,x)\mathrm{d}x$. The set of solution laws to the McKean--Vlasov SDE  \eqref{eq:DDSDE} with one-dimensional time marginals $w^y(\delta+t-s,x)\mathrm{d}x$, $t \geq s$, and initial condition $(s,\zeta)$ contains exactly one element $P_{s,\zeta}$. The family $(P_{s,\zeta})_{s\geq 0, \zeta \in \mathcal{P}_0}$ is a nonlinear Markov process in the sense of Definition {\rm \ref{d5.1}}.   In particular, this nonlinear Markov process is uniquely determined by \eqref{eq:DDSDE} and $\omega^y(t)$, $y\in\R^d$, $t\ge0$.
		\item[\rm(ii)] $(P_{s,\zeta})_{s \geq 0, \zeta \in \mathcal{P}_0}$ is time-homogeneous, i.e. $P_{s,\zeta} =
		P_{0,\zeta}\circ \hat{\Pi}_s^{-1}$ for all $(s,\zeta)\in[0,\infty)\times \mathcal{P}_0$, where 
		\begin{equation}\label{e5.8'}
			\hat{\Pi}_s: C([0,\infty),\R^d) \to C([s,\infty),\R^d),\quad \hat{\Pi}_s: (\omega(t))_{t\geq 0} \mapsto (\omega(t-s))_{t\geq s}.\end{equation} Moreover, for $\zeta = w^y(\delta,x)\mathrm{d}x$, we have $P_{0,\zeta} = P_{0,y}\circ ({\Pi}_{\delta}^{0})^{-1} $ $($the map $ \Pi^0_{\delta}$ is defined in \eqref{e5.3} below$)$.
	\end{enumerate}
\end{theorem}
\begin{rem}
The $p$-Laplace-case of this result, i.e. $p>2$, $q=1$, was proven in \cite{R.BR24-pLapl}. 
\end{rem}

We stress that in the following proof the times $s$ and $\delta$ are not related. In particular, for the initial condition $\zeta = w^y(\delta,x)\mathrm{d}x$, we not only consider the initial pair $(\delta,\zeta)$, but necessarily \emph{any} $(s,\zeta)$, $s \geq 0$. 	Set 
\begin{equation}\label{def:mathfrakP}
    \mathfrak{P}_0 := \{w^y(\delta,x)\mathrm{d}x : y \in \R^d, \delta>0\}=\mathcal{P}_0\setminus\{\delta_y:\ y\in\R^d\}.
\end{equation} 	
A crucial ingredient of the proof is Theorem \ref{restr-lin-u.theorem}, which is formulated and proven in Section \ref{sect:restr-lin-u} below.

	\medskip \noindent{\it Proof of Theorem}   \ref{t5.2}.
\begin{enumerate}
	\item [(i)] Setting, for $\zeta = w^y(\delta,x)\mathrm{d}x$, $\delta\geq 0$, $y\in\R^d$,
	$$\mu^{s,\zeta}_t := w^y(\delta+t-s,x)\mathrm{d}x, \quad t \geq s,$$
	it is straightforward to check that the family of probability measures $\{\mu^{s,\zeta}_t\}_{s\geq 0, t\geq s, \zeta \in \mathcal{P}_0}$ has the {flow property} \eqref{e5.1'}.    
	From Lemmas \ref{lem:prop-Barenblatt} and \ref{lem:equiv-sol-notion} and from what was shown in the proof of Theorem \ref{prop:SP-for-Barenblatt} it follows that $(\mu^{s,\zeta}_t)_{t\geq s}$ is a weakly continuous probability solution to the nonlinear FPE \eqref{eq:Leibenson-FPE} with initial condition $(s,\zeta)$ in the sense of Definition \ref{dD.2}.
	By Theorem \ref{restr-lin-u.theorem}, Lemma \ref{lem:equiv} and Remark \ref{restr-lin-u.remark:qRange} below, condition ($\mathfrak{P}_0$-$\textup{lin}_{\textup{ex}}$) holds. Moreover, by definition of $\mathcal{P}_0$ and $\mathfrak{P}_0$, we have $\mu^{s,\zeta}_t \in \mathfrak{P}_0$ for all $(s,t,\zeta)$ such that either $s \leq t$ and $\zeta \in \mathfrak{P}_0$ or $s < t$ and $\zeta \in \mathcal{P}_0$. Thus, Corollary \ref{c5.3} applies and yields:
	
	There is a nonlinear Markov process $(P_{s,\zeta})_{(s,\zeta)\in [0,\infty)\times \mathcal{P}_0}$, $P_{s,\zeta} \in \mathcal{P}(C([s,\infty),\R^d))$, such that
	\begin{enumerate}
		\item [(I)] $P_{s,\zeta}\circ (\pi^s_t)^{-1} = \mu^{s,\zeta}_t, t\geq s$;
		\item [(II)] $P_{s,\zeta}$ is a solution path law to the McKean--Vlasov SDE \eqref{eq:DDSDE} on $[s,\infty)$;
		\item [(III)] For $s\geq 0$ and $\zeta \in \mathfrak{P}_0$, $P_{s,\zeta}$ is \emph{unique} with properties (I)--(II).
	\end{enumerate}
	Therefore, since $\mathcal{P}_0 \backslash \mathfrak{P}_0 = \{\delta_y: y \in \R^d\}$, it remains to prove the following claim.
	
	\medskip\noindent{\bf Claim.} For $(s,y)\in [0,\infty)\times \R^d$, there is a \emph{unique} path law $P_{s,\delta_y}$ with properties (I)--(II).
	
	\textit{Proof of Claim.} Let $P^1,P^2$ have properties (I)--(II) for $s\geq 0, \zeta = \delta_y$, $y \in \R^d$. For $s,r \geq 0$, we define the map $\Pi^s_r : C([s,\infty),\R^d) \to C([s,\infty),\R^d)$ via
	\begin{equation}\label{e5.3}
		\Pi^s_r: \omega(t)_{t\geq s} \mapsto \omega(t+r)_{t\geq s}.
	\end{equation}			
	For any $r >0$, $i \in \{1,2\}$, we have
	\begin{equation}\label{e5.9'}
		(P^i\circ (\Pi^s_r)^{-1})\circ (\pi^s_t)^{-1}=P^i \circ (\pi^s_{t+r})^{-1} = \mu^{s,\delta_y}_{t+r} = w^y(t+r-s,x)\mathrm{d}x,\quad \forall t \geq s.
	\end{equation}
	It is straightforward to check that $P^i\circ (\Pi^s_r)^{-1}  \!\in\! \mathcal{P}(C([s,\infty),\R^d))$, \mbox{$i \!\in\! \{1,2\}$,} is a solution law to \eqref{eq:DDSDE} with initial condition $(s,w^y(r,x)\mathrm{d}x)$.
	As $w^y(r,x)\mathrm{d}x \in \mathfrak{P}_0$, \eqref{e5.9'} and (III) yield	$$P^1\circ (\Pi^s_r)^{-1}=P^2\circ (\Pi^s_r)^{-1}.$$
	Now let $s \leq u_1 < \dots < u_n$, $n \in \N$. First assume $u_1 >s$. Then for $i \in \{1,2\}$
	\begin{align*}
		P^i\circ (\pi^s_{u_1},...,\pi^s_{u_n})^{-1}
		=(P^i\circ (\Pi^s_{u_1-s})^{-1}) \circ (\pi^s_s,...,\pi^s_{u_n+s-u_1})^{-1}
	\end{align*}
	and by the previous part of the proof the right hand side  coincides for $i=1$ and $i=2$, since $u_1-s>0$. Now assume $s=u_1 <\dots < u_n$. Then, since $P^i\circ (\pi^s_s)^{-1} = \delta_y$, we find
	\begin{align*}
		P^i\circ (\pi^s_{u_1},...,\pi^s_{u_n})^{-1}
		=\delta_y\otimes
		(P^i\circ (\pi^s_{u_2},...,\pi^s_{u_n})^{-1})	
	\end{align*}
	(where $\mu \otimes \nu$ denotes the product measure of the measures $\mu$ and $\nu$).
	Since $u_2>s$, the argument of the first case again yields that the right hand side  coincides for $i=1$ and $i=2$. Hence we have proven
	$$P^1 \circ (\pi^s_{u_1},...,\pi^s_{u_n})^{-1}
	=P^2 \circ (\pi^s_{u_1},...,\pi^s_{u_n})^{-1}$$
	for all $s\leq u_1 < \dots < u_n$, $n \in \N$, i.e. $P^1 = P^2$, which proves the claim and, thereby, the assertion.
	\item[(ii)] First note 
$$P_{s,\zeta} \circ (\pi^s_t)^{-1} = \mu^{s,\zeta}_t = \mu^{0,\zeta}_{t-s} = (P_{0,\zeta} \circ (\hat\Pi_s)^{-1}) \circ (\pi^s_t)^{-1},\   \forall t \geq s,$$ with $\hat\Pi_s$ as in \eqref{e5.8'}.  
Both $P_{s,\zeta}$ and $P_{0,\zeta}\circ (\hat\Pi_s)^{-1}$ are solution laws to \eqref{eq:DDSDE} on $[s,\infty)$ with initial condition $(s,\zeta)$, hence, by (i),they coincide.

For the final statement, note that for $\zeta = w^y(\delta,x)\mathrm{d}x$ the measures $P_{0,\zeta}$ and $P_{0,y} \circ (\Pi^0_{\delta})^{-1}$  have identical one-dimensional time marginals $w^y(\delta+t,x)\mathrm{d}x,$ $ t \geq 0$, and both are solution laws to \eqref{eq:DDSDE} with initial condition $(0,w^y(\delta,x)\mathrm{d}x)$. Thus the assertion follows from the uniqueness assertion in (i).\hfill$\Box$
\end{enumerate}
	\begin{rem}\label{r5.4}
	Theorem \ref{t5.2} (ii) implies that the nonlinear Markov process $(P_{s,\zeta})_{(s,\zeta) \in \R_+\times \mathcal{P}_0}$ from (i) is uniquely determined by $$(P_y)_{y \in \R^d}, \ P_y := P_{0,\delta_y}.$$ 
	Therefore, we also refer to $(P_y)_{y\in \R^d}$ as the \emph{unique  nonlinear Markov process	 determined by \eqref{eq:DDSDE} and the one-dimensional time marginals $w^y(t), y \in \R^d, t \geq 0$}.
	Note that $P_y$ is the path law of the solution $X^y$ from Corollary \ref{cor:SP-princ-Barenblatt-all-y}.
\end{rem}
Finally, we arrive at the following definition.
\begin{dfn}\label{def:Leibenson-process}
	Let $d\geq 2, p>\frac{d}{d-1},q>0$ such that $q(p-1)>1$. If $p<2$, assume additionally that $q(p-1)>\frac{2-p+d}{d}(>1)$. We call the nonlinear Markov process $(P_y)_{y\in \R^d}$ from the previous remark the \emph{Leibenson process}.
\end{dfn}
In particular, the Barenblatt solutions to \eqref{eq:Leibenson} have a probabilistic representation as the one-dimensional time marginal densities of the Leibenson process, which they uniquely determine as a nonlinear Markov process. The construction of this process and the derivation of the corresponding stochastic equation, i.e. \eqref{eq:DDSDE}, is completely analogous to the linear special case $(p,q) = (2,1)$ of the heat equation and Brownian motion, compare Section \ref{sect:class-case}.

The special case $d \geq 2$, $p >2$, $q=1$ was studied in \cite{R./Rckner_NL-Markov22}; therein the corresponding nonlinear Markov process was called \emph{$p$-Brownian motion}.

\section{The Leibenson process is a functional of Brownian motion}\label{sect:Strongsolution}
In Section \ref{sect:Leibenson-Markov} we constructed the Leibenson process, which consists of path laws of probabilistically weak solutions $(X,W)$ to the McKean--Vlasov SDE \eqref{eq:DDSDE} with $u(t)=w^y(t+\delta), t>0, \delta\geq 0$.
In this section, we prove that for $\delta>0$ these solutions are in fact probabilistically strong solutions, that is, $X$ is a measurable adapted functional of the driving Brownian motion $W$ and its initial condition $X(0)$. For simplicity, we only consider the case $y=0$. The general case $y\in \rd$ can be proved analogously.
\\

First, we reduce the question of whether there exists a probabilistically strong solution to \eqref{eq:Leibenson-MVSDE} to a problem for an ordinary SDE.
For $\delta>0$, we set for $(t,x)\in \R_+\times \rd$
$$w_\delta(t,x) := w^0(t+\delta,x),\quad \vrho_\delta(t,x):= |\nabla w_\delta(t,x)|^{p-2} w_\delta(t,x)^{(p-1)(q-1)},\quad R_\delta(t):=R(t+\delta),$$
where again $w^0$ denotes the Barenblatt solution from \eqref{def-w}, and $R$ is defined as in \eqref{eq:def-R}.
As computed in the proof of Theorem \ref{prop:SP-for-Barenblatt}, we have
\begin{align}\label{StrongSolution.vrho_delta.function}
	\vrho_\delta(t,x) = C_{\vrho}(t+\delta) f_+(t+\delta,x)  |x|^{\frac {p-2} {p-1}},\ \ (t,x) \in \R_+\times\rd,
\end{align}
where we set
\begin{align}\label{StrongSolution.vrho_delta.function:constant}
C_{\vrho}(t+\delta):=\left(\frac{\gamma \kappa p }{p-1}\right)^{p-2} (t+\delta)^{-\frac{d(p-2)}{\beta}-\frac{p(p-2)}{\beta(p-1)}-\frac{d(p-1)(q-1)}{\beta}},\ \ (t,x) \in \R_+\times\rd.
\end{align}
Consider the ordinary SDE
\begin{align}\label{StrongSolution.eq:DDSDE-lin}\begin{cases}
	\mathrm{d}X(t) &= q^{p-1}\nabla \vrho_\delta(t,X(t))\mathrm{d}t   + \sqrt{2q^{p-1}\vrho_\delta(t,X(t))} \mathrm{d}W(t),	\\\mathcal{L}_{X(t)}(\mathrm{d}x) &= w_\delta(t,x)\mathrm{d}x.
    \end{cases}
\end{align}
Note that \eqref{StrongSolution.eq:DDSDE-lin} is \textit{equivalent} to \eqref{eq:DDSDE} and ${\mathcal{L}_{X(t)}(\mathrm{d}x) = w_\delta(t,x)\mathrm{d}x}$; it is obtained by fixing $u = w_\delta$ in the coefficients in \eqref{eq:DDSDE}.

Now we define when a probabilistically weak solution to \eqref{eq:DDSDE} (this includes \eqref{StrongSolution.eq:DDSDE-lin}) is said to be strong.
We introduce the following notation.
Let $t\in [0,\infty)$. We set $\mathcal{B}_t(C([0,\infty);\rd)):= \sigma(\pi_s : s\in [0,t])$ and, correspondingly, $\mathcal{B}_t(C([0,\infty);\rd)_0):= \sigma(\pi_s: s\in [0,t])\cap C([0,\infty);\rd)_0$. Furthermore, $P^W$ denotes the Wiener measure on $(C([0,\infty);\rd)_0, \mathcal{B}(C([0,\infty);\rd)_0)$.

\begin{dfn}\label{def:DDSDE_strong}
A probabilistically weak solution $(X,W,(\Omega,\Fscr,(\Fscr_t)_{t\geq 0},\PP))$ to \eqref{eq:Leibenson-MVSDE} with initial condition $\nu\in \Pscr(\rd)$ is called a \emph{probabilistically strong solution}, if there exists
$$F_{\nu} : \rd\times C([0,\infty);\rd)_0 \to C([0,\infty);\rd),$$
which is $\overline{\mathcal{B}(\rd)\otimes \mathcal{B}(C([0,\infty);\rd)_0)}^{{\nu}\otimes P^W}\slash\mathcal{B}(C([0,\infty);\rd))$-measurable
 such that for $\nu$-a.e. $x\in \rd$, $F_{\nu}(x,\cdot)$ is $\overline{\mathcal{B}_t(C([0,\infty);\rd)_0)}^{P^W}\slash \mathcal{B}_t(C([0,\infty);\rd))$-measurable for every $t\in[0,\infty)$, and
 $$X=F_{\nu}(X(0),W)\ \PP\text{-a.s.}$$
	Here, 
    $\overline{\mathcal{B}(\rd)\otimes \mathcal{B}(C([0,\infty);\rd)_0)}^{{\nu}\otimes P^W}$ denotes the completion of $\mathcal{B}(\rd)\otimes \mathcal{B}(C([0,\infty);\rd)_0)$ with respect to ${\nu}\otimes P^W$, and
    $\overline{\mathcal{B}_t(C([0,\infty);\rd)_0)}^{P^W}$ denotes the completion of $\mathcal{B}_t(C([0,\infty);\rd)_0)$ with respect to $P^W$ in $\mathcal{B}(C([0,\infty);\rd)_0)$.
\end{dfn}
Before we state the main result of this section, we recall the definition of pathwise uniqueness for \eqref{StrongSolution.eq:DDSDE-lin}.
\begin{dfn}
    \textit{Pathwise uniqueness} holds for \eqref{StrongSolution.eq:DDSDE-lin} if for every pair of probabilistically weak solutions $(X,W), (Y,W)$ to \eqref{StrongSolution.eq:DDSDE-lin} defined on the same stochastic basis $(\Omega,\Fscr,(\Fscr_t),\PP)$ with respect to the same $(\Fscr_t)$-Brownian motion $W$ and with $X(0)=Y(0)$ $\PP$-a.s., we have $\sup_{t}|X(t)-Y(t)|=0$ $\PP$-a.s.
\end{dfn}
 
In Theorem \ref{prop:SP-for-Barenblatt}, we proved there exists a probabilistically weak solution $(X,W)$ to \eqref{StrongSolution.eq:DDSDE-lin} such that $\law{X(t)}=w_\delta(t,x)\mathrm{d}x, t\in \R_+$.
The following main result of this section
shows that under quite general conditions these solutions are actually strong.
\begin{theorem}[The Leibenson process is a functional of Brownian motion]\label{StrongSolution.theorem:existenceStrongSolution}
	Let $\delta>0$.
	Let $d\geq 2$, $p>1, q>0$ such that 
	$p>\frac{d}{d-1}$ 
	and 
	$q>\frac{|p-2|+d}{d(p-1)}\left(>\frac{1}{p-1}\right)$.
	Then, the probabilistically weak solution to \eqref{StrongSolution.eq:DDSDE-lin} constructed in Theorem \ref{prop:SP-for-Barenblatt} is a probabilistically strong solution.
Furthermore, for all $T>0$, pathwise uniqueness holds among weak solutions $(X,W)$ 
to \eqref{StrongSolution.eq:DDSDE-lin} on $[0,T]$.
\end{theorem}
\begin{rem}
    Note that Theorem \ref{StrongSolution.theorem:existenceStrongSolution} includes the case of $p$-Browninan motion, i.e., $p>2, q=1$.
\end{rem}
As already mentioned above, the case of arbitrary initial points $y\in\rd$ can be obtained analogously:
\begin{kor}\label{StrongSolution.corollary:existenceStrongSolution}
	Let $y\in \rd$ and $\delta>0$.
	Let $d\geq 2$, $p>1, q>0$ such that 
	$p>\frac{d}{d-1}$ 
	and 
	$q>\frac{|p-2|+d}{d(p-1)}$.
	Then, the probabilistically weak solution to \eqref{StrongSolution.eq:DDSDE-lin} constructed in Corollary \ref{cor:SP-princ-Barenblatt-all-y} is a probabilistically strong solution.
Furthermore, for all $T>0$, pathwise uniqueness holds among weak solutions $(X,W)$ 
to \eqref{StrongSolution.eq:DDSDE-lin} on $[0,T]$.
\end{kor}

\subsection{The restricted Yamada--Watanabe theorem}

Our result is based on the \textit{restricted Yamada--Watanabe theorem}, which we will recall here in a form tailored to the ordinary SDE \eqref{StrongSolution.eq:DDSDE-lin}. For the general formulation for stochastic partial differential equations in the variational framework, we refer to \cite[Theorem 1.3.1]{Grube-thesis}.

\begin{theorem}
\label{StrongSolution.theorem.restrictedYamadaWatanabe}
Assume
\begin{enumerate}[(i)]
	\item \label{StrongSolution.theorem.restrictedYamadaWatanabe:i} There exists a probabilistically weak solution $(X,W)$ to \eqref{StrongSolution.eq:DDSDE-lin};
	\item \label{StrongSolution.theorem.restrictedYamadaWatanabe:ii} Pathwise uniqueness holds for \eqref{StrongSolution.eq:DDSDE-lin}.
\end{enumerate}
Then the weak solution $(X,W)$ to \eqref{StrongSolution.eq:DDSDE-lin} from \eqref{StrongSolution.theorem.restrictedYamadaWatanabe:i} is a probabilistically strong solution.
\end{theorem}
\begin{proof}
    The assertion follows from \cite[Theorem 1.3.1]{Grube-thesis} by setting (using the notation of \cite{Grube-thesis})
    \begin{align*}
        P_{w_\delta(0,x)\mathrm{d}x} := \{ Q \in \Pscr(C([0,\infty);\rd)): Q\circ \pi_t^{-1} =w_\delta(t,x)\mathrm{d}x\ \ \forall t\in \R_+\}.
    \end{align*}
\end{proof}
By Theorem \ref{prop:SP-for-Barenblatt}, Assumption \eqref{StrongSolution.theorem.restrictedYamadaWatanabe:i} of Theorem \ref{StrongSolution.theorem.restrictedYamadaWatanabe} is satisfied. Therefore, the rest of the proof of Theorem \ref{StrongSolution.theorem:existenceStrongSolution} boils down to a pathwise uniqueness result for \eqref{StrongSolution.eq:DDSDE-lin}. The proof of Theorem \ref{StrongSolution.theorem:existenceStrongSolution} is carried out in Section \ref{subsect:StrongSolution.existenceStrongSolution.proof}.
In the following two subsections we present some technical preparations.

\subsection{About the Schwartz distributional derivative of \texorpdfstring{$\nabla \vrho_\delta(t,\cdot)$} \ \  }\label{section.condition.2}

Here we discuss regularity and integrability aspects of the drift coefficient of \eqref{StrongSolution.eq:DDSDE-lin}, which will be essential for the proof of Theorem \ref{StrongSolution.theorem:existenceStrongSolution}.
\begin{lem}\label{lemma.drift.regularity}
Let $d\geq 2$, $p>1,q>0$ such that $q(p-1)>1$ and $p>\frac{d}{d-1}$. Let $\delta>0$ and $t\geq 0$.
    Then, $$\nabla \vrho_\delta(t,\cdot)\in BV(\rd;\rd)$$ with Schwartz distributional derivatives given by the finite signed Borel measures (on $\rd$)
	\begin{align}\label{lemma.drift.regularity:1}
		\partial_i\partial_j \vrho_\delta(t,A) = \int_A g^{i,j}_{\delta}(t,z)\mathrm{d}z + C  (t+\delta)^{-\frac{d(p-2)}{\beta}-\frac{p}{\beta}-\frac{d(p-1)(q-1)}{\beta}}\int_{A\cap \partial B_{R_{\delta}(t)}(0)}z_iz_j S(\mathrm{d}z) \ \ 
		\forall A \in \mathcal{B}(\rd),
	\end{align}
	where $1\leq i,j \leq d$, $C>0$ is a constant not depending on $t$, and $g_{\delta}^{i,j} (t,\cdot)\in L^1(\rd)$ denotes the classical second order derivative of $\vrho_\delta(t,\cdot)$ with respect to the coordinates $x_i$ and $x_j$ on $\rd\backslash (\partial B_{R_{\delta}(t)}(0) \cup \{0\})$.
\end{lem}
\begin{proof}
In the proof of Theorem \ref{prop:SP-for-Barenblatt} we already showed that $\vrho_\delta(t,\cdot)\in L^1(\rd)$ and that $\vrho_\delta(t,\cdot)$ is weakly differentiable on $\rd$ with spatial weak derivative (and also classical derivative in $x\in \rd \backslash (\partial B_{R_{\delta}(t)}(0)\cup \{0\})$)
\begin{align}\label{StrongSolution.vrho_delta.derivative.function}
	\nabla \vrho_\delta(t,x)
	=C_{\vrho}(t+\delta) \left(\frac{p-2}{p-1} f_+(t+\delta,x) |x|^{-\frac p {p-1}} -\frac {\kappa p}{p-1} (t+\delta)^{- \frac p {\beta(p-1)}}\mathbbm{1}_{B_{R_\delta(t)}(0)}(x)\right) x,
\end{align}
where $C_{\vrho}$ is as in \eqref{StrongSolution.vrho_delta.function:constant}. Note that $\nabla \vrho_\delta(t,\cdot)$ has compact support, so that  $\nabla \vrho_\delta(t,\cdot)\in L^1(\rd)$.
Furthermore, it is clear that $\vrho_\delta(t,\cdot)$ is twice continuously differentiable on $\rd \backslash (\partial B_{R_{\delta}(t)}(0)\cup \{0\})$.
 Recall that we denote the ordinary second-order derivative of $\vrho_\delta(t,\cdot)$ in directions $x_i$ and $x_j$ by $g^{i,j}_{\delta}(t,\cdot)$.
A direct computation yields for all $x\in \rd \backslash (\partial B_{R_\delta(t)}(0)\cup \{0\})$,
\begin{align*}
	g^{i,j}_{\delta}(t,x)
	 =&\ C_{\vrho}(t+\delta)\mathbbm{1}_{B_{R_\delta(t)}(0)}(x)\bigg[
	    \frac{p-2}{p-1} \left(-\frac{\kappa p}{p-1} (t+\delta)^{-\frac{p}{\beta(p-1)}}|x|^\frac{2-p}{p-1}x_i\right)|x|^{-\frac{p}{p-1}}x_j \\
	    &+ \frac{p-2}{p-1}	 
	    f(t+\delta,x)\left(-\frac{p}{p-1}|x|^{\frac{2-3p}{p-1}}x_i x_j + |x|^{-\frac{p}{p-1}}\delta_{ij}\right)-\frac{\kappa p}{p-1} (t+\delta)^{-\frac{p}{\beta(p-1)}}\delta_{ij}\bigg]\\
	=&\ C_{\vrho}(t+\delta)\mathbbm{1}_{B_{R_\delta(t)}(0)}(x)\bigg[
	    -\frac{\kappa p(p-2)}{(p-1)^2}(t+\delta)^{-\frac{p}{\beta(p-1)}}|x|^{-2}x_i x_j \\
	   &+ \frac{p-2}{p-1}
	    f(t+\delta,x)\left(-\frac{p}{p-1}|x|^{\frac{2-3p}{p-1}}x_i x_j + |x|^{-\frac{p}{p-1}}\delta_{ij}\right)-\frac{\kappa p}{p-1} (t+\delta)^{-\frac{p}{\beta(p-1)}}\delta_{ij}\bigg],
\end{align*}
where $\delta_{ij}$ denotes the usual Kronecker delta.
Note that $g^{i,j}_{\delta}(t,\cdot)\in L^1(\rd)$, since $p>\frac{d}{d-1}$.
Now, we are prepared to prove \eqref{lemma.drift.regularity:1}.
Let $\varphi \in C_c^\infty(\rd)$.
By Lebesgue's dominated convergence theorem, we have
\begin{align*}
	\int_{B_{R_\delta(t)}(0)} \partial_i \vrho_\delta(t,x) \partial_j\varphi(x) \dx
	= \lim_{\varepsilon\downarrow 0} \int_{B_{R_\delta(t)-\varepsilon}(0)\backslash B_{\varepsilon}(0)} \partial_i \vrho_\delta(t,x) \partial_j \varphi(x)\dx.
	\end{align*}
Let  $0<\varepsilon<R_{\delta}(t)/2$. Then integration by parts yields 
\begin{align}\label{lemma.drift.regularity:2}
	\int_{B_{R_{\delta}(t)-\varepsilon}(0)\backslash B_{\varepsilon}(0)} &\partial_i \vrho_\delta(t,x) \partial_j \varphi(x)\dx 
	= -\int_{B_{R_{\delta}(t)-\varepsilon}(0)\backslash B_{\varepsilon}(0)} g_{\delta}^{i,j}(t,x)  \varphi(x)\dx\\
	&+ \int_{\partial B_{R_{\delta}(t)-\varepsilon}(0)} \partial_i  \vrho_\delta(t,x)\varphi(x) \frac{x_j}{R_{\delta}(t)-\varepsilon}\mathrm{d}S(x)
	-\int_{\partial B_{\varepsilon}(0)} \partial_i  \vrho_\delta(t,x)\varphi(x) \frac{x_j}{\varepsilon}\mathrm{d}S(x)\notag.
\end{align}
First, the third summand on the right-hand side in \eqref{lemma.drift.regularity:2} vanishes as $\varepsilon \to 0$, since
\begin{align*}
	\left|\int_{\partial B_{\varepsilon}(0)} \partial_i  \vrho_\delta(t,x)\varphi(x) \frac{x_j}{\varepsilon}\mathrm{d}S(x)\right|
	\lesssim_\delta \left(\varepsilon^{-\frac{1}{p-1}} + \varepsilon\right)\varepsilon^{d-1} \to 0, \text{ as $\varepsilon\to 0$,}
\end{align*}
where we used $p> \frac{d}{d-1}$.
Secondly, due to \eqref{StrongSolution.vrho_delta.derivative.function}, the second summand on the right-hand side in \eqref{lemma.drift.regularity:2} can be written as
\begin{align*}
	\int_{\partial B_{R_{\delta}(t)-\varepsilon}(0)} \frac{\partial_i  \vrho_\delta(t,x)x_j}{R_{\delta}(t)-\varepsilon}\varphi(x) \mathrm{d}S(x)
	 &=  \frac{C_{\vrho}(t+\delta)(p-2)}{(R_{\delta}(t)-\varepsilon)(p-1)}\int_{\partial B_{R_{\delta}(t)-\varepsilon}(0)} f_+(t+\delta,x) |x|^{-\frac{p}{p-1}}x_ix_j\varphi(x)\mathrm{d}S(x)\\
	& \ \ \ \ -\frac{C_{\vrho}(t+\delta)(t+\delta)^{-\frac{p}{\beta(p-1)}}\kappa p}{(p-1)(R_{\delta}(t)-\varepsilon)} \int_{\partial B_{R_{\delta}(t)-\varepsilon}(0)} x_ix_j\varphi(x)\mathrm{d}S(x).
\end{align*}
Using the transformation $x\mapsto (R_{\delta}(t)-\varepsilon)x$, it is easy to see that the first summand on the right-hand side of the last equality vanishes as $\varepsilon\to 0$. A similar argument regarding the last integral on the right-hand side yields
\begin{align*}
    \int_{\partial B_{R_{\delta}(t)-\varepsilon}(0)} x_ix_j\varphi(x)\mathrm{d}S(x) \to \int_{\partial B_{R_{\delta}(t)}(0)} x_ix_j\varphi(x)\mathrm{d}S(x), \text{ as $\varepsilon\to 0$.}
\end{align*}

Hence,
letting $\varepsilon\to 0$ in \eqref{lemma.drift.regularity:2}, we conclude that the spatial Schwartz distributional derivatives of the components of $\nabla \vrho_\delta$ can be written as
\begin{align*}
	_{\mathcal{D}'}\langle\partial_j \partial_i\vrho_\delta(t,\cdot),\varphi\rangle_{\mathcal{D}} = \int_\rd g_{\delta}^{i,j}(t,x)\varphi(x)\dx + \frac{C_{\vrho}(t+\delta)(t+\delta)^{-\frac{p}{\beta(p-1)}}\kappa p}{(p-1)(R_{\delta}(t)-\varepsilon)}\int_{\partial B_{R_{\delta}(t)}(0)} x_ix_j\varphi(x)\mathrm{d}S(x),
\end{align*}
for $\varphi \in C_c^\infty(\rd)$ and $1\leq i,j\leq d$, where $_{\mathcal{D}'}\langle\cdot,\cdot\rangle_D$, denotes the usual action of Schwartz distributions on $\rd$, denoted by $\mathcal{D}'$, on $\mathcal{D}\equiv C_c^\infty(\rd)$.
This finishes the proof.
\end{proof}
We have the following lemma.
\begin{lem}\label{lemma.PU.b.uniformBound}
	Let $d\geq 2$ with $\delta>0$. Let $p>1,q>0$ such that $q(p-1)>1$ and $p>\frac{d}{d-1}$.
	Then 
	\begin{align}\label{lemma.PU.b.uniformBound:toProve}
		\int_0^T \int \M|D_x \nabla\vrho_\delta(t,\cdot)|(x)w_\delta(t,x)\mathrm{d}x\mathrm{d}t<\infty\ \ \forall T\in (0,\infty),
	\end{align}
    where $D_x \nabla\vrho_\delta(t,\cdot)$ denotes the matrix consisting of all second order Schwartz distributional derivatives of $\vrho_\delta(t,\cdot)$ in form of finite signed Borel measures, according to Lemma \ref{lemma.drift.regularity}.
\end{lem}
\begin{proof}
    Let $T>0$.
    According to and using the notation of Lemma \ref{lemma.drift.regularity}, we may estimate as follows.
	\begin{align}
		\int_0^T \int \M|D_x \nabla\vrho_\delta(t,\cdot)|(x)w_\delta(t,x)\mathrm{d}x\mathrm{d}t 
		&\lesssim_\delta \int_0^T \int \M\left(\max_{1\leq i,j\leq d}|g^{i,j}_{\delta}(t,\cdot)|\right)(x)w_\delta(t,x)\mathrm{d}x\mathrm{d}t\\
		&\ \ \ + \int_0^T \int \M|S(\cdot\cap \partial B_{R_{\delta}(t)}(0))|(x)w_\delta(t,x)\mathrm{d}x\mathrm{d}t.
	\end{align}
	According to Lemma \ref{lemma.drift.regularity}, we have
	\begin{align*}
        \max_{1\leq i,j \leq d} |g^{i,j}_{\delta}(t,x)| \lesssim_\delta 1 + |x|^{\frac{-p}{p-1}}.
	\end{align*}
    Note that $|\cdot|^{-\frac{p}{p-1}}\in A_1$, since $p>\frac{d}{d-1}$. 
    Hence, we obtain
    \begin{align*}
		\int_0^T \int \M\left(\max_{1\leq i,j\leq d}|g^{i,j}_{\delta}(t,\cdot)|\right)(x)w_\delta(t,x)\mathrm{d}x\mathrm{d}t
		&\lesssim_\delta 1+ \int_0^T \int  \M(|\cdot|^{-\frac{p}{p-1}})(x)w_\delta(t,x)\mathrm{d}x\mathrm{d}t \\
		&\lesssim 1+ \int_0^T \int  |x|^{-\frac{p}{p-1}}w_\delta(t,x)\mathrm{d}x\mathrm{d}t<\infty.
	\end{align*}
    Furthermore, by Cauchy's surface area formula, we have
    \begin{align}\label{lemma.PU.b.uniformBound:Cauchy}
        S(B_r(x)\cap \partial B_{R_\delta(t)}(0)) 
        \leq S(\partial B_r(0))
        \lesssim r^{d-1}.
    \end{align}
    Also, for each $x\in \rd$ such that $|x|\neq R_\delta(t)$ and all $0<r\leq ||x|-R_\delta(t)|$ we have
    \begin{align*}
        S(B_r(x)\cap \partial B_{R_\delta(t)}(0)) = 0.
    \end{align*}
    Hence, for all $x\in \rd$ such that $|x|\neq R_\delta(t)$, we obtain
	\begin{align}\label{lemma.PU.b.uniformBound:1}
		\M|S(\cdot \cap \partial B_{R_\delta(t)}(0))|(x)
        \lesssim \frac{1}{||x|-R_\delta(t)|}.
	\end{align}
    \noindent\textbf{Claim:} For all $\rho\geq 0$ there exists $C_\rho \geq 1$ such that for all $r\in [0,1]$
    \begin{align}\label{lemma.PU.b.uniformBound:claim}
        1-r^\rho \leq C_\rho (1-r).
    \end{align}
    \textit{Proof of Claim:} By l'Hospital's rule, the function
    \begin{align*}
        [0,1) \ni r \mapsto \frac{1-r^\rho}{1-r}
    \end{align*}
    can be extended to a continuous function on the interval $[0,1]$, which clearly has a maximum $C_\rho \geq 1$ on $[0,1]$.\qed\\
    
	Using \eqref{lemma.PU.b.uniformBound:1}, the transformation $x\mapsto R_\delta(t)x$, integration in polar coordinates, and \eqref{lemma.PU.b.uniformBound:claim}, we obtain
	\begin{align*}
		\int_0^T \int \M|S(\cdot\cap B_{R_{\delta}(t)}(0))|(x)w_\delta(t,x)\mathrm{d}x\mathrm{d}t
		&\lesssim_\delta \int_0^T \int  (||x|-R_\delta(t)|)^{-1} f_+(t+\delta,x)^\gamma \mathrm{d}x\mathrm{d}t\\
	&\lesssim_\delta \int_0^1 \frac{r^{d-1}}{1-r} (1-r^\frac{p}{p-1})^\gamma\mathrm{d}r
    \lesssim \int_0^1 (1-r)^{\gamma-1}\mathrm{d}r
    < \infty.
	\end{align*}
	 Hence, since $\gamma>0$, \eqref{lemma.PU.b.uniformBound:toProve} follows.
\end{proof}
\begin{rem}\label{lemma.PU.b.uniformBound:remark-optimality}
    At first sight, the estimate \eqref{lemma.PU.b.uniformBound:1} seems quite rough.
    But, in fact, elementary geometrical considerations show that, for $d=3$ and fixed $t\geq 0$, $$\M |D_x \nabla \vrho_\delta(t,\cdot)|(x) \cong \frac{1}{|x|||x|-R_\delta(t)|},$$ for all $x$ in a neighborhood of $\partial B_{R_\delta(t)}(0)$.
\end{rem}

\subsection{Weak derivatives (of powers) of \texorpdfstring{$\sqrt\vrho_\delta$}\ }
Here we discuss regularity and integrability aspects of the diffusion coefficient of \eqref{StrongSolution.eq:DDSDE-lin}, which will be essential in the proof of Theorem \ref{StrongSolution.theorem:existenceStrongSolution}.
\begin{lem}\label{StrongSolution.lemma.regularity:rho_delta:kor}
    Let $d\geq 2$, $p>1,q>0$ with $q(p-1)>1$, and
\begin{align}\label{StrongSolution.lemma.regularity:rho_delta:1c}
    \frac{(p-2)(1+\gamma)}{2(p-1)}>1-d.
\end{align}
    Then for each $t>0$, $\vrho_\delta^{\frac{1+\gamma}2}(t,\cdot)$ is weakly differentiable.
\end{lem}
\begin{rem}\label{StrongSolution.remark.regularity:rho_delta:kor}
    Consider
    \begin{enumerate}[(i)]
        \item \label{StrongSolution.corrolary.regularity:rho_delta:i}
    \begin{align}\label{StrongSolution.corrolary.regularity:rho_delta:1a}
        p>\frac{2d}{2d-1} \text{\ and\ }q>\frac{2d(p-1)-p-(p-2)(p-1)}{(p-1)(2d(p-1)-p)};
    \end{align}
        \item \label{StrongSolution.corrolary.regularity:rho_delta:ii}
        \begin{align}\label{StrongSolution.corrolary.regularity:rho_delta:1b}
        (2\geq)\ \frac{2d}{2d-1}>p>1 \text{\ and\ }q<\frac{2d(p-1)-p-(p-2)(p-1)}{(p-1)(2d(p-1)-p)}.
    \end{align}
    \end{enumerate}
    In fact, \eqref{StrongSolution.lemma.regularity:rho_delta:1c} holds if and only if \eqref{StrongSolution.corrolary.regularity:rho_delta:1a} or \eqref{StrongSolution.corrolary.regularity:rho_delta:1b} is fulfilled.
\end{rem}
\begin{proof}
By \eqref{StrongSolution.vrho_delta.function}, clearly $\vrho_\delta^{\frac{1+\gamma}2}(t,\cdot) \in L^1(\rd)$. Note that $\vrho_\delta^{\frac{1+\gamma}2}(t,\cdot)$ is (infinitely many times) continuously differentiable on $\rd\backslash(\partial B_{R_\delta(t)}(0)\cup\{0\})$.
Let $\varphi\in C_c^\infty(\rd)$ and $0<\varepsilon<R_\delta(t)/2$.
By Lebesgue's dominated convergence, we obtain
    \begin{align*}
        \int \vrho_\delta^{\frac{1+\gamma}2}(t,x)\partial_i\varphi(x)\mathrm{d}x
        = \lim_{\varepsilon\to 0} \int_{B_{R_\delta(t)-\varepsilon}\backslash B_\varepsilon(0)}\vrho_\delta^{\frac{1+\gamma}2}(t,x)\partial_i\varphi(x)\mathrm{d}x.
    \end{align*}
    Integrating by parts, we obtain for $1\leq i\leq d$
    \begin{align*}
        \int_{B_{R_\delta(t)-\varepsilon}\backslash B_\varepsilon(0)}\vrho_\delta^{\frac{1+\gamma}2}(t,x)\partial_i\varphi(x)\mathrm{d}x
        &=-\int_{B_{R_\delta(t)-\varepsilon}\backslash B_\varepsilon(0)} \partial_i \vrho_\delta^{\frac{1+\gamma}2}(t,x) \varphi(x) \mathrm{d}x\\
        &+  \int_{\partial B_{R_\delta(t)-\varepsilon}} \vrho_\delta^{\frac{1+\gamma}2}(t,x) \varphi(x) \frac{x_i}{R_{\delta}(t)-\varepsilon} S(\mathrm{d}x)\\
        &-\int_{\partial B_\varepsilon(0)} \vrho_\delta^{\frac{1+\gamma}2}(t,x) \varphi(x) \frac{x_i}{R_{\delta}(t)-\varepsilon} S(\mathrm{d}x),
    \end{align*}
    where $\partial_i \vrho_\delta^{\frac{1+\gamma}2}(t,\cdot)$ denotes the $\mathrm{d}x$-a.e. existing pointwise partial derivative of $\vrho_\delta^{\frac{1+\gamma}2}(t,\cdot)$ in direction $x_i$.
    By Lebesgue's dominated convergence theorem
    \begin{align}
        \int_{\partial B_{R_\delta(t)-\varepsilon}} &\vrho_\delta^{\frac{1+\gamma}2}(t,x) \varphi(x) \frac{x_i}{R_{\delta}(t)-\varepsilon} S(\mathrm{d}x)\notag \\
        &= \int_{\partial B_1(0)}  \vrho_\delta^{\frac{1+\gamma}2}(t,(R_\delta(t)-\varepsilon)x) \varphi((R_\delta(t)-\varepsilon)x)x_i (R_\delta(t)-\varepsilon)^{d-1} S(\mathrm{d}x)
        \xrightarrow{\varepsilon \to 0} 0.
    \end{align}
    Furthermore, by \eqref{StrongSolution.lemma.regularity:rho_delta:1c},
    \begin{align}
        \int_{\partial B_\varepsilon(0)} \vrho_\delta^{\frac{1+\gamma}2}(t,x) \varphi(x) \frac{x_i}{R_{\delta}(t)-\varepsilon} S(\mathrm{d}x)
        \lesssim_\delta \varepsilon^{\frac{(p-2)(1+\gamma)}{2(p-1)}+d-1} \xrightarrow{\varepsilon \to 0} 0.
    \end{align}
    Using chain rule as well as \eqref{StrongSolution.vrho_delta.function} and \eqref{StrongSolution.vrho_delta.derivative.function}, an easy calculation yields that the pointwise partial derivatives of $\vrho_\delta^{\frac{1+\gamma}2}(t,\cdot)$ on $\rd\backslash(\partial B_{R_\delta(t)}(0)\cup\{0\})$ satisfy the bound
    \begin{align}\label{StrongSolution.lemma.regularity:rho_delta:4}
        |\partial_i \vrho_\delta^{\frac{1+\gamma}2}(t,x)|
        \lesssim_\delta f_+(t+\delta,x)^\frac{\gamma-1}{2} |x|^\frac{(p-2)(1+\gamma)+2}{2(p-1)}
        +  f_+(t+\delta,x)^\frac{\gamma+1}{2} |x|^{\frac{(p-2)(1+\gamma)}{2(p-1)}-1},
    \end{align}
    where $i \in \{1,...,d\}$.
    Under condition \eqref{StrongSolution.lemma.regularity:rho_delta:1c}, it is clear that $\partial_i \vrho_\delta^{\frac{1+\gamma}2}(t,\cdot) \in L^1(\rd)$.
    Hence, Lebesgue's dominated convergence theorem yields
    \begin{align}
        \int_{B_{R_\delta(t)-\varepsilon}\backslash B_\varepsilon(0)} \partial_i \vrho_\delta^{\frac{1+\gamma}2}(t,x) \varphi(x) \mathrm{d}x 
        \xrightarrow{\varepsilon \to 0} \int \partial_i \vrho_\delta^{\frac{1+\gamma}2}(t,x) \varphi(x) \mathrm{d}x.
    \end{align}
    This completes the proof.
\end{proof}
We have the following lemma.
\begin{lem}\label{lemma.PU.sigma.uniformBound}
	Let $d\geq 2$ $\delta>0$, and $p>1, q>0$ be such that $p>\frac{d}{d-1}$ and $q>\frac{|p-2|+d}{d(p-1)}$.
	Then 
	\begin{align}\label{lemma.PU.sigma.uniformBound:bound}
		\int_0^T \int_\rd \left(\M|\nabla\vrho_\delta^{\frac{1+\gamma}{2}}(t,\cdot)|(x)\right)^2 |x|^{-\frac{p-2}{q(p-1)-1}}\mathrm{d}x\mathrm{d}t<\infty\quad \forall T>0.
	\end{align}
\end{lem}
\begin{proof} Let $T>0$.
	Note
     \begin{align}
         \frac{2d(p-1)-p-(p-2)(p-1)}{(p-1)(2d(p-1)-p)}<\frac{|p-2|+d}{d(p-1)}.
	\end{align}
    Hence, by Lemma \ref{StrongSolution.lemma.regularity:rho_delta:kor} and Remark \ref{StrongSolution.remark.regularity:rho_delta:kor}, we have
    $\nabla\vrho_\delta^{\frac{1+\gamma}{2}}(t,\cdot) \in L^1(\rd)$.
    Since $q>\frac{|p-2|+d}{d(p-1)}$, we have
    \begin{align*}
		|\cdot|^{-\frac{p-2}{q(p-1)-1}} \in A_2
	\end{align*}
    (for the definition of the Muckenhoupt class $A_2$, we refer to Appendix \ref{App-D.muckenhoupt-weights}).
	Hence, in order to show \eqref{lemma.PU.sigma.uniformBound:bound}, it is sufficient to prove
	\begin{align*}
		\int_0^T\int_\rd \left|\nabla\vrho_\delta^{\frac{1+\gamma}{2}}(t,x)\right|^2 |x|^{-\frac{p-2}{q(p-1)-1}}\mathrm{d}x\mathrm{d}t<\infty.
	\end{align*}
    By the proof of Lemma \ref{StrongSolution.lemma.regularity:rho_delta:kor} (more precisely \eqref{StrongSolution.lemma.regularity:rho_delta:4})
    \begin{align}
        \left|\nabla\vrho_\delta^{\frac{1+\gamma}{2}}(t,x)\right|
         \lesssim_\delta f_+(t+\delta,x)^\frac{\gamma-1}{2} |x|^\frac{(p-2)(1+\gamma)+2}{2(p-1)}
        +  f_+(t+\delta,x)^\frac{\gamma+1}{2} |x|^{\frac{(p-2)(1+\gamma)}{2(p-1)}-1},
    \end{align}
whence
	\begin{align*}
		\int_0^T\int_\rd &\left|\nabla\vrho_\delta^{\frac{1+\gamma}{2}}\right|^2 |x|^{-\frac{p-2}{q(p-1)-1}}\mathrm{d}x\mathrm{d}t\\
		&\lesssim \int_0^T \int_\rd f_+(t+\delta,x)^{\gamma-1}|x|^{\frac{p}{p-1}}+f_+(t+\delta,x)^{1+\gamma}|x|^{-\frac{p}{p-1}}\mathrm{d}x\mathrm{d}t<\infty,
	\end{align*}
	where we used $p>\frac{d}{d-1}$ and $\gamma>0$.
This completes the proof.
\end{proof}

\subsection{Proof of Theorem \ref{StrongSolution.theorem:existenceStrongSolution}}\label{subsect:StrongSolution.existenceStrongSolution.proof}

We show that Theorem \ref{StrongSolution.theorem:existenceStrongSolution} follows from Theorem \ref{StrongSolution.theorem.restrictedYamadaWatanabe}. Therefore, we need to show that both conditions of the latter theorem are fulfilled.

In fact, assumption \eqref{StrongSolution.theorem.restrictedYamadaWatanabe:ii} of Theorem \ref{StrongSolution.theorem.restrictedYamadaWatanabe} follows from Theorem \ref{prop:SP-for-Barenblatt}, and \eqref{StrongSolution.theorem.restrictedYamadaWatanabe:ii} follows from Theorem \ref{StrongSolution.theorem:pathwiseU} below.\\

The proof of the following theorem is inspired by the technique in \cite{roeckner2010weakuniqueness}.
\begin{theorem}\label{StrongSolution.theorem:pathwiseU}
    Let $\delta>0$ and $T>0$.
	Let $d\geq 2$, $p>1, q>0$ such that 
	$p>\frac{d}{d-1}$ 
	and 
	$q>\frac{|p-2|+d}{d(p-1)}$.
	Let $(X,W)$ and $(Y,W)$ be two probabilistically weak solutions to \eqref{StrongSolution.eq:DDSDE-lin} on $[0,T]$ on a common stochastic basis $(\Omega,\Fscr,(\Fscr_t)_{t\in [0,T]},\PP)$ with respect to the same $(\Fscr_t)$-Brownian motion $(W(t))_{t\in [0,T]}$ such that $X(0)=Y(0)$ $\PP$-a.s. Then $X=Y$ $\PP$-a.s.
\end{theorem}
\begin{proof}
    Due to Proposition \ref{App-E.proposition}, we may assume without loss of generality
    \begin{align}\label{StrongSolution.theorem:pathwiseU:nablaVrho_delta=0}
    \nabla\vrho_\delta(t,0)=0, \ \ \forall t\in [0,T].
     \end{align}
	From Lemma \ref{lemma.drift.regularity} and the proof of Theorem \ref{prop:SP-for-Barenblatt}, it follows that $\sqrt{\vrho_\delta},|\nabla\vrho_\delta|\in L^1([0,T]\times\rd)$.
	For $n \in \N$ we set
	\begin{align*}
		\tau_n &:= \inf\{t\in [0,\infty) : |X(t)|\vee|Y(t)|\geq n \}.
	\end{align*}
	Note that $\tau_n$ is an $(\Fscr_t)$-stopping time.
	For $\varepsilon >0$, we set $h_\varepsilon(x):=\ln(1+|x|^2/\varepsilon^2)$, $x\in \rd$. Clearly, for all $x \in \rd$ 
    $$|\nabla h_\varepsilon(x)| \lesssim \frac{1}{|x|+\varepsilon},\ \ \  |D^2 h_\varepsilon(x)| \lesssim \frac{1}{|x|^2+\varepsilon^2}.$$
    Let $n\in \N, \varepsilon>0$. Using It\^o's formula, we obtain for each $t\in [0,T]$ (cf. proof of \cite[Theorem 1.1]{roeckner2010weakuniqueness})
	\begin{align}
		\E &\left[\ln\left(1+ \frac{|X(t\wedge \tau_n)-Y(t\wedge \tau_n)|^2}{\varepsilon^2}\right)\right] \notag\\
		&\lesssim \E \left[\int_0^{t\wedge \tau_n} \frac{2\scalarproduct[\rd]{\nabla\vrho_\delta(s,X(s)) - \nabla\vrho_\delta(s,Y(s))}{X(s)-Y(s)} + |\sqrt{\vrho_\delta}(s,X(s))-\sqrt{\vrho_\delta}(s,Y(s))|^2}{|X(s)-Y(s)|^2+\varepsilon^2}\mathrm{d}s\right] \notag\\
		&\lesssim \E \left[\int_0^{t\wedge \tau_n} \frac{|\nabla\vrho_\delta(s,X(s)) - \nabla\vrho_\delta(s,Y(s))|}{|X(s)-Y(s)|+\varepsilon}\mathrm{d}s\right] + \E \left[\int_0^{t\wedge \tau_n}\frac{|\sqrt{\vrho_\delta}(s,X(s))-\sqrt{\vrho_\delta}(s,Y(s))|^2}{|X(s)-Y(s)|^2+\varepsilon^2}\mathrm{d}s\right] \notag\\
		&=: (\star_1) + (\star_2).
        \end{align}
        \textbf{Claim:}
    For \textit{all} $x,z\in \rd$
    \begin{align}
	\left|\nabla\vrho_\delta(t,x)-\nabla\vrho_\delta(t,z)\right|
    \lesssim \left (\M_{}|D_x \nabla\vrho_\delta(t,\cdot)|(x)+\M_{}|D_x \nabla\vrho_\delta(t,\cdot)|(z)\right)|x-z|.\label{CDL.maximalIneq}
    \end{align}\\
    \textit{Proof of Claim:}
    By \eqref{StrongSolution.theorem:pathwiseU:nablaVrho_delta=0}
    and the continuity of $\nabla \vrho_\delta(t,\cdot)$ on $\rd\backslash(\partial B_{R(t)}(0)\cup\{0\})$, it is easy to see that for any standard radial Dirac sequence $(\varphi_\varepsilon)_{\varepsilon>0}$, 
    \begin{align}
        (\nabla \vrho_\delta(t,\cdot)\ast\varphi_\varepsilon)(x)\to  \nabla\vrho_\delta(t,x)\quad  \forall x \in \rd\backslash\partial B_{R(t)}(0),
    \end{align}
    where the convolution is meant componentwise.
    
    Furthermore, $\partial B_{R_\delta(t)}(0)$ consists of all \textit{approximate jump points} of $\vrho_\delta(t,\cdot)$ in the sense of \cite[Definition 3.67]{ambrosio2000BV}. Hence, by Lemma \ref{appendix.lemma.lipschitztypeestimate} and Remark \ref{appendix.maximalfunction.remark}, the assertion follows.
  Indeed: Fix an arbitrary $t\geq 0$. For all $x_0 \in B_{R_\delta(t)}(0)$, we set (in the notation of \cite{ambrosio2000BV}), $\nu := \frac{x_0}{|x_0|}$, $a:=-\frac{\kappa p}{p-1}(t+\delta)^{-\frac{p}{\beta(p-1)}}x_0$, $b:=0$.
    Furthermore, let $B_r^-(x,\nu):= \{y\in B_r(x): \langle y-x,\nu\rangle <0\}$ and $B_r^+(x,\nu):= \{y\in B_r(x): \langle y-x,\nu\rangle >0\}$.
Then a simple transformation of the Lebesgue integral shows (here $|A|$ denotes the Lebesgue measure of a set $A\in\mathcal{B}(\rd)$)
    \begin{align*}
        \frac{1}{|B_r^-(x_0,\nu)|}\int_{B_r^-(x_0,\nu)}|\nabla\varrho_\delta(t,y)-a|\mathrm{d}y \lesssim \int_{B_1^-(0,\nu)} |\nabla\varrho_\delta(t, x_0 + ry)-a|\mathrm{d}y \to 0\text{, as } r\downarrow 0,
    \end{align*}
    where we used the fact that $\nabla\varrho_\delta(t, y)\to a$, whenever $y\to x_0$ with $y \in B_{R_\delta(t)}(0)$, and that $\nabla\varrho_\delta(t,\cdot)$ is bounded on $\rd\backslash B_{\varepsilon}(0)$, for all $\varepsilon>0$.
    Furthermore, since $\text{supp}(|\nabla\varrho_\delta(t,\cdot)|)=\overline{B_{R_\delta(t)}(0)}$, we have
    \begin{align*}
         \frac{1}{|B_r^+(x_0,\nu)|}\int_{B_r^+(x_0,\nu)}|\nabla\varrho_\delta(t,y)-b|\mathrm{d}y =0 \quad \forall r>0.
    \end{align*}
    This completes the proof of the claim.
    
    \qed\\
    
   \noindent\textbf{Proof of Theorem \ref{StrongSolution.theorem:pathwiseU} continued.}
    Using \eqref{CDL.maximalIneq}, we may estimate $(\star_1)$ by
    \begin{align}
	   (\star_1)&\lesssim \E\left[\int_0^{t\wedge \tau_n} \M_{}|D_x \nabla\vrho_\delta(s,\cdot)|(X(s))+\M_{}|D_x \nabla\vrho_\delta(s,\cdot)|(Y(s)) \mathrm{d}s\right] \notag\\
	   &\leq 2\int_0^T \int_{B_n(0)} \M_{}|D_x \nabla\vrho_\delta(s,\cdot)|(x)w_\delta(s,x)\mathrm{d}x\mathrm{d}s.\label{theorem.PU:estimate.b}
    \end{align}
    By Lemma \ref{lemma.PU.b.uniformBound}, the right-hand side of \eqref{theorem.PU:estimate.b} is finite.
    
    Before we estimate $(\star_2)$, we remark that due to the mean value theorem, the following inequalities are true (under the convention $0\cdot \infty=0$)
   for all $x,y \in \rd$ if $p>2$, and for all $x,y \in \rd\backslash\{0\}$ if $p<2$:
    \begin{align}\label{StrongSolutions.theorem.PU:MVTineq}
    &|\sqrt{\vrho_\delta}(t,x)-\sqrt{\vrho_\delta}(t,y)|^2
        = \left|\left(\vrho_\delta^{\frac{1+\gamma}{2}}(t,x)\right)^{\frac{1}{1+\gamma}}-\left(\vrho_\delta^{\frac{1+\gamma}{2}}(t,y)\right)^{\frac{1}{1+\gamma}}\right|^2\notag\\
        &\lesssim \left(\int_0^1 \left[\theta\max{\left(\vrho_\delta^{\frac{1+\gamma}{2}}(t,x), \vrho_\delta^{\frac{1+\gamma}{2}}(t,y)\right)}\right]^{-\frac{\gamma}{1+\gamma}} \mathrm{d}\theta\right)^2\cdot\left|\vrho_\delta^{\frac{1+\gamma}{2}}(t,x)-\vrho_\delta^{\frac{1+\gamma}{2}}(t,y)\right|^2\notag\\
	   &\lesssim_\delta \left(\max\left(w_\delta(t,x) |x|^{\frac{(p-2)}{q(p-1)-1}},w_\delta(t,y) |y|^{\frac{(p-2)}{q(p-1)-1}}\right)\right)^{-1}\cdot\left|\vrho_\delta^{\frac{1+\gamma}{2}}(t,x)-\vrho_\delta^{\frac{1+\gamma}{2}}(t,y)\right|^2.
    \end{align}
    We estimate $(\star_2)$ as follows.
    \begin{align}
	   &\ \ \ \E\left[\int_0^T \frac{|\sqrt{\vrho_\delta}(t,X(t))-\sqrt{\vrho_\delta}(t,Y(t))|^2}{|X(t))-Y(t))|^2+\varepsilon^2}\mathrm{d}t\right]\label{ineq2}\\
	   &\lesssim_{\delta} \E\left[\bigintsss_0^T  \frac{\left|\vrho_\delta^{\frac{1+\gamma}{2}}(t,X(t))-\vrho_\delta^{\frac{1+\gamma}{2}}(t,Y(t))\right|^2}{(|X(t))-Y(t))|^2+\varepsilon^2)\max\left(w_\delta(t,X(t)) |X(t)|^{\frac{(p-2)}{q(p-1)-1}},w_\delta(t,Y(t)) |Y(t)|^{\frac{(p-2)}{q(p-1)-1}}\right)}\mathrm{d}t\right].\notag
    \end{align}
    Recall that, by Lemma \ref{StrongSolution.lemma.regularity:rho_delta:kor}, $\vrho_\delta^{\frac{1+\gamma}{2}}(t,\cdot)$
    is weakly differentiable. 
    Clearly, $\vrho_\delta^{\frac{1+\gamma}{2}}(t,\cdot)\in C(\rd)$ if $p>2$, and $\vrho_\delta^{\frac{1+\gamma}{2}}(t,\cdot)\in C(\rd\backslash \{0\})$ if $p<2$.
    Here, we emphasize that in the latter case, since $X,Y$ satisfy the integrability condition \eqref{def:DDSDE:integrabilityCondition}, respectively, $\{\min(|X|,|Y|)=0 \}$ is a $\mathrm{d}t\otimes \PP$-zero set. Hence, by \eqref{CDL.maximalIneq} where we formally replace $\nabla\vrho_\delta(t,\cdot)$ by $\vrho_\delta^{\frac{1+\gamma}{2}}(t,\cdot)$
    (which, in the case $p>2$, then holds for all $x,z \in \rd$, and, if $p<2$, for all  $x,z \in \rd\backslash \{0\}$ due to the function's respective continuity), we obtain
	\begin{align}
	   &\E\left[\bigintsss_0^T \frac{\left|\vrho_\delta^{\frac{1+\gamma}{2}}(t,X(t))-\vrho_\delta^{\frac{1+\gamma}{2}}(t,Y(t))\right|^2}{(|X(t))-Y(t))|^2+\varepsilon^2)\max\left(w_\delta(t,X(t)) |X(t)|^{\frac{(p-2)}{q(p-1)-1}},w_\delta(t,Y(t)) |Y(t)|^{\frac{(p-2)}{q(p-1)-1}}\right)}\mathrm{d}t\right]\notag\\
 	  &\lesssim \E \left[\bigintsss_0^T  \frac{\left(\M|\nabla\vrho_\delta^{\frac{1+\gamma}{2}}(t,\cdot)|(X(t))\right)^2}{w_\delta(t,X(t)) |X(t)|^{\frac{(p-2)}{q(p-1)-1}}}\mathrm{d}t \right]
      + \E \left[\bigintsss_0^T  \frac{\left(\M|\nabla\vrho_\delta^{\frac{1+\gamma}{2}}(t,\cdot)|(Y(t))\right)^2}{w_\delta(t,Y(t)) |Y(t)|^{\frac{(p-2)}{q(p-1)-1}}}\mathrm{d}t\right]\notag\\
	   &\lesssim 2 \int_0^T \int_\rd \left(\M|\nabla\vrho_\delta^{\frac{1+\gamma}{2}}(t,\cdot)|(x)\right)^2 |x|^{-\frac{p-2}{q(p-1)-1}}\mathrm{d}x\mathrm{d}t.\label{theorem.PU:estimate.sigma}
	\end{align}
    
    Note that, according to Lemma \ref{lemma.PU.sigma.uniformBound}, the right-hand side of \eqref{theorem.PU:estimate.sigma} is finite.
	Hence, by the monotone convergence theorem, we obtain
    \begin{align*}
	 	\E &\left[\lim_{\varepsilon\downarrow 0} \ln\left(1+ \frac{|X(t\wedge \tau_n)-Y(t\wedge \tau_n)|^2}{\varepsilon^2}\right)\right] = \lim_{\varepsilon\downarrow 0} \E \left[\ln\left(1+ \frac{|X(t\wedge \tau_n)-Y(t\wedge \tau_n)|^2}{\varepsilon^2}\right)\right]<\infty.
    \end{align*}
	Since $X$ and $Y$ are $\PP$-a.e. continuous, we conclude $\mathbb{P}$-a.s.
    \begin{align}
	\label{eq}
		&X(t\wedge \tau_n)=Y(t\wedge \tau_n) \ \ \forall t \in [0,T].
	\end{align}
	Again, by the $\PP$-a.s. continuity of $X,Y$, $\tau_n \to \infty$, as $n\to \infty$, $\PP$-a.s.
	Therefore, letting $n \to \infty$ in \eqref{eq}, we obtain $\PP$-a.s.
	\begin{align*}
     	X(t)=Y(t)\ \ \forall t\in [0,T].
	\end{align*}
	This completes the proof.
\end{proof}

\section{A restricted uniqueness result for the linearized Leibenson Fokker--Planck equation}\label{sect:restr-lin-u}
The goal of this section is to prove that, in the situation of Theorem \ref{t5.2}, condition ($\mathfrak{P}_0$-$\textup{lin}_{\textup{ex}}$) from Corollary \ref{c5.3}, more precisely its equivalent condition from Lemma \ref{lem:equiv}, is satisfied with $\mathfrak{P}_0$ as defined in \eqref{def:mathfrakP}, see Theorem \ref{restr-lin-u.theorem}.
As in Section \ref{sect:Strongsolution}, we set for $\delta>0$
$$w_\delta(t,x) := w(t+\delta,x),\quad \vrho_\delta(t,x):= |\nabla w_\delta(t,x)|^{p-2} w_\delta(t,x)^{(p-1)(q-1)},$$
where again $w$ denotes the Barenblatt solution from \eqref{def-w} with inital datum $\delta_0$. As computed in the proof of Theorem \ref{prop:SP-for-Barenblatt}, we have
\begin{align*}
	\vrho_\delta(t,x) \cong (t+\delta)^{   -\frac d \beta (p-2)  -\frac{p(p-2)}{\beta(p-1) }  -   \frac d \beta (q-1)(p-1)  }       f_+(t+\delta,x)  |x|^{\frac {p-2} {p-1}}.
\end{align*}
Here we study the $w_\delta$-linearized version of \eqref{eq:Leibenson-FPE}, i.e. the linear FPE
\begin{equation}\label{eq:Leibenson-FPE-lin}
	\partial_t u = q^{p-1}\big(\Delta\big(\vrho_\delta u \big)  -  \divv\big( \nabla \vrho_\delta  u    \big) \big),
\end{equation}
 obtained by fixing $w_\delta$ in the nonlinear variable of the coefficients in \eqref{eq:Leibenson-FPE}. We set $Q_T:= (0,T)\times \R^d$ for $T>0$.

	\begin{dfn}\label{d5.7}
	A {\it (distributional)} \textit{solution to \eqref{eq:Leibenson-FPE-lin} on $(0,T)$ with initial condition $\nu \in \mathcal{M}^+_b$} is a function $u \in L^1(Q_T)$ such that $t\mapsto u(t,x)\mathrm{d}x$ is a weakly continuous curve of (signed) Borel   measures with
	\begin{equation*}
		\int_0^T \int_{\R^d} (\vrho_\delta + |\nabla \vrho_\delta|)|u|\,\mathrm{d}x \mathrm{d}t < \infty,
	\end{equation*}
	such that for all $\psi \in C^\infty_c(\R^d)$
	\begin{equation}\label{e5.5}
		\int_{\R^d}\psi\,u(t)\,\mathrm{d}x - \int_{\R^d} \psi \,\mathrm{d}\nu = \int_0^t \int_{\R^d} (\vrho_\delta \Delta \psi + \nabla \vrho_\delta \cdot \nabla \psi ) \, u\, \mathrm{d}x \mathrm{d}t,\quad \forall0<  t < T.
	\end{equation}
	$u$ is called {\it probability solution}, if $u\ge0$ and $\nu,u(t,x)\mathrm{d}x\in\mathcal{P}$ for all $t\in(0,T)$.
\end{dfn}
This definition is consistent with Definition \ref{dD.1} for the linear Fokker--Planck equation \eqref{eq:Leibenson-FPE-lin} (except for the additional assumption $u \in L^1(Q_T)$, and on $(0,T)$ instead of $\R_+$).
Clearly, $u(t,x) = w_\delta(t,x)$ is a probability solution with initial datum $w(\delta,x)\mathrm{d}x$ to \eqref{eq:Leibenson-FPE-lin}.
	\begin{rem}\label{r5.8}
	An equivalent condition to \eqref{e5.5} is
	\begin{equation}\label{e*}
		\int_{(0,T)\times \R^d} \bigg( \partial_t \varphi + q^{p-1}\divv\big(\vrho_\delta \nabla \varphi \big)\bigg)u\, \mathrm{d}x \mathrm{d}t + \int_{\R^d} \varphi(0)\,\mathrm{d}\nu = 0
	\end{equation}	
	for all $\varphi \in C^{\infty}_c([0,T)\times \R^d)$.
\end{rem}

	We collect basic properties of $w_\delta$ and $\vrho_\delta$ used in the sequel, of which some were already stated in Lemma \ref{lem:prop-Barenblatt}.
\begin{lem}\label{l5.9}
	Let $\delta >0$. Then $w_\delta$ and $\vrho_\delta$ are nonnegative functions with the following properties.
	\begin{enumerate}
		\item[\rm(i)] $w_\delta \in \bigcap_{R>0}C_c([0,R]\times \R^d)\cap L^\infty([0,\infty)\times \R^d)$.
		\item [\rm(ii)] For $p \geq 2$, we have $\vrho_\delta \in \bigcap_{R>0} C_c([0,R]\times \R^d)$ and $\{x\in \R^d: \vrho_\delta(t,x)>0\}=\{x\in \R^d: 0 \leq |x| < \big(\frac{C}{\kappa}\big)^{\frac{p-1}{p}}(t+\delta)^{\frac 1 \beta}\}$, for all $t\geq 0$, where $\beta = p + d(q(p-1)-1)$ as above.
		\item [\rm(iii)] $\nabla \vrho_\delta \in L^\infty_{\loc}([0,\infty), L^r(\R^d;\R^d))$ for $r \in [1,d(p-1))$ (this interval is nonempty if $p> \frac{d+1}{d}$).
	\end{enumerate}
\end{lem}
\begin{proof}
	The properties of $w_\delta$ follow from Lemma \ref{lem:prop-Barenblatt}. Nonnegativity and regularity of $\vrho_\delta$ is obvious from its definition and it is clear that the support of $\vrho_\delta(t)$ equals the support of $w_\delta(t)$, which is stated in Lemma \ref{lem:prop-Barenblatt}.
Regarding (iii), by taking into account \eqref{1a} and considering $d$-dimensional spherical coordinates, it suffices to note that $\int_0^1 r^{-\frac{r}{p-1}+d-1} \mathrm{d}r <\infty$ holds. Indeed, $-\frac{r}{p-1}+d >0 \iff r < d(p-1).$ 
\end{proof}
Consider the class
	\begin{equation*}
	\Ascr_{\delta,T}:= 	\big\{u \in L^1\cap L^\infty(Q_T): t\mapsto u(t,x)\mathrm{d}x \in C([0,T],\mathcal{P}), \, \exists C\geq 1: u(t) \leq C w_\delta(t) \,\,\mathrm{d}x\text{-a.s.},\, \forall t \in [0,T]\big\},
\end{equation*}
where $C([0,T],\mathcal{P})$ is the set of weakly continuous maps $t\mapsto\mu_t\in\mathcal{P}$. Clearly $w_\delta \in \Ascr_{\delta,T}$. Comparing with Theorem \ref{t52} (more precisely, with Corollary \ref{c5.3} and Lemma \ref{lem:equiv}) and the situation of Theorem \ref{t5.2}, the following result is the remaining piece needed to close the proof of Theorem \ref{t5.2}. For simplicity, we only formulate the result for $y=0$ and $s=0$. The proof for the general case is identical (note that we need this result indeed for all $s\geq 0$ and $y \in \R^d$).

\begin{theorem}\label{restr-lin-u.theorem}
  Let $T \in (0,\infty)$, $\delta>0$, $d \geq 2$, $q>0$ and $q(p-1)>1$. Assume additionally either
  \begin{enumerate}
      \item[(i)] $p>2$ and $q(p-1)<1+d(p-1)^2$
  \end{enumerate}
  or
  \begin{enumerate}
      \item [(ii)] $p>\frac{d}{d-1}$
    and
    $q>\frac{|p-2|+d}{d(p-1)}$.
  \end{enumerate}
Then $w_\delta$ is the unique distributional solution to \eqref{eq:Leibenson-FPE-lin} on $(0,T)$ in the sense of Definition {\rm\ref{d5.7}} with initial condition $w_\delta(0,x)\mathrm{d}x$ in $\mathcal{A}_{\delta,T}$.
\end{theorem}

Notably, as shown below, the proof of the assertion under assumption (i) follows by analytic methods, while for case (ii) the probabilistic tools developed in Section \ref{sect:Strongsolution} are employed. Note that (i) does not imply (ii) and (ii) does not imply (i).

\begin{rem}\label{rem:stronger-u}
   In fact, a stronger uniqueness assertion holds in the situation of the previous theorem. Precisely, for every initial condition $\nu \in \Pscr$ there is at most one distributional solutions to \eqref{eq:Leibenson-FPE-lin} in $\mathcal{A}_{\delta,T}$ with initial condition $\nu$. In the proofs below we prove this stronger assertion.
\end{rem}

\begin{rem}\label{restr-lin-u.remark:qRange}
If $p>2$, then $p> \frac{d}{d-1}$ for $d \geq 2$ and we have
\begin{align*}
    \frac{|p-2|+d}{d}  < 1+d(p-1)^2.
\end{align*}
Hence for $p>2$ the assertion of the previous theorem holds for all $q>0$ such that $q(p-1)>1$.
\end{rem}

\begin{rem}
Even though this restricted uniqueness result concerns a \emph{linear} PDE, we could not obtain its proof by a standard result from the literature. Note that the diffusion coefficient $\vrho_\delta$ is degenerate (it is compactly supported in $\overline{Q_T}$), which renders the assertion rather delicate. For instance, we cannot apply \cite[Thm.9.8.2]{FPKE-book15}, since this requires $\vrho_\delta^{-\frac 1 2} \nabla \vrho_\delta$ to be bounded, which is not true. Our proofs below heavily use the explicit form of $w_\delta$ and $\vrho_\delta$.
\end{rem}

\subsection{Analytic proof of Theorem \ref{restr-lin-u.theorem} 
(i)}\label{subsect:restr-lin-u-analytic}
In this case, the proof is an extension of the proof of Theorem 6.5 in \cite{R.BR24-pLapl}. For the convenience of the reader, we give all details, also for those parts which are very similar to \cite{R.BR24-pLapl}.

Let $T\in(0,\infty),\ \delta>0$, $d\geq 2,\ p>2$, and $q(p-1)>1$, $q < d(p-1) + \frac 1 {p-1}$. For convenience, below we omit the absolute factor $q^{p-1}$ from the RHS of \eqref{eq:Leibenson-FPE-lin}. It is clear that this factor does not influence the proof at all.

\begin{rem}\label{r6.2}  When $u \in (L^1\cap L^\infty)(\R^d)$, then in Remark  \ref{r5.8}, since
	$${\rm div}(\vrho_\delta\nabla  \varphi)=\vrho_\delta\Delta \varphi+\nabla\vrho_\delta
	\cdot\nabla \varphi$$
	in $L^1(Q_T)$,
	by a standard localization argument we can replace $C^\infty_c([0,T)\times\R^d)$ by
	$$C^\infty_{c,b}([0,T)\times\R^d):=\{\varphi\in C^\infty_b([0,T)\times\R^d)\mid\exists\lambda>0\mbox{ such that }\varphi(t,x)=0,\forall(t,x)\in[T-\lambda,T)\times\R^d\}.$$
	
\end{rem}
We give more details on this claim in Appendix \ref{App-A}.
In the forthcoming proof, we abbreviate partial time derivatives $\frac{\mathrm{d}f}{\mathrm{d}t}$ by $f_t$. Constants depending only on absolute constants are denoted $C_i, i = 1,2,...$.
	\begin{proof}[Proof of Theorem \ref{restr-lin-u.theorem} (i)] 
    Since there is no change in the argument, we prove the stronger assertion mentioned in Remark \ref{rem:stronger-u}. Let $u,\bar{u} \in \mathcal{A}_{\delta,T}$ be distributional solutions to \eqref{eq:Leibenson-FPE-lin} with any common initial condition $\nu \in \Pscr$ and set
$$v:= u-\bar{u}.$$
Since $u,\bar{u} \in \mathcal{A}_{\delta,T}$, there are nonnegative measurable maps $g,\bar{g}:Q_T \to \R$ such that $u(t,x)\mathrm{d}x = g(t,x)w_\delta(t,x)\mathrm{d}x$ and $\bar{u}(t,x)\mathrm{d}x = \bar{g}(t,x)w_\delta(t,x)\mathrm{d}x$ for all $t \in [0,T]$, and $g(t,x), \bar{g}(t,x)\leq C$ $\mathrm{d}t\mathrm{d}x$-a.s. for some $C>0$. Hence $v(t,x) = h(t,x)w_\delta(t,x)$ $\mathrm{d}t\mathrm{d}x-$a.s. for $h:= g-\bar{g}$. In particular, $-C\leq h \leq C$.

$v$ is a (signed) distributional solution to \eqref{eq:Leibenson-FPE-lin} in the sense of Definition \ref{d5.7} with initial condition the zero measure. Let $f\in C^\infty_c(Q_T)$, set and  for $\varepsilon\in(0,1)$  consider the equation
\begin{align}\label{e6.4z}
	&(\varphi_\varepsilon)_t+{\rm div}((\vrho_\delta+\varepsilon)\nabla\varphi_\varepsilon)=f \text{ in } Q_T
	\\
	\notag	&\varphi_\varepsilon(T,x)=0,\ \ x\in\R^d.
	\end{align}
	By standard existence theory for linear parabolic equations (see, e.g., \cite[Thm.10.9, p.341]{B10-book}) it follows that equation \eqref{e6.4z} has a unique solution
$$\varphi_\varepsilon\in C([0,T];L^2)\cap L^2(0,T;H^1),$$
with $(\varphi_\varepsilon)_t\in L^2(0,T;H^{-1})$. Moreover, we have
\begin{equation}\label{e6.5z}
	|\varphi_\varepsilon(t)|^2_2+\int^T_0\int_{\R^d}(\vrho_\delta(t,x)+\varepsilon)|\nabla\varphi_\varepsilon(t,x)|^2\mathrm{d}t\mathrm{d}x
	 \leq \int^T_0\int_{\R^d}|f(t,x)|^2\mathrm{d}t\mathrm{d}x.
\end{equation}
In particular, \eqref{e6.4z} is equivalent to
\begin{align}\label{e6.5}
&(\varphi_\varepsilon)_t+(\vrho_\delta+\varepsilon)\Delta\varphi_\varepsilon+\nabla
	\vrho_\delta\cdot\nabla\varphi_\varepsilon=f \text{ on }Q_T,
\\&\notag 
	\varphi_\varepsilon(T,x)=0,\ x\in\R^d.
\end{align}
\begin{claim}\label{claim1a}  We have $\varphi_\varepsilon\in W_2^{1,2}(Q_T)$, that is,
	\begin{equation*}\label{e6.6z}
	(\varphi_\varepsilon)_t, \partial_i \varphi_\varepsilon, \partial_i \partial_j \varphi_\varepsilon \in L^2(Q_T), \quad i,j=1,...,d.\end{equation*}
\end{claim}		

\medskip\noindent{\it Proof of Claim} \ref{claim1a}.
We set $g_\varepsilon=\nabla\vrho_\delta\cdot\nabla
\varphi_\varepsilon.$ By \eqref{e6.5z} we have $\varphi_\varepsilon,|\nabla\varphi_\varepsilon|\in L^2(Q_T)$. By Lemma \ref{l5.9} (iii), we know that $\nabla\vrho_\delta\in L^\infty(0,T;L^r)$, $r\in[1,d(p-1))$. Fix $r\in(d,d(p-1))$ (this interval is nonempty, since $p>2$). Then, by H\"older's inequality,
\begin{equation}\label{eqeqeq}
g_\varepsilon\in L^{\gamma_1}(Q_T),\ \mbox{ for }\gamma_1:=\frac{2r}{2+r}\in(1,2).
\end{equation}  
Taking into account that, by \eqref{e6.5},
$$(\varphi_\varepsilon)_t+(\vrho_\varepsilon+\varepsilon)\Delta\varphi_\varepsilon=f-g_\varepsilon\in L^{\gamma_1}(Q_T),$$
we get by \cite[Theorem 9.1, p.~341]{L68-book} that $\varphi_\varepsilon\in W^{1,2}_{\gamma_1}(Q_T)$,   that is,
$$(\varphi_\varepsilon)_t, \partial_i \varphi_\varepsilon,\partial_i \partial_j \varphi_\varepsilon\in L^{\gamma_1}(Q_T),\ i,j=1,...,d.$$
On the other hand, by the Sobolev--Nirenbergh--Gagliardo theorem (see, e.g., \cite[p.283]{B10-book}) we have (recall $d \geq 2$)
$$|\nabla\varphi_\varepsilon|\in L^{\alpha_1}(Q_T)\mbox{\ \ for }\alpha_1:=\frac{\gamma_1d}{d-\gamma_1},$$
and this yields as above
$$g_\varepsilon\in L^{\gamma_2}(Q_T),\ \mbox{ for }  \gamma_2:=\frac{\alpha_1r}{\alpha_1+r}\in(1,\infty).$$
Then, again by \cite[Theorem 9.1, p.~341]{L68-book}, it follows that $\varphi_\varepsilon\in W^{1,2}_{\gamma_2}(Q_T)$ and, therefore, again by Sobolev--Nirenbergh--Gagliardo theorem,
$$|\nabla\varphi_\varepsilon|\in L^{\alpha_2}(Q_T)\ \mbox{ for }\alpha_2:=\frac{\gamma_2d}{d-\gamma_2}.$$
Continuing, we obtain sequences $\gamma_i$, $\alpha_i$ such that
\begin{equation}
	\label{e6.8z}
	g_\varepsilon\in L^{\gamma_i}(Q_T),\ \varphi_\varepsilon\in W^{1,2}_{\gamma_i}(Q_T),
\end{equation}
and
\begin{align*}
\gamma_1=\frac{2r}{2+r},\ \gamma_{i+1}=\frac{\alpha_ir}{\alpha_i+r},\ i\in\mathbb{N},\vspace{1,5mm}\\ 
\alpha_i=\frac{\gamma_id}{d-\gamma_i},\mbox{ as long as }\gamma_i<d.
\end{align*}
This yields the recursive formula
$$ \gamma_1=\frac{2r}{2+r}\in(1,2),\ \gamma_{i+1}=\gamma_i\,\frac{rd}{\gamma_id+r(d-\gamma_i)},\  i\in\mathbb{N}, \mbox{ as long as }\gamma_i<d.$$ 
This iterative procedure stops after the smallest $i$ such that $\gamma_i \geq d$.
Since $r>d$, we have for all $i$ with $\gamma_{i+1}<d$ that 
\begin{equation}
	\label{e6.9z}
	\gamma_{i+1}=\gamma_i\,\frac{rd}{rd+\gamma_i(d-r)}>\gamma_i.
\end{equation}
Suppose
\begin{equation}
	\label{e6.10z}
	\gamma_i <  2\mbox{ for all }i\in\mathbb{N}.\end{equation}
Then, by \eqref{e6.9z} there exists
$$\gamma:=\lim_{i\to\infty}\gamma_i\leq2$$
and, passing to the limit in \eqref{e6.9z}, we obtain
$$\gamma=\gamma\,\frac{rd}{rd+\gamma(d-r)}>\gamma.$$
This contradiction implies that \eqref{e6.10z} is wrong, so there exists $i\in\mathbb{N}$ such that $\gamma_i\geq 2$ and Claim \ref{claim1a} follows by \cite[Theorem 9.1, p.~341]{L68-book}, because $g_\varepsilon\in L^2(Q_T)$ by \eqref{e6.8z}, \eqref{eqeqeq} and interpolation, since $\gamma_1\in(1,2).$ 
 \hfill$\Box$
 
Next, by the maximum principle
\begin{equation}\label{e6.9}
	\sup_{\varepsilon\in(0,1)}|\varphi_\varepsilon|_{L^\infty(Q_T)}\le C_0|f|_{L^\infty(Q_T)}.\end{equation}
This follows in a standard way from \eqref{e6.5} by multiplying the equation with $(\varphi_\varepsilon-(T-t)|f|_{L^\infty(Q_T)})_+$ and $(\varphi_\varepsilon+(T-t)|f|_{L^\infty(Q_T)})_-$, respectively, and integrating over  $Q_T$.  

Setting $\vrho^\varepsilon_\delta:=\vrho_\delta+\varepsilon$ and multiplying \eqref{e6.5} by $-\varphi_\varepsilon$, integrating over $(t,T)\times\R^d$ and using that $\varphi_\varepsilon(T,\cdot)=0$, we obtain by Gronwall's lemma (see \eqref{e6.5z})
\begin{equation}\label{e6.10}
	|\varphi_\varepsilon(t)|^2_{L^2(\R^d)}+
	\int^T_t\int_{\R^d}\vrho^\varepsilon_\delta|\nabla\varphi_\varepsilon|^2\mathrm{d}x\mathrm{d}s
	\le C_1\int^T_t\int_{\R^d}|f|^2\mathrm{d}x\mathrm{d}s,\ \forall t\in(0,T).\end{equation}		
Define for $\lambda\in(0,1)$
$$\varphi^\lambda_\varepsilon(t,x):=\eta_\lambda(t)\varphi_\varepsilon(t,x),\ (t,x)\in Q_T,$$
where $\eta_\lambda(t)=\eta\left(\frac t\lambda\right)\eta\left(\frac{T-t}\lambda\right)$ and 
$\eta\in C^2([0,\infty))$ is such that
$$\eta(r)=0\mbox{ for }r\in[0,1],\ \eta(r)=1\mbox{ for }r>2.$$

	\begin{claim}\label{claim}    
	We have
	\begin{equation}\label{e6.11}
		\int^T_0{}_{H^1}\left<\vrho^\varepsilon_\delta(t)
		\nabla\varphi_\varepsilon(t),\nabla(\varphi_\varepsilon)_t(t)\right>_{H^{-1}}\mathrm{d}t
		\le-\frac12\int_{Q_T}(\vrho_\delta)_t|\nabla\varphi_\varepsilon(t,x)|^2\mathrm{d}t\mathrm{d}x,
	\end{equation}where ${}_{H^1}\left<\cdot,\cdot\right>_{H^{-1}}$ is the duality pairing on $H^1(\R^d;\R^d)\times H^{-1}(\R^d;\R^d)$.
\end{claim}

\medskip\noindent{\it Proof of Claim} \ref{claim}. 
By Claim 1 and its proof we have $\nabla(\varphi_\varepsilon)_t\in L^2((0,T);H^{-1})$ and $\vrho^\varepsilon_\delta
\nabla\varphi_\varepsilon \in L^2((0,T);H^1)$,
 and so  the LHS  of \eqref{e6.11}  is well defined (indeed, in the proof of Claim 1 we showed $g_\varepsilon = \nabla \vrho_\delta \cdot \nabla \varphi_\varepsilon \in L^2(Q_T)$). Now choose a sequence $\{\varphi^\varepsilon_\nu\}\subset C^1([0,T];H^1)$ such that $\varphi^\varepsilon_\nu(T,\cdot)=0$ and, for $\nu\to0$,
\begin{align}\label{e6.12}
	\nabla\varphi^\varepsilon_\nu\to&\nabla\varphi_\varepsilon\mbox{ strongly in }L^2((0,T);H^1)\\
	\label{e6.12.5}
	\nabla(\varphi^\varepsilon_\nu)_t\to&\nabla(\varphi_\varepsilon)_t\mbox{ strongly in }L^2((0,T);H^{-1})
\end{align}
An example for such a sequence is
$$\varphi^\varepsilon_\nu(t)=(\varphi_\varepsilon*\theta_\nu)(t)-
(\varphi_\varepsilon*\theta_\nu)(T),$$ where $\theta_\nu=\theta_\nu(t)$, $\nu>0$, is a standard mollifier sequence on $\R$. Here, for technical purposes, we define $\varphi_\varepsilon(r)$ by $\varphi_\varepsilon(0)
$ and $\varphi_\varepsilon(T)$ for $r\in(-1,0)$ and $r\in(T,T+1)$, respectively. 		
Then, since $\nabla\varphi^\varepsilon_\nu(T)=0$, we have
\begin{align*}
	\int^T_0{}_{H^{1}}\left<
	\vrho^\varepsilon_\delta(t)\nabla\varphi^\varepsilon_\nu(t),
	\nabla(\varphi^\varepsilon_\nu(t))_t\right>_{H^{-1}}\mathrm{d}t
	=\frac12\int_{Q_T}
	\vrho^\varepsilon_\delta(t,x)
	|\nabla\varphi^\varepsilon_\nu(t,x)|^2_t \mathrm{d}t\mathrm{d}x
	\\
	=-\frac12\int_{\R^d}
	\vrho^\varepsilon_\delta(0,x)
	|\nabla\varphi^\varepsilon_\nu(0,x)|^2 \mathrm{d}x
	-\frac12\int_{Q_T}
	\vrho_\delta(t,x)
	|\nabla\varphi^\varepsilon_\nu(t,x)|^2 \mathrm{d}t\mathrm{d}x\\
	\le-\frac12\int_{Q_T}(\vrho_\delta)_t
	(t,x)
	|\nabla\varphi^\varepsilon_\nu(t,x)|^2\mathrm{d}t\mathrm{d}x.
\end{align*}
Letting $\nu\to0$,   we get by 
\eqref{e6.12}-\eqref{e6.12.5} that \eqref{e6.11} holds, as claimed.\hfill$\Box$\bigskip

Now, by \eqref{e6.5} and \eqref{e6.11}, we have
\begin{align*}
	\int_{Q_T}|(\varphi_\varepsilon)_t(t,x)|^2\mathrm{d}t\mathrm{d}x
	&=\int^T_0{}_{H^1}\!\!\left<	\vrho^\varepsilon_\delta(t)\nabla\varphi_\varepsilon(t),
	\nabla(\varphi_\varepsilon)_t(t)\right>_{H^{-1}}\mathrm{d}x
	+\int_{Q_T}f(t,x)(\varphi_\varepsilon)_t(t,x)\mathrm{d}t\mathrm{d}x\\
	&\le-\frac12\int_{Q_T}(\vrho_\delta)_t(t,x)
	|\nabla\varphi_\varepsilon(t,x)|^2\mathrm{d}t\mathrm{d}x
	+\int_{Q_T}f(t,x)(\varphi_\varepsilon)_t(t,x)\mathrm{d}t\mathrm{d}x.
\end{align*}
Since $(\vrho_\delta)_t\in L^\infty(Q_T)$, this implies by Young' inequality
$$\int_{Q_T}|(\varphi_\varepsilon)_t|^2\mathrm{d}t\mathrm{d}x\le-\int_{Q_T}(\vrho_\delta)_t|\nabla\varphi_\varepsilon|^2\mathrm{d}t\mathrm{d}x+\int^T_0|f(t)|^2_{L^2(\R^d)}\,\mathrm{d}t.$$
From the definition of $\vrho_\delta$ we infer
\begin{equation}\label{e6.13}	-(\vrho_\delta)_t
	\le C_2\delta^{-1}\vrho_\delta,\end{equation}		
and hence we obtain	
\begin{equation}\label{e6.14}
	\int_{Q_T}|(\varphi_\varepsilon)_t|^2\mathrm{d}t\mathrm{d}x\le C \int_{Q_T}\vrho_\delta|\nabla\varphi_\varepsilon|^2\mathrm{d}t\mathrm{d}x+
	\int^T_0|f(t)|^2_{L^2(\R^d)}\,\mathrm{d}t,\end{equation}	
where $C>0$ only depends on $p,q,d$ and $\delta$.
	Then, by \eqref{e6.5} we have
\begin{equation}\label{e6.15}
	\begin{array}{l}
	 (\varphi^\lambda_\varepsilon)_t+\divv((\vrho_\delta+\varepsilon)\nabla\varphi^\lambda_\varepsilon)=f\eta_\lambda+\eta'_\lambda\varphi_\varepsilon\mbox{ on }Q_T \vspace*{1mm}\\
		\varphi^\lambda_\varepsilon(T,x)=0,\ \forall x\in\R^d.\end{array}
\end{equation}
From what we derived above, we have $\varphi^\lambda_\varepsilon\in W^{2,1}_2(Q_T)$ and  
\begin{eqnarray}
	&\varphi^\lambda_\varepsilon\in L^2((0,T);H^2)\cap C([0,T];L^2),\ (\varphi^\lambda_\varepsilon)_t \in L^2((0,T);L^2),\ \varphi^\lambda_\varepsilon\in L^\infty(Q_T).\label{e6.16}	\end{eqnarray}
By Remark \ref{r6.2}, $v$ satisfies \eqref{e*} for all $\varphi\in C^\infty_{c,b}([0,T)\times\R^d)$. Hence, we infer by density that $v$ satisfies \eqref{e*} also for all  functions $\varphi=\varphi^\lambda_\varepsilon$ with properties \eqref{e6.16}.
Therefore,  we have
\begin{equation}\label{e6.17} \int_{Q_T}v\left((\varphi^\lambda_\varepsilon)_t+\divv(\vrho_\delta\nabla\varphi^\lambda_\varepsilon)\right)\mathrm{d}t\mathrm{d}x=0,\ \forall\varepsilon,\lambda\in(0,1) 
\end{equation}
 (see Appendix \ref{App-B} for details).

	Next, we get by \eqref{e6.15} the following equality
$$\frac12\,|\varphi^\lambda_\varepsilon(t)|^2_{L^2(\R^d)}+\varepsilon\int^T_t|\nabla\varphi^\lambda_\varepsilon(s)|^2_{L^2(\R^d)}\mathrm{d}s+\int^T_t\int_{\R^d}\vrho_\delta|\nabla\varphi^\lambda_\varepsilon|^2\mathrm{d}s\mathrm{d}x
=-\int^T_t\int_{\R^d}(f\eta_\lambda+\eta'_\lambda\varphi_\varepsilon)\varphi^\lambda_\varepsilon \mathrm{d}t\mathrm{d}x.$$
Taking into account \eqref{e6.9}, we get
\begin{equation}\label{e6.18}
	|\varphi^\lambda_\varepsilon(t)|^2_{L^2(\R^d)}+\varepsilon\int^T_t
	|\nabla\varphi^\lambda_\varepsilon(s)|^2_{L^2(\R^d)}\mathrm{d}s
	+\int_{Q_T}\vrho_\delta|\nabla\varphi^\lambda_\varepsilon|^2\mathrm{d}t\mathrm{d}x\le C_3,\ \forall\varepsilon,\lambda\in(0,1).
\end{equation}
Now, we have as in \eqref{e6.11} that
\begin{equation*}
	\int^T_0{}_{H^1}\left<\vrho^\varepsilon_\delta(t)
	\nabla\varphi^\lambda_\varepsilon(t),
	\nabla(\varphi^\lambda_\varepsilon(t))_t\right>_{H^{-1}}\,\mathrm{d}t
	\leq-\frac12\int_{Q_T}
	(\vrho_\delta)_t(t,x)
	|\nabla\varphi^\lambda_\varepsilon(t,x)|^2\mathrm{d}t\mathrm{d}x, 
\end{equation*}
and this yields
\begin{equation*}
	\int_{Q_T}|(\varphi^\lambda_\varepsilon)_t|^2\mathrm{d}t\mathrm{d}x
	\le-\int_{Q_T}(\vrho_\delta)_t
	|\nabla \varphi^\lambda_\varepsilon|^2 \mathrm{d}t\mathrm{d}x
	+C_4\int_{Q_T}(|f\eta_\lambda|^2+
	|\eta'_\lambda\varphi_\varepsilon|^2) \mathrm{d}t\mathrm{d}x,\ \forall\varepsilon\in(0,1). 
\end{equation*}
Then, by \eqref{e6.10}, \eqref{e6.13}  and \eqref{e6.18}  we get
\begin{align}\label{e6.19}
	\int_{Q_T}|(\varphi^\lambda_\varepsilon)_t|^2\mathrm{d}t\mathrm{d}x
	\le\int_{Q_T}\vrho_\delta|\nabla\varphi^\lambda_\varepsilon|^2\mathrm{d}t\mathrm{d}x+
	C_4\int_{Q_T}
	(|f|^2+\lambda^{-1}|\varphi_\varepsilon|^2)\mathrm{d}t\mathrm{d}x\le \left(1+\frac1\lambda\right)C_5,\forall\varepsilon,\lambda\in(0,1).
\end{align}
(Here, we have denoted by $C_i$, $i=0,1,2,...$, several positive constants independent of $\varepsilon,\lambda$.) Now, we fix a sequence $\lambda_n\to0$. 
Hence, by \eqref{e6.13}, \eqref{e6.14}, \eqref{e6.18} and \eqref{e6.19} and a diagonal argument, we  find a subsequence $\varepsilon_k\to0$ such that, for all $n\in\mathbb{N}$ as $k\to\infty$, 
\begin{equation}
	\label{e6.20}
	\left\{
	\begin{array}{rcll}
		\varphi_{\varepsilon_k}\ \to\ \varphi,\ \varphi^{\lambda_n}_{\varepsilon_k}&\to&\varphi^{\lambda_n}&\mbox{weakly in }L^2(Q_T)\medskip\\ 
		(\varphi_{\varepsilon_k})_t\ \to\ (\varphi)_t,\ (\varphi^{\lambda_n}_{\varepsilon_k})_t&\to&(\varphi^{\lambda_n})_t&\mbox{weakly in }L^2(Q_T)\medskip\\ 
		\divv((\vrho_\delta+\varepsilon_k)\nabla\varphi^{\lambda_n}_{\varepsilon_k})&\to&\zeta_{\lambda_n}&\mbox{weakly in }L^2(Q_T).
	\end{array}\right.
\end{equation}
To finish the proof, we need the following three claims.

\begin{claim}\label{claim1} $t\mapsto u(t)$ and $t\mapsto \bar{u}(t)$ are weakly continuous from $[0,T]$ to $L^2(\R^d)$.\end{claim}

\medskip\noindent{\it Proof of Claim {\rm\ref{claim1}}.} Since $\sup\limits_{t\in[0,T]}\{|u(t)|_{L^1(\R^d)}, |u(t)|_{L^\infty(\R^d)}\}<\infty$, also $\sup\limits_{t\in[0,T]}\{|u(t)|_{L^2(\R^d)}\}<\infty.$ Hence, for every sequence $\{t_n\}_{n\in\mathbb{N}}\subseteq[0,T]$ with limit $t\in[0,T]$ there is a subsequence $\{t_{n_k}\}_{k\in\mathbb{N}}$ such that $u(t_{n_k})$ has a weak limit in $L^2(\R^d)$.   Due to the weak continuity (in the sense of measures) of $t\mapsto u(t,x)\mathrm{d}x$, this limit is $u(t)$, which implies the $L^2(\R^d)$-weak continuity of $[0,T]\ni t\mapsto u(t).$ The same argument applies to $\bar{u}$.\hfill$\Box$

	\begin{claim}\label{claim2} {\it We have}
	\begin{equation*}\label{e6.21}
		\lim_{n\to\infty}\ \lim_{k\to\infty}
		\int_{Q_T} v\eta'_{\lambda_n}\varphi_{\varepsilon_k}\,\mathrm{d}t\mathrm{d}x=0.
	\end{equation*}
\end{claim}

\noindent{\it Proof of Claim {\rm\ref{claim2}}.} We first note that by \eqref{e6.20} for every $n\in\mathbb{N}$
$$\lim_{k\to\infty}\int_{Q_T}  v\eta'_{\lambda_n}\varphi_{\varepsilon_k}\,\mathrm{d}t\mathrm{d}x
=\int_{Q_T}  v\eta'_{\lambda_n}\varphi\,\mathrm{d}t\mathrm{d}x,$$
and that $\varphi\in C([0,T];L^2)$. Furthermore,  by \eqref{e6.10}, $\varphi(T,\cdot)=0$. Then, for every $\lambda\in(0,1)$,
$$\begin{array}{lcl}
	\int_{Q_T}  v\eta'_\lambda\varphi\,\mathrm{d}t\mathrm{d}x
	&=&\frac1\lambda\int^{2\lambda}_\lambda\eta'
	\left(\frac t\lambda\right)
	\int_{\R^d}  v(t,x)\varphi(t,x)\mathrm{d}t\mathrm{d}x
	+\frac1\lambda\int^{T-\lambda}_{T-2\lambda}\eta'\left(\frac{T-t}\lambda\right)\int_{\R^d}  v(t,x)\varphi(t,x)\mathrm{d}x\smallskip\\
	&=&\int^2_1\eta'(\tau)d\tau\int_{\R^d}  v(\lambda\tau,x)\varphi(\lambda\tau,x)\mathrm{d}x
	+\int^2_1\eta'(\tau)d\tau\int_{\R^d}   v(T-\lambda\tau,x)\varphi(T-\lambda\tau,x)\mathrm{d}x,
\end{array}$$where, as $\lambda\to0$,  both terms converge to zero by Claim \ref{claim1}  and Lebesgue's dominated convergence theorem.\hfill$\Box$ 

	\begin{claim}\label{claim3} There is $\alpha\in(0,1)$ such that for every $\lambda\in(0,1)$ there exists $C_\lambda\in(0,\infty)$ such that
	\begin{equation}\label{e6.22}
		\int_{Q_T}\varepsilon|\Delta\varphi^\lambda_\varepsilon|^{1+\alpha}w_\delta\,\mathrm{d}t\mathrm{d}x\le C_\lambda,\ \forall \varepsilon\in(0,1).\end{equation}
	Hence, $\{\varepsilon^{\frac1{1+\alpha}}\,\Delta\varphi^\lambda_\varepsilon\mid\varepsilon\in(0,1)\}$ is equi-integrable, hence weakly relatively compact in $L^{1}(Q_T;w_\delta\,\mathrm{d}t\mathrm{d}x)$. 
	Therefore, selecting another subsequence  if necessary, for every $n\in\mathbb{N}$ as $k\to\infty$, we find
	\begin{equation}\label{e6.23}
		\varepsilon_k\Delta\varphi^{\lambda_n}_{\varepsilon_k}\to0\mbox{ weakly both in }L^1(Q_T;w_\delta\,\mathrm{d}t\mathrm{d}x)\mbox{ and }L^1(Q_T;h w_\delta\,\mathrm{d}t\mathrm{d}x),
	\end{equation}
where $h$ is as in the beginning of the proof.
\end{claim}

	\noindent{\it Proof of Claim} \ref{claim3}. By the de la   Vall\'ee Poussin theorem and a diagonal argument, \eqref{e6.23} follows from \eqref{e6.22},   so  we only have to prove \eqref{e6.22}. To this purpose, fix
	\begin{equation}\label{e6.24}
		s\in \bigg (\frac{2d(p-1)}{p},d(p-1)\bigg)
	\end{equation}
such that 
\begin{equation}
s > \frac 1 \gamma = \frac{q(p-1)-1}{p-1}
\end{equation}
and
\begin{equation}\label{eeq}
	s> \frac{2d(p-1)}{d(p-1)-(p-2)}
\end{equation}
This is possible since $q < d(p-1) + \frac 1 {p-1}$ implies $\frac 1 \gamma < d(p-1)$ and since $d\geq 2, p>2$, respectively.
	We also note that the interval in \eqref{e6.24} is not empty and that its left boundary point is strictly bigger than $2$, since $d\ge2$, $p>2$. Then,
	$$\frac{p-2}{p-1}\ \frac s{s-2}<d,$$and so we may choose $\alpha\in(0,\frac13)$ such that
	\begin{equation}\label{e6.25}
		0<\frac{p-2}{p-1}\ \frac{(3\alpha+1)s}{s(1-\alpha)-2(\alpha+1)}<d\end{equation}
	and also such that
	\begin{equation}\label{e6.26}
		\alpha<     \frac{d(p-1)-p}{d(p-1)+p} \quad  \bigg( <    \frac{d(p-1)}{(d+2)(p-2)+d}\bigg)
	\end{equation}
as well as
\begin{equation}\label{ghz}
	\alpha < \gamma,
\end{equation}
where $\gamma>0$ is defined in Definition \ref{def:Barenblatt-sol},
and $\alpha$ sufficiently small such that
\begin{equation}\label{eqtg}
	2s \gamma - (3\alpha + 1)s > 2(\alpha +1) - s(1-\alpha).
\end{equation}
The final inequality is true for $\alpha>0$ sufficiently small due to $s>\frac 1 \gamma$.
	Then, we multiply \eqref{e6.15}  by $sign\ \Delta\varphi^\lambda_\varepsilon|\Delta\varphi^\lambda_\varepsilon|^\alpha w_\delta$ and integrate over  $Q_T$ to obtain after rearranging
	\begin{align}\label{e6.27}
	\notag	\int_{Q_T}(\vrho_\delta+\varepsilon)|\Delta\varphi^\lambda_\varepsilon|^{1+\alpha}w_\delta\,\mathrm{d}t\mathrm{d}x
		\le
		\int_{Q_T}|\nabla\vrho_\delta|\,
		|\nabla\varphi^\lambda_\varepsilon|\,|\Delta\varphi^\lambda_\varepsilon|^\alpha\,w_\delta\,\mathrm{d}t\mathrm{d}x\vspace*{1,5mm}\\
		\qquad+
		\int_{Q_T}|(\varphi^\lambda_\varepsilon)_t|\,
		|\Delta\varphi^\lambda_\varepsilon|^\alpha\,w_\delta\,\mathrm{d}t\mathrm{d}x
		+\int_{Q_T}|f\eta_\lambda+\eta'_\lambda\varphi_\varepsilon|\,
		|\Delta\varphi^\lambda_\varepsilon|^\alpha\, w_\delta\,\mathrm{d}t\mathrm{d}x.
	\end{align}
	We now show, using \eqref{e6.16}, that the right hand side of \eqref{e6.27} is finite, hence so is its left hand side. Indeed, by Young inequality, for any $r \in (0,1)$ and a large enough constant $C_r>1$ (both independent of $\varepsilon$ and $\lambda$), the second term on the right hand side  of \eqref{e6.27} can be estimated by
	\begin{equation}\label{e6.28}
		|(\varphi^\lambda_\varepsilon)_t|^2_{L^2(Q_T)}+r\int_{Q_T}|\Delta\varphi^\lambda_\varepsilon|^{1+\alpha}\vrho_\delta\,w_\delta\,\mathrm{d}t\mathrm{d}x
		+C_r\int_{Q_T}w^{\frac2{1-\alpha}}_\delta\,\vrho^{-\frac{2\alpha}{1-\alpha}}_\delta\,\mathrm{d}t\mathrm{d}x.
	\end{equation}
Likewise, we can estimate the last term in \eqref{e6.27} by
	\begin{equation}\label{e6.29}
		|f\eta_\lambda+\eta'_\lambda\varphi_\varepsilon|^2_{L^2(Q_T)}
		+r\int_{Q_T}|\Delta\varphi^\lambda_\varepsilon|^{1+\alpha}\,\vrho_\delta\,w_\delta\,\mathrm{d}t\mathrm{d}x
		+C_r\int_{Q_T}w_\delta^{\frac2{1-\alpha}}\,\vrho^{-\frac{2\alpha}{1-\alpha}}_\delta\,\mathrm{d}t\mathrm{d}x.
	\end{equation}
We note that, for $\alpha\in(0,\frac13)$ satisfying \eqref{e6.26} and \eqref{ghz}, the definition of $\omega_\delta$ and $\vrho_\delta$ implies that the last term
	in \eqref{e6.28} and \eqref{e6.29} are finite. Furthermore, by \eqref{e6.19} the respective first terms in \eqref{e6.28}, \eqref{e6.29} are uniformly bounded in $\varepsilon\in(0,1)$. Finally, for any $r \in (0,1)$ and $C_r>1$ sufficiently large, both independent of $\varepsilon$ and $\lambda$, the first term on the right hand side  of \eqref{e6.27}  can be estimated by
	$$r\int_{Q_T}|\Delta\varphi^\lambda_\varepsilon|^{\alpha+1}\vrho_\delta\,w_\delta\,\mathrm{d}t\mathrm{d}x+
	C_r\int_{Q_T}\frac{|\nabla\vrho_\delta|^{\alpha+1}}{\vrho^\alpha_\delta}|\nabla\varphi^\lambda_\varepsilon|^{\alpha+1}\,w_\delta\,\mathrm{d}x,$$
where the second integral is up to a constant bounded by
	$$\int_{Q_T}|\nabla\vrho_\delta|^{\frac{2(\alpha+1)}{1-\alpha}}\,\vrho^{-\frac{3\alpha+1}{1-\alpha}}_\delta\,w^{\frac2{1-\alpha}}_\delta\,\mathrm{d}t\mathrm{d}x+\int_{Q_T}|\nabla\varphi^\lambda_\varepsilon|^2\vrho_\delta\,\mathrm{d}t\mathrm{d}x$$
	of which the second integral is uniformly bounded in $\lambda,\varepsilon\in(0,1)$ by \eqref{e6.18} and, by \eqref{e6.26} and Young inequality, the first integral is up to a constant bounded by
	$$\int_{Q_T}|\nabla\vrho_\delta|^s\,\mathrm{d}t\mathrm{d}x+\int_{Q_T}w^{\frac{2s}{s(1-\alpha)-2(\alpha+1)}}_\delta\,
	\vrho^{-\frac{(3\alpha+1)s}{s(1-\alpha)-2(\alpha+1)}}_\delta\,\mathrm{d}t\mathrm{d}x,$$
	where by \eqref{e6.24}, \eqref{e6.25} and \eqref{eeq} as well as \eqref{eqtg}   
	this quantity is finite. Choosing $r$ small enough, we hence get from \eqref{e6.27} and the nonnegativity of $\vrho_\delta$ the estimate \eqref{e6.22}, and Claim \ref{claim3} follows.\hfill$\Box$
	
	Finally, we conclude the proof as follows.
We have by Claim \ref{claim2} and \eqref{e6.15}
\begin{align*}
\int_{Q_T}fv\,\mathrm{d}t\mathrm{d}x&=\lim_{n\to\infty}\ \lim_{k\to\infty}
\int_{Q_T}(f\eta_{\lambda_n}+
\eta'_{\lambda_n}\varphi_{\varepsilon_k})v\,\mathrm{d}t\mathrm{d}x
\\
&=\lim_{n\to\infty}\ \lim_{k\to\infty}\int_{Q_T}
(\varphi^{\lambda_n}_{\varepsilon_k})_t
+{\rm div}((\vrho_\delta+\varepsilon_k)\nabla\varphi^{\lambda_n}_{\varepsilon_k}))v\,\mathrm{d}t\mathrm{d}x,
\end{align*}
which, taking into account that $v=hw_\delta$,  by Claim \ref{claim3} is equal to
$$\lim_{n\to\infty}\ \lim_{k\to\infty}\int_{Q_T}
( \varphi^{\lambda_n}_{\varepsilon_k})_t
+{\rm div}(\vrho_\delta\nabla\varphi^{\lambda_n}_{\varepsilon_k}))v\,\mathrm{d}t\mathrm{d}x,$$
which in turn equals zero by \eqref{e6.17}. Hence, $\int_{Q_T}vf\,\mathrm{d}t\mathrm{d}x=0$ and so, because $f\in C^\infty_c(Q_T)$ was arbitrary, the assertion follows.
\end{proof}
\subsection{Probabilistic proof of Theorem \ref{restr-lin-u.theorem} 
(ii)}\label{subsect:restr-lin-u-probabilistic}
We start with the following claim.\\

\noindent\textbf{Claim:} Let $T>0$. Pathwise uniqueness holds for \eqref{StrongSolution.eq:DDSDE-lin} on $[0,T]$ among probabilistically weak solutions $(X,W)$ with $\law{X} \in P^{\Ascr_{\delta,T}}$, where we set
\begin{align*}
    P^{\Ascr_{\delta,T}}:=\{Q \in \Pscr(C([0,T];\rd)) : (Q\circ \pi_t^{-1})_{t\in [0,T]} \in \Ascr_{\delta,T}\}.
\end{align*}
\begin{proof}[Proof of Claim.]
    Let $(X,W), (Y,W)$ be two probabilistically weak solutions to \eqref{StrongSolution.eq:DDSDE-lin} (on $[0,T]$) on the same stochastic basis, with the same Brownian motion $W$, $X(0)=Y(0)$ a.s., and $\law{X}, \law{Y}\in P^{\Ascr_{\delta,T}}$ with $\law{X(t)}(\mathrm{d}x)=u_X(t,x)\mathrm{d}x$ and $\law{Y(t)}(\mathrm{d}x)=u_Y(t,x)\mathrm{d}x$, for all $t\in [0,T]$.
    Imitating the proof of Theorem \ref{StrongSolution.theorem:pathwiseU}, we may similarly estimate for all $n\in\N$ and $\varepsilon>0$
    \begin{align*}
		\E &\left[\ln\left(1+ \frac{|X(t\wedge \tau_n)-Y(t\wedge \tau_n)|^2}{\varepsilon^2}\right)\right] \notag\\
        \lesssim_\delta
        &\int_0^T \int_{B_n(0)} \M_{}|D_x \nabla\vrho_\delta(s,\cdot)|(x)u_X(s,x)\mathrm{d}x\mathrm{d}s
        +
        \int_0^T \int_{B_n(0)} \M_{}|D_x \nabla\vrho_\delta(s,\cdot)|(x)u_Y(s,x)\mathrm{d}x\mathrm{d}s \\
        &+
        \int_0^T \int_\rd \left(\M|\nabla\vrho_\delta^{\frac{1+\gamma}{2}}(s,\cdot)|(x)\right)^2 w_\delta(s,x)^{-1}|x|^{-\frac{p-2}{q(p-1)-1}}u_X(s,x)\mathrm{d}x\mathrm{d}s\\
        &+
        \int_0^T \int_\rd \left(\M|\nabla\vrho_\delta^{\frac{1+\gamma}{2}}(s,\cdot)|(x)\right)^2 w_\delta(s,x)^{-1}|x|^{-\frac{p-2}{q(p-1)-1}}u_Y(s,x)\mathrm{d}x\mathrm{d}s\\
        &\leq
       2\int_0^T \int_{B_n(0)} \M_{}|D_x \nabla\vrho_\delta(s,\cdot)|(x)w_\delta(s,x)\mathrm{d}x\mathrm{d}s
        +
        2\int_0^T \int_\rd \left(\M|\nabla\vrho_\delta^{\frac{1+\gamma}{2}}(s,\cdot)|(x)\right)^2 |x|^{-\frac{p-2}{q(p-1)-1}}\mathrm{d}x\mathrm{d}s.
    \end{align*}
    By the same arguments as in Theorem \ref{StrongSolution.theorem:pathwiseU}, we conclude the finiteness of the right-hand side of the above chain of inequalities, and, finally, that $X\equiv Y$ a.s., by letting $\varepsilon\downarrow 0$. This finishes the proof of the claim.
\end{proof}
\noindent
\textbf{Proof of Theorem \ref{restr-lin-u.theorem} under assumption (ii).}
Let $u, \bar{u}\in \Ascr_{\delta,T}$ be distributional solutions to \eqref{eq:Leibenson-FPE-lin} in the sense of Definition \ref{d5.7}, both with initial condition $w_\delta(0,x)\mathrm{d}x$.
By Theorem \ref{thm:SP-princ-general}, there exist probabilistically weak solutions $(X,W^1), (Y,W^2)$ to \eqref{StrongSolution.eq:DDSDE-lin} such that $\law{X(t)}(\mathrm{d}x)=u(t,x)\mathrm{d}x$ and $\law{Y(t)}(\mathrm{d}x)=\bar{u}(t,x)\mathrm{d}x$, for all $t\in [0,T]$, respectively.
In particular, $\law{X},\law{Y}\in P^{\Ascr_{\delta,T}}$. Since the previous claim implies weak uniqueness among probabilistically weak solutions whose laws are elements in $P^{\Ascr_{\delta,T}}$ (cf. \cite[Proof of Theorem 1.3.1]{Grube-thesis}), we obtain $\law{X}=\law{Y}$. It follows that $u(t,\cdot)=\bar{u}(t,\cdot)\ \mathrm{d}x\text{-a.s.}$, for all $t\in [0,T]$.
This completes the proof.\qed
\appendix
\section{Proof of Remark \ref{r6.2}}\label{App-A}
    Clearly, $C^\infty_c([0,T)\times\R^d)$ is dense in $C^{\infty}_{c,b}([0,T)\times\R^d)$ (defined in Remark \ref{r6.2}) with respect to uniform convergence of all partial derivatives (including zero derivatives) on $[0,T)\times\R^d$. Let $\varphi\in C^\infty_{c,b}([0,T)\times\R^d)$ and $(\varphi_k)_{k\in\mathbb{N}}\subseteq C^\infty_c([0,T)\times\R^d)$, such that $\varphi_k\to\varphi$ as $k\to\infty$  in the above sense. Then, in particular,
	\begin{itemize}
		\item[(i)] $\varphi_k(0)\to\varphi(0)$ uniformly on $\R^d$.
		\item[(ii)] $(\varphi_k)_t\to\varphi_t,\ \nabla\varphi_k\to\nabla\varphi,\ \Delta\varphi_k\to\Delta\varphi$ uniformly on $[0,T)\times\R^d.$
	\end{itemize}Since $u\in L^1(Q_T)\cap L^\infty(Q_T)$, $\vrho_\delta\in L^\infty(Q_T)$, $\nabla\vrho_\delta\in L^1(Q_T)$, (i)--(ii) yield
	$$0=\lim_{k}
	\left(\int_{Q_T}\!\!\!u((\varphi_k)_t+{\rm div}(\vrho_\delta\nabla\varphi_k))\mathrm{d}t\mathrm{d}x+\int_{\R^d}\!\!\!\varphi_k(0) \mathrm{d}\nu\right)
	=\int_{Q_T}\!\!\!u(\varphi_t+{\rm div}(\vrho_\delta\nabla\varphi))\mathrm{d}t\mathrm{d}x+\int_{\R^d}\!\!\!\varphi(0)  \mathrm{d}\nu,
	$$ which proves  Remark \ref{r6.2}.
	
	\section{Details for \texorpdfstring{\eqref{e6.17}}\ }\label{App-B}
We know $\varphi^\lambda_\varepsilon\in W^{2,1}_2(Q_T)$. It is standard that $C^\infty_{c,b}([0,T)\times\R^d)$ (as introduced in Remark \ref{r6.2}) is dense in $W^{2,1}_2(Q_T)\cap\{g\in W^{2,1}_2(Q_T):g(T)=0\}$ with respect to  the usual norm
$$(|g|^{2,1}_2)^2:=|g|^2_{L^2(Q_T)}+|g_t|^2_{L^2(Q_T)}+|\nabla g|^2_{L^2(Q_T)}+|\Delta g|^2_{L^2(Q_T)}.$$
Now, let $(\varphi_k)_{k\in\mathbb{N}}\subseteq C^\infty_{c,b}([0,T)\times\R^d)$ such that  $\varphi_k\stackrel k\to\varphi^\lambda_\varepsilon$ with respect to  this norm. Since $\nu$ is the zero measure and since $v\in L^1\cap L^\infty(Q_T)$, $\vrho_\delta\in L^\infty(Q_T)$, $\nabla\vrho_\delta\in L^2(Q_T),$ (cf. Lemma \ref{l5.9}) and ${\rm div}(\vrho_\delta\nabla\varphi^\lambda_\varepsilon)=\vrho_\delta\Delta\varphi^\lambda_\varepsilon+\nabla\vrho_\delta\cdot\nabla\varphi^\lambda_\varepsilon,$ we deduce
$$0=\lim_k\int_{Q_T}v((\varphi_k)_t+{\rm div}(\vrho_\delta\nabla\varphi_k))\mathrm{d}t\mathrm{d}x=\int_{Q_T}v((\varphi^\lambda_\varepsilon)_t+{\rm div}(\vrho_\delta\nabla\varphi^\lambda_\varepsilon))\mathrm{d}t\mathrm{d}x,$$which proves \eqref{e6.17}.

 \section{Linear and nonlinear Fokker--Planck equations} 
Here we recall the definition of linear and nonlinear Fokker--Planck equations and their standard notion of distributional solution. The linear FPE associated with Borel measurable coefficients $a_{ij},b_i:(0,\infty)\times\R^d\to\R,$ $1\le i,j\le d$, is the second-order parabolic differential equation for measures
\begin{align}\label{eD.1} 
	\partial_t\mu_t=\partial_{ij}(a_{ij}(t,x)\mu_t)-\partial_i(b_i(t,x)\mu_t),\ (t,x)\in(0,\infty)\times{\R^d}.
\end{align}
Usually, an initial condition $\mu_0=\nu\in \Mscr^+_b$ is imposed.

\begin{dfn}\label{dD.1} 
	A {\it (distributional) solution to \eqref{eD.1} with initial condition $\nu \in \mathcal{M}^+_b$}  is a weakly continuous curve $(\mu_t)_{t\ge0}$ of signed locally finite Borel measures on $\R^d$ such that
	\begin{align*}
		\int^T_0\int_{\R^d}|a_{ij}(t,x)|+|b_i(t,x)|\mathrm{d}\mu_t(x)\mathrm{d}t<\infty,\ \forall T>0, 
	\end{align*}	
	and
	$$\int_{\R^d}\psi \mathrm{d}\mu_t(x)-\int_{\R^d}\psi \mathrm{d}\nu=\int^t_0\int_{\R^d} a_{ij}(s,x)\,\partial_{ij} \psi(x)+b_i(s,x)\partial_i\psi(x)\mathrm{d}\mu_s(x)\mathrm{d}s,\ \forall t\ge0,$$
	for all $\psi\in C^2_c({\R^d})$. A solution  is called {\it probability solution}, if $\nu$ and each $\mu_t,$ $t\ge0$, are probability measures. Instead of the initial time $0$, one may consider an initial time $s>0$. It is obvious how to generalize the definition in this regard. In this case, the initial condition is the pair $(s,\nu)$ and the solution is defined on $[s,\infty)$.
\end{dfn} 

For the nonlinear FPE \eqref{e1.3}, the notion of distributional solution is similar.

\begin{dfn}\label{dD.2}
	A {\it (distributional) solution to \eqref{e1.3} with initial condition $\nu$} is a weakly continuous curve $(\mu_t)_{t\ge0}$ of signed locally finite Borel measures on $\R^d$ such that $(t,x)\mapsto a_{ij}(t,x,\mu_t)$ and $(t,x)\mapsto b_i(t,x,\mu_t)$ are product Borel measurable on $(0,\infty)\times \R^d$,
	\begin{align*}
		\int^T_0\int_{\R^d}|a_{ij}(t,x,\mu_t)|+|b_i(t,x,\mu_t)|\mathrm{d}\mu_t(x)\mathrm{d}t<\infty,\ \forall T>0, 
	\end{align*}	
	and
	$$\int_{\R^d}\psi \mathrm{d}\mu_t(x)-\int_{\R^d}\psi \mathrm{d}\nu=\int^t_0\int_{\R^d} a_{ij}(s,x,\mu_s)\,\partial_{ij} \psi(x)+b_i(s,x,\mu_s)\partial_i\psi(x)\mathrm{d}\mu_s(x)\mathrm{d}s,\ \forall t\ge0,$$
	for all $\psi\in C^2_c(\R^d)$. The notion of probability solution and the extension to initial times $s>0$ is as in  the linear case.
\end{dfn}

\section{On the Hardy--Littlewood maximal operators}\label{App-D}
Let $d,n\in \N$, and denote by $\lambda^d$ the $d$-dimensional Lebesgue measure.
\subsection{Definitions and basic properties}
 We recall the definition and properties of (local) Hardy--Littlewood maximal operators on the set of (extended) signed Borel measures on $\rd$, where the latter are denoted as $M_{\mathrm{loc}}(\rd;\R^n)$ in the following. Furthermore, if $\mu \in M_{\mathrm{loc}}(\rd;\R^n)$, then $|\mu|$ denotes its variation measure.

The following definition and lemmatas are essentially taken from \cite{crippa2008estimates}.
\begin{dfn}[{\cite[Definition A.1]{crippa2008estimates}}]\label{appendix.definition.localmaximalfunction}
	Let $\mu \in M_{\mathrm{loc}}(\rd;\R^{n})$ and $R\in (0,\infty]$. We define the (local) maximal function as
	\begin{align*}
		\M_{R}|\mu|(x) := \sup_{0<r<R}\frac{1}{\lambda^d(B_r(0))}\int_{B_r(x)}|\mu|(\mathrm{d}x), x\in\rd,
	\end{align*}
	In the case $R=\infty$ we set $\M:=\M_\infty$.
	Furthermore, if $\mu$ is of the form $\mu(\mathrm{d}x)=f(x)\dx$, where $f\in L^1_{\mathrm{loc}}(\rd;\R^n)$, then we write $\M_{R}|f|:=\M_{R}|\mu|$.
\end{dfn}
\begin{lem}[{\cite[Lemma 3.1]{champagnat2018}}]\label{appendix.lemma.lipschitztypeestimate}
	Let $f\in BV(\rd;\R^{n})$. Then there exists a constant $C_d>0$ depending only on the dimension $d$, and a set $N \in \mathcal{B}(\rd)$ with $\lambda^d(N)=0$ such that for all $x,y \in N^\complement$
	\begin{align}\label{appendix.lemma.lipschitztypeestimate:1}
		|f(x)-f(y)| \leq C_d\left(\M|D f|(x) + \M|D f|(y)\right)|x-y|,
	\end{align}
	where $Df=(\partial_{x_j}f^i)_{1\leq i \leq n, 1\leq j \leq d}$ denotes the matrix of all Schwartz distributional derivatives of $f$'s component-functions in form of finite signed Borel measures on $\rd$.
\end{lem}
\begin{rem}\label{appendix.maximalfunction.remark}
Consider the situation of Lemma \ref{appendix.lemma.lipschitztypeestimate}.
Let $\varphi_\varepsilon=\varepsilon^{-d}\varphi(\cdot \slash \varepsilon), \varepsilon>0,$ be a standard Dirac sequence with respect to some nonnegative $\varphi \in C_c^\infty(\rd)\cap L^1(\rd)$ with $\int \varphi(x)\mathrm{d}x =1$.
Let $x,y \in \rd$. If for all $z \in \{x, y\}$ we have 
\begin{align}
    f(z) = \lim_{\varepsilon\to 0} (f\ast \varphi_\varepsilon)(z),
\end{align}
then \eqref{appendix.lemma.lipschitztypeestimate:1} holds for this pair $x,y$. This follows from the proof of \cite[Lemma 3.1]{champagnat2018}.
In particular, if $f \in C(\rd;\R^n)$, then \eqref{appendix.lemma.lipschitztypeestimate:1} holds for \textit{all} $x,y \in \rd$.

Furthermore, it follows from \cite[Proposition 3.92 (b)]{ambrosio2000BV}, that if $x \in \rd$ is an \textit{approximate jump point of} $f$ (in the sense of \cite[Definition 3.67]{ambrosio2000BV}), then $\M|D f|(x)=+\infty$. Then, with the convention $0\cdot \infty = 0$, \eqref{appendix.lemma.lipschitztypeestimate:1} is trivially satisfied with respect to such $x$, and every $y\in\rd$.
\end{rem}
\begin{lem}[{\cite[Lemma A.2]{crippa2008estimates}}]\label{appendix.lemma.boundednesslocalmaximalfunction}
\begin{enumerate}[(i)]
    \item Let $\mu \in \mathcal{M}_{\mathrm{loc}}(\rd;\R^{n})$ and $R\in (0,\infty]$. Then for $\lambda^d$-a.e. $x\in \rd$, $M_R|\mu|(x)<\infty$.
    \item Let $p\in (1,\infty)$. Then there exists a constant $C_{d,p}>0$ such that for all $f\in L^p_{\mathrm{loc}}(\rd;\R^n)$ and all $\rho >0$
	\begin{align*}
		\int_{B_\rho(0)} (\M_R |f|(x))^p\dx  \leq C_{d,p} \int_{B_{\rho+R}(0)}	|f(x)|^p\dx.
	\end{align*}
	For $p=1$, the previous statement does not hold. However, for $p=1$ one has the following weak estimate: There exists a constant $C_d>0$ such that for all $f\in L^1_{\mathrm{loc}}(\rd;\R^n)$
	\begin{align*}
		\lambda^d\left(\left\{x \in B_{\rho}(0) : \M_R|f|(x) >\alpha\right\}\right) \leq \frac{C_d}{\alpha}\int_{B_{\rho+R}(0)}|f(x)|\dx.
	\end{align*}
\end{enumerate}
\end{lem}
\subsection{Muckenhoupt weights}\label{App-D.muckenhoupt-weights}
\begin{dfn}[{\cite[p. 194]{stein1993harmonic}}]
    Let $w \in L^1_{\mathrm{loc}}(\rd)$. Let $p,p'\in (1,\infty)$ such that $p^{-1} +(p')^{-1} = 1$. If there exists $C>0$ such that for all balls $B\subset \rd$
    \begin{align}
        \frac{1}{\lambda^d(B)}\int_B w(x)\mathrm{d}x \cdot\left[ \frac{1}{\lambda^d(B)}\int_B w(x)^{-\frac{p'}{p}}\mathrm{d}x \right]^\frac{p}{p'}\leq C<\infty,
    \end{align}
    then $w\in A_p$.
\end{dfn}
\begin{theorem}[{\cite[p. 201, Theorem 1]{stein1993harmonic}}]
    Suppose $1<p<\infty$ and $w\in A_p$. Then there exists a constant $C>0$ such that
    \begin{align}
        \int_\rd (\M|f|(x))^p w(x)\mathrm{d}x \leq C \int_\rd |f(x)|^p w(x)\mathrm{d}x.
    \end{align}
\end{theorem}
\subsection{Maximal operator of the surface measure restricted to the boundary of a ball}\label{App-D.maximalFunction-surfaceMeasure}
We consider the case $d=3$. We set $w_3:=\lambda^3(B_1(0))$.
\begin{lem}\label{App-D.lemma.maximalFunction.surfaceMeasure}
    Let $R>0$ and $x\in\R^3\backslash (\partial B_R(0)\cup\{0\})$. Then
        \begin{align}\label{App-D.lemma.maximalFunction.surfaceMeasure:d=3}
            \M |S(\cdot\cap B_R(0))|(x)=
            \Bigg\{\begin{array}{lr}
                \frac{2\pi R}{w_3\sqrt{27}}(|x|||x|-R|)^{-1}, & \text{ if } \sqrt{3}||x|-R| \leq |x|+R; \\
                \frac{4\pi R^2}{w_3}((|x|+R)^3)^{-1}, & \text{ if } |x|+R<\sqrt{3}||x|-R|.
            \end{array}
        \end{align}
\end{lem}
\begin{proof}
	Let us assume for a moment that $B_r(x)\cap \partial B_{R}(0)\neq \emptyset$ and  $B_r(x)\cap \partial B_{R}(0)\neq \partial B_{R}(0)$, which is exactly the case when $
	||x|-R|<r\leq |x|+R$. Then $S(B_r(x)\cap \partial B_{R}(0))$ can be interpreted geometrically as the area of the \textit{spherical cap} $B_r(x)\cap \partial B_{R}(0)$.
	For further calculations we wish to express the height $h$ of the spherical cap as a function of $r$ with parameters $|x|$ and $R$, more precisely as
    \begin{align}
        h\equiv h_{R,|x|}:(||x|-R|, |x|+R]\to [0,2R].
    \end{align}
    We note that the surface measure $S$ is invariant under rotation. Hence, we may assume without loss of generality $x=(|x|,0,0)$.
	Moreover, there are exactly two different vectors $\rho_+, \rho_-$ in the $x_1/x_2$-plane\footnote{The hyperplane spanned by $(1,0,0),(0,1,0)$ in $\R^3$.} of the form $\rho_+ = (a, b,0)$, $\rho_- = (a, -b,0)$ where $a,b \in (0,R)$ are
    determined by the following two equations
	\begin{align*}
		\begin{cases}
			a^2 + b^2 &= R^2,\\
			(a-|x|)^2 +b^2 &= r^2.
		\end{cases}
	\end{align*}

\begin{figure}
\begin{tikzpicture}
\tikzset{
    position label/.style={
       below = 3pt,
       text height = 0.5em,
       text depth = 1em
    },
   brace/.style={
     decoration={brace, mirror},
     decorate
   }
}
\coordinate [label={[xshift=-0.15cm, yshift=-0.4cm] $0$}] (A) at (0,0);
\coordinate [] (B) at (2,0);
\coordinate [label=below:$|x|$] (C) at (3,0);
\draw [->] (-3,0) -- (7,0) node[below right] {$x_1$};
\draw [->] (0,-3) -- (0,3) node[below right] {$x_2$};
\node (D) [draw,circle through=(B)] at (A) {\textbullet};
\node (E) [draw,circle through=(A)] at (C) {\textbullet};
\node (K) at (intersection 1 of D and E) {};
\draw[red, thick] (K) arc[start angle=-70.53, end angle=70.53, radius=2];
\node[fill,circle,scale=0.4,label=below:$\rho_-$] (K) at (intersection 1 of D and E) {};
\node[fill,circle,scale=0.4,label=above:$\rho_+$] (P) at (intersection 2 of D and E) {};
\path let \p1=(K) in coordinate (K_x) at (\x1,0);
\draw[dotted] (K_x) -- (K);
\draw[dotted] (K_x) -- (P);
\path let \p1=(K) in coordinate (K_x+eps) at (\x1,1.5pt);
\path let \p1=(K) in coordinate (K_x-eps) at (\x1,-1.5pt);
\draw[] (K_x) -- (K_x+eps) node[below left] {$a$};
\draw [brace] (K_x) -- (B) node [position label, pos=0.5] {$h$};
\draw[] (K_x) -- (K_x-eps);
\path let \p1=(K) in coordinate (K_y) at (0,\y1);
\path let \p1=(K) in coordinate (K_y+eps) at (1.5pt,\y1);
\path let \p1=(K) in coordinate (K_y-eps) at (-1.5pt,\y1);
\draw[] (K_y) -- (K_y+eps) node[above left] {$-b$};
\draw[] (K_y) -- (K_y-eps);
\path let \p1=(P) in coordinate (P_y) at (0,\y1);
\path let \p1=(P) in coordinate (P_y+eps) at (1.5pt,\y1);
\path let \p1=(P) in coordinate (P_y-eps) at (-1.5pt,\y1);
\draw[] (P_y) -- (P_y+eps) node[below left] {$b$};
\draw[] (P_y) -- (P_y-eps);
\draw[dotted] (K_y) -- (K);
\draw[dotted] (P_y) -- (P);
\pic [draw, "$\theta$", angle eccentricity=1.5,pic text options={shift={(-10pt,-7pt)}}] {angle =K_x--A--P};
\draw[thick] (C) -- (P) node[midway,above right] {$r$};
\draw[thick] (A) -- (P) node[midway,above right] {$R$};
\end{tikzpicture}
    \centering
    \caption{Projection of two intersecting balls onto the $x_1/x_2$ plane. The projection of the spherical cap $\partial B_R(0)\cap B_r(x)$ is highlighted in red.}
    \label{fig:intersectingBalls}
\end{figure}
By subtracting the second equation from the first, one easily deduces that $a=\frac{R^2+|x|^2-r^2}{2|x|}$. Here, we refer to \cite{weissteinWolfram} for the detailed elementary calculation and a similar geometric visualization as in Figure \ref{fig:intersectingBalls}, where the latter is, however, tailored to our setting. Hence, the height of the spherical cap is $h(r)= R-\frac{R^2+|x|^2-r^2}{2|x|}$, $r\in (||x|-R|, |x|+R]$.
(In fact, the previous arguments translate to higher dimensions, and the formula for the height of a general hyperspherical cap is the same.)
Let us extend $h$ (continuously) to a function on $[0,\infty)$ by setting
\begin{align*}
    h(r):= 2R,  \text{ for } r> |x|+R, \text{ and } h(r):= 0 , \text{ for } 0<r\leq ||x|-R|.
\end{align*}
	Using polar coordinates, it is standard to note
	\begin{align}
		S(B_r(x)\cap \partial B_{R}(0))=2\pi R h_{R,|x|}(r)\ \ \forall r\geq 0.
	\end{align}
	We aim to identify the maximum of the function 
    \begin{align*}
        f(r):=2\pi R r^{-3}  h_{R,|x|}(r)
            = 2\pi R r^{-3}\left(R-\frac{|x|^2+R^2-r^2}{2|x|}\right), r\in[||x|-R|, |x|+R].
    \end{align*}
	Note that $f$ is twice differentiable in $(||x|-R|, |x|+R)$ with derivatives
	\begin{align*}
		f'(r)&=\frac{\pi R (3(|x|-R)^2-r^2)}{|x|r^{4}},\\
		f''(r)&=\frac{2\pi R (r^2-6(|x|-R)^2)}{|x|r^{5}}.
	\end{align*}
	So, if $\sqrt{3}||x|-R| \in (||x|-R|, |x|+R)$, $f$ has a global maximum point only at $\sqrt{3}||x|-R|$.
	In this case, we have
	\begin{align*}
		f(\sqrt{3}||x|-R|) = \frac{2\pi R}{\sqrt{27}}(|x|||x|-R|)^{-1}.
	\end{align*}
	Since
    \begin{align*}
        \max\left( f(||x|-R|), f(|x|+R) \right)
        = \frac{4\pi R^2}{(|x|+R)^3},
        \end{align*}
    \eqref{App-D.lemma.maximalFunction.surfaceMeasure:d=3} is evident.
    This finishes the proof.
\end{proof}
\section{Details for the proof of Theorem \ref{StrongSolution.theorem:pathwiseU}}
\label{App-E}
For $1\leq i,j\leq d$, let
\begin{align*}
    b_i,\sigma_{ij}:[0,\infty)\times\rd \to \R
\end{align*}
be Borel measurable functions.
Consider the SDE
\begin{align}\label{App-E.SDE}
\begin{cases}
\mathrm{d}X(t)&=b(t,X(t))\mathrm{d}t + \sigma(t,X(t))\mathrm{d}W(t),\quad t\geq 0,\\
\law{X(t)}&\ll\mathrm{d}x,\ \text{for d}t\text{-a.e. } t\geq 0,
\end{cases}
\end{align}
where we denote the absolute continuity of a Borel probability measure $\mu$ on $\rd$ with respect to $d$-dimensional Lebesgue measure by $\mu<<\mathrm{d}x$.

\begin{dfn}\label{def:SDE}
A \emph{probabilistically weak solution} to \eqref{App-E.SDE} is an adapted stochastic process $X = (X(t))_{t\geq 0}$ on a stochastic basis $(\Omega, \Fscr, (\Fscr_t)_{t\geq 0},\mathbb{P})$ with an $(\Fscr_t)$-standard Brownian motion $W$, such that
\begin{equation}
	\mathbb{E}\bigg[\int_0^T |b(t,X(t))|    +   |\sigma(t,X(t))\sigma^T(t,X(t))| \mathrm{d}t    \bigg]<\infty \quad \forall T>0,
\end{equation}
and $\mathbb{P}$-a.s.
\begin{align*}
	X(t) = X(0) + \int_0^t b(s,X(s)) \mathrm{d}s  + \int_0^t \sigma(s,X(s)) \mathrm{d} W(s) \quad \forall t \geq 0.
\end{align*}
\end{dfn}

\begin{prop}\label{App-E.proposition}
    Let $(X,W)$ be a probabilistically weak solution to \eqref{App-E.SDE}. Then, $(X,W)$ is also a probabilistically weak solution to \eqref{App-E.SDE} where $b,\sigma$ are replaced by Borel measurable functions $\bar{b},\bar{\sigma}$ with components
    \begin{align*}
    \bar{b}_i,\bar{\sigma}_{ij}:[0,\infty)\times\rd \to \R,
\end{align*}
$1\leq i,j\leq d$,
satisfying
\begin{align*}
    b &=\bar{b}\ \  \mathrm{d}t\otimes \mathrm{d}x\text{-a.s.},\\
    \sigma &=\bar{\sigma}\ \ \mathrm{d}t\otimes \mathrm{d}x\text{-a.s.}
\end{align*}
In particular, probabilistically weak solutions to \eqref{App-E.SDE} do not depend on the $\mathrm{d}t\otimes \mathrm{d}x$-version of the coefficients $b,\sigma$.
\end{prop}
\begin{proof}
    For $t\geq 0$, we set
    \begin{align*}
        Y(t):= X(0) + \int_0^t \bar{b}(s,X(s))\mathrm{d}s + \int_0^t \bar{\sigma}(s,X(s))\mathrm{d}W(s).
    \end{align*}
    We have
    \begin{align}
        \E [|X(t)-Y(t)|]
        &\leq \E \left[\left| \int_0^t b(s,X(s))-\bar{b}(s,X(s)) \mathrm{d}s\right|\right]
        + \left(\E \left[\left| \int_0^t \sigma(s,X(s))-\bar{\sigma}(s,X(s))\mathrm{d}W(s) \right|^2\right]\right)^\frac{1}{2} \\
        &\leq \int_0^t \E \left[\left| b(s,X(s))-\bar{b}(s,X(s))\right|\right]\mathrm{d}s
        +  \left(\int_0^t\E\left[\left|  \sigma(s,X(s))-\bar{\sigma}(s,X(s)) \right|^2\right]\mathrm{d}s\right)^\frac{1}{2}=0,
    \end{align}
    where we used Jensen's inequality in the first inequality, and It\^o's isometry and Fubini's theorem in the second inequality. For the last equality, we used the assumption $\law{X(s)}\ll\mathrm{d}x$ for $\mathrm{d}s$-a.e. $s\geq 0$.
    By the (a.s.-)continuity of $X$ and $Y$, we conclude $X\equiv Y$ a.s.
    This finishes the proof.
\end{proof}

\paragraph{Acknowledgements.} The second and fourth named authors are funded by the Deutsche Forschungsgemeinschaft (DFG, German Research Foundation) - Project-ID 317210226 - SFB 1283.
We would like to thank Alexander Grigor'yan for valuable discussions on his work on the Leibenson equation on Riemannian manifolds.

\bibliographystyle{plain}
\bibliography{bib-collection}
\end{document}